\documentclass[a4paper,10pt]{amsart}
\usepackage[utf8]{inputenc}
\usepackage[cmtip,arrow]{xy}
\usepackage{lscape}
\usepackage{latexsym}
\usepackage[activeacute,english]{babel}
\usepackage{graphicx}
\usepackage{amsmath}
\usepackage{amsfonts}
\usepackage{amsthm}
\usepackage{amssymb}
\usepackage{amscd}
\usepackage{mathrsfs}
\usepackage{tikz}
\usepackage{pb-diagram,pb-xy}
\xyoption{all}
\usepackage{graphics}

\newtheorem{theorem}{Theorem}[section]
\newtheorem{lemma}[theorem]{Lemma}
\newtheorem{proposition}[theorem]{Proposition}

\newtheorem{definition}[theorem]{Definition}

\begin{document}
\title{The Auslander-Reiten components seen as Quasi-hereditary Categories}

\author{Mart\'in Ortiz-Morales}
\address{Facultad de Ciencias, Universidad Aut\'onoma del Estado de M\'exico, M\'exico.}
\email{mortizmo@uaemex.mx}
\urladdr{http://www.uaemex.mx/fciencias/}
\thanks{The author thanks  PROMEP/103.5/13/6535  for giving him financial support for the development of this project. }
\date{October 1, 2015}
\subjclass{2000]{Primary 05C38, 15A15; Secondary 05A15, 15A18}}
\keywords{ Quasi-hereditary algebras, Functor categories}
\dedicatory{}
\thanks{This paper is in final form, and no version of it will be submitted
for publication elsewhere.}
\maketitle

\begin{abstract}
Quasi-hereditary  were introduced by L. Scott \cite{Scott, CPS1,CPS2} in order to deal highest weight categories  as they arise in the representation theory of
semi-simple complex Lie algebras and algebraic groups,  and they  have been a very important tool in the study of finite-dimensional algebras.
On the other hand, functor categories were introduced in representation
theory by M. Auslander [A], [AQM] and used in his proof of the first Brauer-Thrall conjecture [A2] and later on used systematically in his joint work with
I. Reiten on stable equivalence [AR], [AR2] and many other applications.
Recently, functor categories were used in [MVS3] to study the Auslander-Reiten components of finite-dimensional algebras.
The aim of the paper is to  introduce  the concept of  quasi-hereditary category, 
 and we can think of the components of the Auslander-Reiten components  as 
 quasi-hereditary categories. In this way, we have applications 
 to the functor category $\mathrm{Mod}(\mathcal{C} )$, with
$\mathcal C$  a component of the Auslander-Reiten quiver.
\end{abstract} 

\section{Introduction and basic concepts}

The notion of quasi-hereditary algebra was introduced by E. Cline, B. Parshall and L. Scott in their 
work on highest weight categories arising in representation theory of Lie algebras and algebraic groups \cite{Scott, CPS1,CPS2}.
Later quasi-hereditary algebras were amply studied by Dlab and Ringel \cite{Dlab,DR, Rin2}.

On the other hand, functor categories were introduced in representation theory by Auslander\cite{Auslander}
 and used in his proof of
the first Brauer–Thrall conjecture \cite{A2} and later  used systematically
in his joint work with I. Reiten on stable equivalence and many other
applications \cite{Dualizing, AR2}. Recently, functor categories were employed by  Martínez-Villa and Solberg to study the Auslander-Reiten 
components of finite-dimensional algebras
\cite{MVS3} and to develop  tilting theory  in arbitrary  functor categories  \cite{MVO1, MVO2}.

   In [MVS3], by using the concepts and the results on Koszul and Artin-Schelter
    regular categories from \cite{MVS1,MVS2},  results analogous to those presented  by Auslander in [A, Theorems 3.12 and 3.13] were shown. 
      There Auslander characterizes when the category of all additive contravariant functors from the category of finitely generated modules over a finite-dimensional algebra $\Lambda$ to abelian groups is left (or right) Noetherian  (as defined in \cite{MB}). This happens if and only
   if the algebra $\Lambda$ is of finite representation type (that is, has only a finite number of non-isomorphic indecomposable finitely generated modules). 
However, instead of considering the  functor category $\mathrm{Mod}(\mathrm{mod}\Lambda )$, they  consider the category of all additive contravariant graded functors from the associated graded category of the category of finitely generated $\Lambda$-modules to graded vector spaces. The associated
graded category of the finitely generated left $\Lambda$-modules with respect to the radical
of the category is a disjoint union of all the associated graded categories  of the additive closures of all the components  $\mathcal{K}$
 in the Auslander–Reiten quiver. Therefore, they reduced  the problem  to consider one component
$\mathcal K$ at the time and the associated graded category $\mathcal{A}_{\mathrm{gr}}(\mathcal K)$ of
that component $\mathcal{K}$.  They showed that for  regular components $\mathcal{K}$  the category of
graded functors from $\mathcal{A}_{\mathrm{gr}}(\mathcal{K})$ to graded vector spaces is left Noetherian if and only if the sections of
$\mathcal{K}$ are infinite Dynkin diagrams $A_\infty, A_\infty^\infty, D_\infty$ or extended  Dynkin diagrams.

In the same  spirit as in the above-mentioned results, in  this paper we  consider  applications to  
functor categories: categories formed by functors defined on a subcategory $\mathcal{C}$ of
the category of modules over a finite-dimensional algebra which take values in the
category of abelian groups. We generalize the theory of quasi-hereditary algebras  from modules to functor
categories, having in mind as one of the main motivations  to apply the developed results to study the Auslander-Reiten componentes
of finite-dimensional algebras seen as quasi-hereditary categories.

This  paper  consists of four sections.
In the first section, we fix the notation and recall some notions from functor
categories that will be used throughout the paper. In the second section, we generalize
some results about quasi-hereditary algebras from modules to functor categories, 
generalizing  the concept of  sequence of standard modules  from modules to sequence of standard subcategories $\Delta$ of $\mathcal{C}$. In addition we give some characterizations of the subcategory $\mathcal{F}(\Delta)$ consisting of the functors which have a   $\Delta$-filtration, starting with a $\mathrm{Hom}$-finite Krull-Schmidt category we obtain  general results and after we add some other conditions like dualizing and Noetherian, in order to obtain similar results for  the case of finite-dimensional algebras.    In the third section, we show that the category $\mathcal{F}(\Delta)$ is functorially finite if $\mathcal{C}$ is   dualizing, Krull-Schmidt,  Noetherian and quasi-hereditary with respect to a finite filtration (see [Theorem 1 , Rin2]).  The fourth section is dedicated to the examples in which  we exhibit some filtrations for the different Auslander-Reiten components to consider them  and as quasi-hereditary categories. We also obtain ad hoc a tilting category for $\mathbb{Z}A_\infty$. Finally,  we show that tensor product of quasi-hereditary categories is again a quasi-hereditary category.  

\subsection{The category $\mathrm{Mod}(\mathcal{C})$}
Throughout this section $\mathcal{C}$ will be an arbitrary skeletally small preadditive
category, and $\mathrm{Mod}(\mathcal{C})$ will denote the category
of contravariant functors from $\mathcal{C}$ to the category of abelian
groups.  According to  Mitchell \cite{MB}, we can think of $\mathcal{C}$
as a ring ''with several objects'' and $\mathrm{Mod}(\mathcal{C})$ as a
category of $\mathcal{C}$-modules. The aim of the paper is to show that the
notions of quasi-hereditary algebras  can be extended to $\mathrm{Mod}(\mathcal{C)}$, 
obtaining as a reult generalizations of some results  on quasi-hereditary algebras. 
In order to obtain generalizations of some  theorems  for quasi-hereditary algebras on  finite-dimensional algebras, we need to add restrictions to our category $\mathcal{C}$, such as Krull-Schmidt, $\mathrm{Hom}$-finite,
dualizing, etc.. To fix the notation, we refer to known results on functors
and categories that we use throughout the paper, making reference  to
the papers by Auslander and Reiten \cite{Auslander},\cite{AQM},\cite{Dualizing}.

  $\mathrm{Mod}(\mathcal{C})$ is then an abelian category with arbitrary sums and products; in
fact it has arbitrary limits and colimits, and the filtered limits are exact
(Ab5 in Grothendieck terminology). Furthermore,   $\mathrm{Mod}(\mathcal{C})$ has enough projective and injective
objects. For any object $C\in \mathcal{C}$, the representable functor $%
\mathcal{C}(\;,C)$ is projective, the arbitrary sums of representable functors
are projective, and any object $M\in \mathrm{Mod}(\mathcal{C})$ is covered by 
an epimorphism $\coprod_{i}\mathcal{C}(\;,C_{i})\rightarrow M\rightarrow 0$.  We say that an object $M$ in $\mathrm{Mod}(C)$ is finitely generated if
there exists an epimorphism $\coprod_{i\in I}\mathcal{C}(\;,C_{i})%
\rightarrow M\rightarrow 0$, with $I$ a finite set.

An object $P$ in $\mathrm{Mod}(\mathcal{C})$ is projective (finitely
generated projective) if and only if $P$ is a summand of $\coprod_{i\in I}
\mathcal{C}(\;,C_{i})$ for a (finite) family  $\{C_{i}\}_{i\in I}$ of objects
in $\mathcal{C}$. The subcategory $\mathfrak{p}(\mathcal{C})$ of $\mathrm{Mod%
}(\mathcal{C})$, of all finitely generated projective objects, is a
skeletally small additive category in which idempotents split; morever, the functor $
P:\mathcal{C}\rightarrow \mathfrak{p}(\mathcal{C})$, $P(C)=\mathcal{C}(\;,C)$
 is fully faithful and induces by restriction $\mathrm{res}:\mathrm{Mod}(%
\mathfrak{p}(\mathcal{C}))\rightarrow \mathrm{Mod}(\mathcal{C})$ an
equivalence of categories. For this reason, we may assume that our
categories are skeletally small additive categories in which idempotents
split; they are then called annuli varieties in \cite{Auslander}.

\begin{definition}
\cite{AQM} Given a preadditive skeletally small category $\mathcal{C}$, we
say $\mathcal{C}$ has pseudokernels, if given a map $f:C_{1}\rightarrow
C_{0} $ there exists a map $g:C_{2}\rightarrow C_{1}$ such that the
sequence $\mathcal{C}(\;,C_{2})\overset{(\;,g)}{\rightarrow }\mathcal{C}%
(\;,C_{1})\overset{(\;,f)}{\rightarrow }\mathcal{C}(\;,C_{0})$ is exact.
\end{definition}

A functor $M$ is finitely presented if there exists an exact sequence $%
\mathcal{C}(\;,C_{1})\rightarrow \mathcal{C}(\;,C_{0})\rightarrow
M\rightarrow 0$. We denote by $\mathrm{mod}(\mathcal{C})$ the full
subcategory of $\mathrm{Mod}(\mathcal{C})$ consisting of finitely presented
functors. It was proved in \cite{AQM} that $\mathrm{mod}(\mathcal{C})$ is abelian if
and only if $\mathcal{C}$ has pseudokernels.

We will  indistinctly say that $M$ is an object of $\mathrm{Mod}(\mathcal{C})$
or that $M$ is a $\mathcal{C}$-module. A representable functor $\mathcal{C}%
(\;,C)$ will  sometimes be denoted by $(\;,C).$

\subsection{ Krull-Schmidt and Dualizing Categories}

We start by giving some definitions from \cite{Dualizing}.

\begin{definition}
Let $R$ be a commutative Artin ring. A $R$-category $\mathcal{C}$, is a preadditive
 category such that $\mathcal{C}(C_{1},C_{2})$ is a $R$-module and, the
composition is $R$-bilinear. Under these conditions $\mathrm{\mathrm{Mod}}(%
\mathcal{C})$ is a $R$-category which we identify with the category of
functors $(\mathcal{C}^{op},\mathrm{Mod}(R))$.

A $R$-category $\mathcal{C}$ is $\mathrm{Hom}$-\textbf{finite} if for each
pair of objects $C_{1},C_{2}$ in $\mathcal{C}$ the $R$-module $\mathcal{C}%
(C_{1},C_{2})$ is finitely generated. We denote by $(\mathcal{C}^{op},%
\mathrm{mod}(R))$ the full subcategory of $(\mathcal{C}^{op},\mathrm{%
\mathrm{Mod}}(R))$ consisting of the $\mathcal{C}$-modules such that for
every $C$ in $\mathcal{C}$ the $R$-module $M(C)$ is finitely generated. The
category $(\mathcal{C}^{op},\mathrm{mod}(R))$ is abelian, and the inclusion $(%
\mathcal{C}^{op},\mathrm{mod}(R))\rightarrow (\mathcal{C}^{op},\mathrm{%
\mathrm{Mod}}(R)) $ is exact.
\end{definition}

The category $\mathrm{mod}(C)$ is a full subcategory of $(\mathcal{C}^{op},%
\mathrm{mod}(R))$. The functors $D:(\mathcal{C}^{op},\mathrm{mod}%
(R))\rightarrow (\mathcal{C},\mathrm{mod}(R))$ and $D:(\mathcal{C},\mathrm{%
mod}(R))\rightarrow (\mathcal{C}^{op},\mathrm{mod}(R))$ are defined as
follows: for any $C$ in $\mathcal{C}$, $D(M)(C)=\mathrm{Hom}%
_{R}(M(C),I(R/r)) $, with $r$ the Jacobson radical of $R,$ and $I(R/r)$ is
the injective envelope of $R/r.$ The functor $D$ defines a duality between $(%
\mathcal{C},\mathrm{mod}(R))$ and $(\mathcal{C}^{op},\mathrm{mod}(R))$. If $%
\mathcal{C}$ is a $\mathrm{Hom}$-finite $R$-category and $M$ is in $\mathrm{%
mod}(\mathcal{C})$, then $M(C)$ is a finitely generated $R$-module and is
therefore in $\mathrm{mod}(R)$.

\begin{definition}
A $\mathrm{Hom}$-finite \  $R$-category\ $\mathcal{C}$ is \textbf{dualizing,}
\ if\ the functor\ $D:(\mathcal{C}^{op},\mathrm{mod}(R))\rightarrow (%
\mathcal{C},\mathrm{mod}(R))$ induces a duality between the categories $%
\mathrm{mod}(\mathcal{C})$ and $\mathrm{mod}(\mathcal{C}^{op}).$
\end{definition}

It is clear from the definition that for dualizing categories $\mathrm{mod}(%
\mathcal{C})$ has enough injectives. Morever, $\mathcal{C}$ has pseudokerneles,
and therefore $\mathrm{mod}(\mathcal{C})$ is an abelian category  by Theorem 2.4 in \cite{Dualizing}.

To finish, we recall the following definition.

\begin{definition}
An additive category $\mathcal{C}$ is \textbf{Krull-Schmidt}, if every
object in $\mathcal{C}$ decomposes in a finite sum of objects whose
endomorphism ring is local.
\end{definition}

Assume $\mathcal{C}$ is  Krull-Schmidt, then for all $C\in\mathcal{C}$ the ring of endomorphisms $\mathrm{End}(C)$ is semiperfect. In other, every
 $\mathrm{End}(C)$-module finitely generated has a projective cover and therefore every finitely presented $\mathcal{C}$-module has a minimal projective
  presentation by Theorem 1 and Theorem 2 in \cite{MVO2}.

\subsection{Change of categories}

According to [A], there exists a unique up to isomorphism functor $-\otimes
_{\mathcal{C}}-:\mathrm{Mod}(\mathcal{C}^{op})\times \mathrm{Mod}(\mathcal{C}%
)\rightarrow \mathbf{Ab,}$ called the \textbf{tensor product}, with the
following properties:

\begin{itemize}
\item[(a)]

\begin{itemize}
\item[(i)] For each $\mathcal{C}$-module $N$, the functor $\otimes _{%
\mathcal{C}}N:\mathrm{Mod}(\mathcal{C}^{op})\rightarrow \mathbf{Ab}$ given
by $(\otimes _{\mathcal{C}}N)(M)=M\otimes _{\mathcal{C}}N$ is right exact.

\item[(ii)] For each $\mathcal{C}^{op}$-module $M$, the functor $M\otimes _{%
\mathcal{C}}:\mathrm{Mod}(\mathcal{C})\rightarrow \mathbf{Ab}$ given by $%
(M\otimes _{\mathcal{C}})(N)=M\otimes _{\mathcal{C}}N$ is right exact.
\end{itemize}

\item[(b)] The functors $M\otimes _{\mathcal{C}}-$ and $-\otimes _{\mathcal{C%
}}N$ preserve arbitrary sums.

\item[(c)] For each object $C$ in $\mathcal{C}$ $M\otimes _{\mathcal{C}%
}(\;,C)=M(C)$ and $(C,\;\;)\otimes _{\mathcal{C}}N=N(C)$.
\end{itemize}

Given a full subcategory $\mathcal{C}^{\prime }$ of $\mathcal{C}$, the
restriction $\mathrm{res}:\mathrm{Mod}(\mathcal{C})\rightarrow \mathrm{Mod}(%
\mathcal{C}^{\prime })$ has a right adjoint, also  called  the tensor product,
and denoted by $\mathcal{C}\otimes _{\mathcal{C}^{\prime }}:\mathrm{Mod}(%
\mathcal{C}^{\prime })\rightarrow \mathrm{Mod}(\mathcal{C}).$ This functor
is defined by $(\mathcal{C}\otimes _{\mathcal{C}^{\prime }}M)(C)=(C,\;\;)|_{%
\mathcal{C}^{\prime }}\otimes _{\mathcal{C}^{\prime }}M$, for any $M$ in $%
\mathrm{Mod}(\mathcal{C}^{\prime })$ and $C$ in $\mathcal{C}$. The following
proposition is proved in [A Prop. 3.1].

\begin{proposition}
\label{cap1.9} Let $\mathcal{C}^{\prime }$ be a full subcategory of $%
\mathcal{C}$. The functor $\mathcal{C}\otimes _{\mathcal{C}^{\prime }}:%
\mathrm{Mod}(\mathcal{C}^{\prime })\rightarrow \mathrm{Mod}(\mathcal{C})$
satisfies the following conditions:

\begin{itemize}
\item[(a)] $\mathcal{C}\otimes _{\mathcal{C}^{\prime }}$ is right exact and
preserves arbitrary sums.

\item[(b)] The composition $\mathrm{Mod}(\mathcal{C}^{\prime })%
\xrightarrow{\mathcal
C\otimes_{\mathcal C'}}\mathrm{Mod}(\mathcal{C})\xrightarrow{\mathrm{res}}%
\mathrm{Mod}(C^{\prime })$ is the identity in $\mathrm{Mod}(\mathcal{C}%
^{\prime }).$

\item[(c)] For each object $C^{\prime }$ in $\mathcal{C}^{\prime }$, $%
\mathcal{C}\otimes _{\mathcal{C}^{\prime }}\mathcal{C}^{\prime
}(\;\;,C^{\prime })=\mathcal{C}(\;\;,C^{\prime })$.

\item[(d)] $\mathcal{C}\otimes _{\mathcal{C}^{\prime }}$ is fully faithful.

\item[(e)] $\mathcal{C}\otimes _{\mathcal{C}^{\prime }}$ preserves
projective objects.
\end{itemize}
\end{proposition}

The functor $M$ in $\mathrm{Mod}(\mathcal{C})$ is called projectively
presented over $\mathcal{C}^{\prime },$ if there exists an exact sequence $%
\coprod_{i\in I}\mathcal{C}(\;\;,C_{i}^{\prime })\rightarrow \coprod_{j\in J}%
\mathcal{C}(\;\;,C_{j}^{\prime })\rightarrow M\rightarrow 0$, with $%
C_{i}^{\prime },C_{j}^{\prime }\in \mathcal{C}^{\prime }$. The category $%
\mathcal{C}\otimes _{\mathcal{C}^{\prime }}\mathrm{Mod}(\mathcal{C}^{\prime
})$ is the subcategory of $\mathrm{Mod}(\mathcal{C})$ with  objects that are the
functors projectively presented over $\mathcal{C}^{\prime }$. The functors $%
\mathrm{res}$ and $\mathcal{C}\otimes _{\mathcal{C}^{\prime }}$ induce an
equivalence between $\mathrm{Mod}(\mathcal{C}^{\prime })$ and $\mathcal{C}%
\otimes _{\mathcal{C}^{\prime }}\mathrm{Mod}(\mathcal{C}^{\prime })$.

\subsection{The radical of a category}

The notion of the Jacobson radical of a category was introduced in [M] and
[A], it is defined in the following way:

\begin{definition}
The Jacobson radical of $\mathcal{C}$, $\mathrm{rad}_{\mathcal{C}}(\;,\;)$
is a subbifunctor of $\mathrm{Hom}_{\mathcal{C}}(\;,\;)$ defined in objects
as $\mathrm{rad}_{\mathcal{C}}(X,Y)=\{f\in \mathrm{Hom}_{\mathcal{C}%
}(X,Y)\mid $for any map $g:Y\rightarrow X$, $1-gf$ is invertible \}.

If $M$ is a $\mathcal{C}$-module, then we denote by $\mathrm{rad}M$ the
intersection of all maximal subfunctors of $M$.
\end{definition}

\begin{proposition}
\label{cap1.21}\emph{[A], [BR], [M]} Let $\mathcal{C}$ be an additive
category and $\mathrm{rad}_{\mathcal{C}}(\;,\;)$ the Jacobson radical of $%
\mathcal{C}$. We then have the following:

\begin{itemize}
\item[(a)] For every object $C$ in $\mathcal{C}$, $\mathrm{rad}_{\mathcal{C}%
}(C,C)$ is just the Jacobson radical of $\mathrm{End}_{\mathcal{C}}(C).$

\item[(b)] If $C$ and $C^{\prime }$ are indecomposable objects in $\mathcal{C%
}$, then  $\mathrm{rad}_{\mathcal{C}}(C,C^{\prime })$ consists of
all non isomorphisms from $C$ to $C^{\prime }$.

\item[(c)] For every object $C$ in $\mathcal{C}$, $\mathrm{rad}_{\mathcal{C}%
}(C,\;\;)=\mathrm{rad}\mathcal{C}(C,\;\;)$ and $\mathrm{rad}_{\mathcal{C}%
}(\;\;,C)=\mathrm{rad}\mathcal{C}(\;\;,C)$.

\item[(d)] For every pair of objects $C$ and $C^{\prime }$ in $\mathcal{C}$,
$\mathrm{rad}_{\mathcal{C}}(C^{\prime },\;\;)(C)=\mathrm{rad}_{\mathcal{C}%
}(\;\;,C)(C^{\prime })$.
\end{itemize}
\end{proposition}

\begin{definition}
By an ideal of the additive category $\mathcal{C}$, we understand a subbifunctor of $\mathrm{Hom}_{\mathcal{C}}(\;,\;).$
\end{definition}

Given two ideals $I_{1}$ and $I_{2}$ of $\mathcal{C}$, we define $%
I_{1}I_{2}$ as follows: $f\in I_{1}I_{2}(C_{1},C_{3}),$ if and only if $f$
is a finite sum of morphisms $C_{1}\xrightarrow{h}C_{2}\xrightarrow{g}C_{3}$%
, with $h\in I_{1}(C_{1},C_{2})$ and $g\in I_{2}(C_{2},C_{3})$.


\section{Quasi-hereditary Categories}

In \cite{DR}  quasi-hereditary rings are  introduced in terms of filtrations  of idempotent ideals; meanwhile in  [Dlab, A.2.6 ] quasi-hereditary algebras are defined in terms of filtrations into ideales determined by traces of certain projective modules.
\bigskip

  In this work we assume that $\mathcal{C}$ is a  $\mathrm{Hom}$-finite $K$-category, with an algebraically closed field $K$.

Here we introduce the concept of quasi-hereditary category by bringing the concepts mentioned above, bearing in mind generalize some classical results on quasi-hereditary algebras. To achieve this, we begin studying certain types of idempotent ideals in a $ K $-category $\mathcal{C}$. We then introduce the concept of heredity ideal (see [DR]), so we can define quasi-hereditary categories  through a  filtration of the  bifunctor $\mathcal { C } (- ,?) $ into  ideals, which corresponds to a filtration of the category $\mathcal C $ into subcategories.
\bigskip

In the main theorem of this section, we give conditions on the filtration of a category in subcategories to make this a quasi-hereditary category. 
\bigskip

\subsection{The ideal $I_{\mathcal{B}}(-,?)$}
We start  this section by studying an important ideal that will be very important in the study of quasi-hereditary categories.
\bigskip 

 Let $\mathcal{B}$ be a full subcategory of $\mathcal{C}$. We denote by  $I_\mathcal{B}(C,C\rq{})$  the subgroup  of $\mathcal{C}(C,C\rq{})$ consisting of morphisms which  factor through  some  object of $\mathcal{B}$. Let 
$E\in\mathcal{B}$ and $k_E=\dim_K\mathcal{C}(E,X)$. If  $\{f_1,\ldots,f_{k_E}\}$ is a $K$-basis for $\mathcal{C}(E,X)$, then we can define
a map  for all $C\in\mathcal{C}$:
\begin{eqnarray*}
 \varphi^E_C&:&\mathcal{C}(C,E)^{k_E}\rightarrow\mathcal{C}(C,X)\\
             & &(g_1,\ldots,g_{k_E})\mapsto\sum_{i=1}^{k_E}f_ig_i  
\end{eqnarray*}

We can readily ascertain  that $\varphi^E:\mathcal{C}(-,E)^{k_E}\rightarrow\mathcal{C}(-,X)$ is a morphism of 
$\mathcal{C}$-modules. Therefore, we have an induced morphism 
$(\varphi^E)_{E\in \mathcal{B}}:\coprod_{E\in \mathcal{B}}\mathcal{C}(-,E)^{k_E}\rightarrow\mathcal{C}(-,X)$.
\bigskip

In the next lemma we study the relationship between the ideal $I_{\mathcal{B}}(-,?)$ and the trace $\mathrm{Tr}_{\{\mathcal{C}(-,E)\}_{E\in\mathcal{B}}}\mathcal{C}(-,X)$ of the family of projective $\mathcal{C}$-modules $\{\mathcal{C}(-,E)_{E\in\mathcal{B}}\}$ in $\mathcal{C}(-,X)$.

\begin{lemma}\label{IdempotentIdeal}
 Let $(\varphi^E)_{E\in \mathcal{B}}:\coprod_{E\in \mathcal{B}}\mathcal{C}(-,E)^{k_E}\rightarrow\mathcal{C}(-,X)$ the morphism defined as above. Then
 $\mathrm{Im}((\varphi^E)_{E\in \mathcal{B}})=\mathrm{Tr}_{\{\mathcal{C}(-,E)\}_{E\in\mathcal{B}}}\mathcal{C}(-,X)=I_{\mathcal{B}}(-,X)$.
\end{lemma}

\begin{proof}
 \textbf{a)} For all $E\in\mathcal{B}$ we have $\mathrm{Im} (\varphi^E)=\sum_{\hat{f}\in\mathcal{C}(E,X)}\mathrm{Im}(\mathcal{C}(-,\hat{f}))$. Indeed
 let $\{f_1,\ldots f_{k_E}\}$ be a $K$-basis for $\mathcal{C}(E,X)$; we then have
 \[
  \mathrm{Im}(\varphi^E_C)=\{\sum_{i=1}^{k_E}f_ig_i|g_i\in\mathcal{C}(C,E)\}
 \]
 On the other hand, we have 
 \begin{eqnarray*}
 \sum_{\hat{f}\in\mathcal{C}(E,X)}\mathrm{Im}(\mathcal{C}(-,\hat{f})_C)
 & =&\sum_{\hat{f}\in\mathcal{C}(E,X)}\mathrm{Im}(\mathcal{C}(C,\hat{f}))\\
& =&\{\sum_{s=1}^m\hat{f}_sg_s|\hat{f_s}\in\mathcal{C}(E,X),g_s\in\mathcal{C}(C,E), m\ge 0\}.
 \end{eqnarray*} 
 
 It is clear that $\mathrm{Im}(\varphi^E)\subset\sum_{\hat{f}\in\mathcal{C}(E,X)}\mathrm{Im}(\mathcal{C}(-,\hat{f}))$.
 To prove the other inclusion, let $\hat{f}_s=\sum_{i=1}^{k_E}c_{is}f_i$, $c_{is}\in K$, a map in $\mathcal{C}(E,S)$. Then, it follows that
 \begin{eqnarray*}
  \sum_{s=1}^m\hat{f}_sg_s=\sum_{s=1}^m(\sum_{i=1}^{k_E}c_{is}f_i)g_s=\sum_{i=1}^{k_E}f_i(\sum_{s=1}^m c_{is}g_s).
 \end{eqnarray*}
Thus,  $\sum_{\hat{f}\in\mathcal{C}(E,X)}\mathrm{Im}(\mathcal{C}(-,\hat{f}))\subset \mathrm{Im}(\varphi ^E)$.

\textbf{b)} For all $E\in\mathcal{B}$, we have $\mathrm{Im}(\varphi^E)=\mathrm{Tr}_{(-,E)}\mathcal{C}(-,X)$.  By Yoneda's lemma it follows that  $$\mathrm{Tr}_{(-,E)}\mathcal{C}(-,X)=
\sum_{\psi\in (\mathcal{C}(-,E),\mathcal{C}(-,X))}\mathrm{Im}(\psi)=
\sum_{\hat{f}\in \mathcal{C}(E,X)}\mathrm{Im}(\mathcal{C}(-,\hat{f}))$$.

\textbf{c)} Finally, observe that $f\in \mathrm{Im}((\varphi^E)_{E\in \mathcal{B}})(C)$ if and only if $f\in I_{\mathcal{B}}(C,X)$.
\end{proof}

\subsection{Heredity ideals and quasi-hereditary categories}

In order to introduce the concept of quasi-hereditary  category we  first  introduce the concept of heredity ideal \cite{DR}.

\begin{definition}
 A two-sided ideal $I$ in $\mathcal{C}$ is called  heredity if the following holds:
 \begin{itemize}
  \item[(i)] $I^2=I$, i.e, $I$ is an idempotent ideal;
  \item [(ii)]$I\mathrm{rad}\mathcal{C}(-,?)I=0$;
   \item[(iii)]$I(-,X)$ is a finitely generated projective $\mathcal{C}$-module for all $X\in\mathcal{C}$.
 \end{itemize}
\end{definition}

Let $\mathcal{B}\subset\mathcal{C}$ a full subcategory; it is clear that $I_{\mathcal{B}}^2=I_{\mathcal{B}}$. Then, we want to know when  the ideal $I_{\mathcal{B}}$ is heredity. The following lemma says  when it holds.

\begin{lemma}\label{IdempotentIdeal2}
Let $\mathcal{B}$ be  a full subcategory of $\mathcal{C}$. Then,  the ideal   $I_{\mathcal{B}}$ is heredity if and only if the following holds:
 \begin{itemize}
  \item [(i)] $\mathrm{rad}\mathcal{C}(E_1,E_2)=0$ for all pairs of non-isomorphic objects $E_1,E_2\in\mathcal{B}$;
   \item[(ii)]$I_{\mathcal{B}}(-,X)$ is a  finite direct sum of projective $\mathcal{C}$-modules of the form $\mathcal{C}(-,E)$, $E\in\mathcal{B}$, for all
   $X\in\mathcal{C}$.
   \end{itemize}
\end{lemma}
\begin{proof}
 
 (i)  Assume that $I_{\mathcal{B}}\mathrm{rad}\mathcal{C}(-,?)I_{\mathcal{B}}=0$, and let $E_1,E_2\in\mathcal{B}$  a pair of non-isomorphic objects. 
 Assume that  $t\in\mathrm{rad}\mathcal{C}(E_1,E_2)$.  We then  have 
 \[
 t=1_{E_2}t1_{E_1}\in I_{\mathcal{B}}(E_1,E_1)\mathrm{rad}\mathcal{C}(E_1,E_2)I_{\mathcal{B}}(E_2,E_2)=0.
 \]
  Thus,  $\mathrm{rad}\mathcal{C}(E_1,E_2)=0$. Conversely, assume that  $\mathrm{rad}\mathcal{C}(E_1,E_2)=0$ for all pairs of non-isomorphic objects $E_1,E_2\in\mathcal{B}$.  Let  $f\in I_{\mathcal{B}}(X,Y)\mathrm{rad}\mathcal{C}(Y,Z)I_{\mathcal{B}}(Z,W)$ which has the form $f=rst$ with $r\in I_{\mathcal{B}}(X,Y)$, $t\in\mathrm{rad}\mathcal{C}(Y,Z)$ and  $s\in I_{\mathcal{B}}(Z,W)$. Therefore, we can write these maps as  $r:X\rightarrow \coprod B_i\rightarrow Y$ and  $s:Z\rightarrow \coprod B_j\rightarrow W$, where  the terms in the middle are finite sums of indecomposable objects in $\mathcal{B}$. It follows that the induced maps $B_i\rightarrow Y\xrightarrow{t}Z\rightarrow B_j$ are all
 in $\mathrm{rad}\mathcal{C}(B_i,B_j)=0$. Finally, $f=0$.

 (ii) By Lemma \ref{IdempotentIdeal}, there exists an epimorphism $\coprod_{i=1}^n\mathcal{C}(-,E_i)\rightarrow I_{\mathcal {B}}(-,X)\rightarrow 0$, $E_i\in\mathcal{B}$. Thus, $ I_{\mathcal {B}}(-,X)$ is a projective finitely generated  $\mathcal{C}$-module if and only if $ I_{\mathcal {B}}(-,X)$  is a finite direct sum of projective $\mathcal{C}$-modules of the form $\mathcal{C}(-,E)$, $E\in\mathcal{B}$.
 
\end{proof}

In order to generalize the definition given in \cite{DR}, we have the following. 

\begin{definition}
Let $K$ be an  algebraically closed field. A $K$-category is called quasi-hereditary if for all $C\in\mathcal{C}$ there exists a
 chain of two-sided ideals
\[
0=I_0\subset I_1\subset I_2\cdots I_{i-1}\subset I_i\subset\cdots\subset  \mathcal{C}(-,?),
\]
which is exhaustive ( i.e.,   $\cup_{i\ge 0}I_i=  \mathcal{C}(-,?)$), and $I_i/I_{i-1}$ is heredity in the category $\mathcal{C}/I_{i-1}$.  Such
 a chain is called a \emph{heredity chain}.
\end{definition}

Assume we have a filtration $\{\mathcal{B}_i\}_{i\ge 0}$  of $\mathcal{C}$ into subcategories
\[
\{0\}=\mathcal{B}_0\subset\mathcal{B}_1\subset\cdots\subset \mathcal{C},
\]
which is exhaustive (i.e., $\cup_{i\ge 0}  \mathcal{B}_i=\mathcal{C}$). Then, we have an exhaustive chain of two sided ideals
\begin{eqnarray}\label{filt1}
0=I_{\mathcal B_0}\subset I_{\mathcal B_1}\subset\cdots\subset  I_{\mathcal{B}_{i-1}}\subset I_{\mathcal{B}_{i}}\subset\cdots\subset  \mathcal{C}(-,?).
\end{eqnarray}

In relation to the above definition,  we want to know when   (\ref{filt1}) is a heredity chain. It is clear that  $ I_{\mathcal{B}_{i}}/ I_{\mathcal{B}_{i-1}}$ is an idempotent ideal in the category $\mathcal{C}/I_{i-1}$, because $ I_{\mathcal{B}_{i}}$ and $ I_{\mathcal{B}_{i-1}}$ are idempotent in $\mathcal{C}$ and 
\[
\left (\frac{I_{\mathcal{B}_{i}}}{ I_{\mathcal{B}_{i-1}}}\right)^2=\frac{I_{\mathcal{B}_{i}}}{ I_{\mathcal{B}_{i-1}}}\frac{I_{\mathcal{B}_{i}}}{ I_{\mathcal{B}_{i-1}}}=\frac{I_{\mathcal{B}_{i}}^2+I_{\mathcal{B}_{i-1}}}{I_{\mathcal{B}_{i-1}}}=\frac{I_{\mathcal{B}_{i}}}{ I_{\mathcal{B}_{i-1}}}
\]

 Therefore,  we have a new definition
 \begin{definition}
 Let $\{\mathcal{B}_i\}_{i\ge 0}$ be an exhaustive  filtration of $\mathcal{C}$ into subcategories. 
 Thus,  $\mathcal{C}$ is said to be quasi-hereditary with respect to $\{\mathcal{B}_{i\ge 0}\}$ if $ I_{\mathcal{B}_{i}}/ I_{\mathcal{B}_{i-1}}$ is heredity in the category $\mathcal{C}/ I_{\mathcal{B}_{i-1}}$.
 \end{definition}

Before announcing the main theorem of this section we need the following.

\begin{lemma}\label{quotient01}
 Consider  a pair of subcategories $\mathcal{B}\subset \mathcal{B}'\subset\mathcal{C}$ which are closed under direct summands. 
 Then $I_{\mathcal{B}'}(-,X)/I_{\mathcal{B}}(-,X)$ is a projective $\mathcal{C}/I_{\mathcal{B}}$-module if and only if  it is isomorphic to
 $\frac{\mathcal{C}(-,E)}{I_{\mathcal{B}}(-,E)}$,   $E\in\mathcal{B}'$.
\end{lemma}
\begin{proof}
First, assume that there exists an isomorphism of $\mathcal{C}/I_{\mathcal{B}}$-modules
\[
\varphi: \frac{\mathcal{C}(-,E)}{I_{\mathcal{B}}(-,E)}\rightarrow\frac{I_{\mathcal{B}'}(-,X)}{I_{\mathcal{B}}(-,X)}.
\]
Proving that  $E\in\mathcal{B}'$ is then sufficient.

Let  $E'$  be an indecomposable summand of $E$ and $E'\xrightarrow{j}E\xrightarrow{p}E'$
be the canonical inclusion and projection respectively. 

Set $\varphi_{E'}(j +I_{\mathcal{B}}(E',E))=(E'\xrightarrow{f}B\xrightarrow{g}X)+I_{\mathcal{B}}(E',X)\in 
\frac{I_{\mathcal{B}'}(E',X)}{I_{\mathcal{B}}(E',X)}$, with  $B\in\mathcal{B}'$. Thus, we have the following
commutative diagram:
\[
 \begin{diagram}
  \node{\frac{\mathcal{C}(B,E)}{I_{\mathcal{B}}(B,E)}}\arrow{e,t}{\varphi_B}\arrow{s,l}{\frac{\mathcal{C}(f,E)}{I_{\mathcal{B}}(f,E)}}
   \node{  \frac{I_{\mathcal{B}'}(B,X)}{I_{\mathcal{B}}(B,X)}}\arrow{s,r}{\frac{I_{\mathcal{B}'}(f,X)}{I_{\mathcal{B}}(f,X)}}\\
    \node{\frac{\mathcal{C}(E',E)}{I_{\mathcal{B}}(E',E)}}\arrow{e,t}{\varphi_{E'}}
     \node{\frac{I_{\mathcal{B}'}(E',X)}{I_{\mathcal{B}	}(E',X)}}
 \end{diagram}
\]

We have $\frac{I_{\mathcal{B}'}(f,E)}{I_{\mathcal{B}}(f,E)}(g+I_{\mathcal{B}}(B,X))=
\varphi_{E'}(j +I_{\mathcal{B}}(E',E))$. Let $h\in \mathcal{C}(B,E)$ such
that $\varphi_{B}(h+I_{\mathcal{B}}(B,E))=g+I_{\mathcal{B}}(B,X)$. Since the diagram is commutative, it follows that $\frac{I_{\mathcal{B}'}(f,E)}{I_{\mathcal{B}}(f,E_j)}(h+I_{\mathcal{B}}
(B,E))=hf+I_{\mathcal{B}}(E',E)=j +I_{\mathcal{B}}(E',E)$. Thus, $1_{E'}=pj=phf$ and, 
$E'$ is a direct summand of $B$, i.e, $E'\in\mathcal{B}'$.
\end{proof}

\begin{theorem}\label{TheoremQh} Let $\{\mathcal{B}_i\}_{i\ge 0}$ be an exahustive  filtration of $\mathcal{C}$ into subcategories which are closed under direct summands.  Then,  $\mathcal{C}$ is quasi-hereditary  with respect to $\{\mathcal{B}_i\}_{i\ge 0}$ if and only if the following holds.
 
 \begin{itemize}
 \item[(i)]  $\mathrm{rad}\mathcal{C}(E,E')=I_{\mathcal{B}_{j-1}}(E,E')$ for all pairs of indecomposable $E,E'\in\mathcal{B}_j$ such that none of them are a direct summand of any object in $\mathcal{B}_{j-1}$.
 \item[(ii)] For all $X\in\mathcal C$ and $j\ge 1$, there exists an exact sequence 
 \[
 \mathcal{C}(-,E_{j-1})\rightarrow \mathcal{C}(-,E_{j})\rightarrow I_{\mathcal{B}_j}(-,X)\rightarrow 0,
 \]
 with $E_j\in\mathcal{B}_j$ and $E_{j-1}\in\mathcal{B}_{j-1}$.
 \end{itemize}
\end{theorem}

\begin{proof}
Given the filtration $\{\mathcal{B}_i\}_{i\ge 0}$, we prove that $I_{\mathcal{B}_i}/I_{\mathcal{B}_{i-1}}$ is heredity in $\mathcal{C}/I_{\mathcal{B}_{i-1}}$ if and only if
(i) and (ii) hold.

(i) \textbf{a)} $\frac{I_{\mathcal{B}_j}}{I_{\mathcal{B}_{j-1}}}\mathrm{rad}(\frac{\mathcal{C}(-,?)}{I_{\mathcal{B}_{j-1}}})\frac{I_{\mathcal{B}_j}}{I_{\mathcal{B}_{j-1}}}=0$
   if and only if   $\frac{I_{\mathcal{B}_j}}{I_{\mathcal{B}_{j-1}}}\left[\frac{\mathrm{rad}\mathcal{C}(-,?)+I_{\mathcal{B}_{j-1}}}{I_{\mathcal{B}_{j-1}}}\right]\frac{I_{\mathcal{B}_j}}{I_{\mathcal{B}_{j-1}}}=0$ if and only if
   $\frac{I_{\mathcal{B}_{j}}(\mathrm{rad}\mathcal{C}(-,?)+I_{\mathcal{B}_{j-1}})I_{\mathcal{B}_{j}}+I_{\mathcal{B}_{j-1}}}{I_{\mathcal{B}_{j-1}}}=0$ 
    if and only if
   \begin{equation}\label{quasih01}
    I_{\mathcal{B}_{j}}\mathrm{rad}\mathcal{C}(-,?)I_{\mathcal{B}_{j}}+I_{\mathcal{B}_{j-1}}I_{\mathcal{B}_{j}}+I_{\mathcal{B}_{j-1}}=
    I_{\mathcal{B}_{j-1}}.
   \end{equation}
    On the other hand, we have $I_{\mathcal{B}_{j-1}}=I_{\mathcal{B}_{j-1}}^2\subset I_{\mathcal{B}_{j-1}}I_{\mathcal{B}_{j}}\subset I_{\mathcal{B}_{j-1}}$, i.e, 
             $I_{\mathcal{B}_{j-1}}=I_{\mathcal{B}_{j-1}}I_{\mathcal{B}_{j}}$. Thus,  equation (\ref{quasih01}) implies
     \begin{equation}
      I_{\mathcal{B}_{j}}\mathrm{rad}\mathcal{C}(-,?)I_{\mathcal{B}_{j}}\subseteq I_{\mathcal{B}_{j-1}}
     \end{equation}

   \textbf{b)} Now we show that $I_{\mathcal{B}_{j}}\mathrm{rad}\mathcal{C}(-,?)I_{\mathcal{B}_{j}}\subseteq I_{\mathcal{B}_{j-1}}$ if and only
   if $\mathrm{rad}\mathcal{C}(E,E')=I_{\mathcal{B}_{j-1}}(E,E')$ for all pairs of indecomposable $E,E'\in\mathcal{B}_j$ such that none of them are a direct summand of any object in $\mathcal{B}_{j-1}$.
   Assume $I_{\mathcal{B}_{j}}\mathrm{rad}\mathcal{C}(-,?)I_{\mathcal{B}_{j}}\subseteq I_{\mathcal{B}_{j-1}}$. Let $t\in\mathrm{rad}\mathcal{C}(E,E')$, then
  $t=1_{E'}t1_{E}\in I_{\mathcal{B}_{j}}(E,E)\mathrm{rad}\mathcal{C}(E,E')I_{\mathcal{B}_{j}}(E',E')
  \subseteq I_{\mathcal{B}_{j-1}}(E,E')$. On the other hand, if $f\in I_{\mathcal{B}_{j-1}}(E,E')$  then  there exists
    $B\in\mathcal{B}_{j-1}$ for which $f=E\xrightarrow{u}B\xrightarrow{v}E'$. Thus $f$  is not an 
   isomorphism.  Otherwise, there exists $g:E'\rightarrow E$ such that $fg=1_{E'}$. Therefore $vug=1_{E'}$, i.e, 
   $E'$ is a direct summand of $B$; hence $E'\in\mathcal{B}_{j-1}$. This contradiction implies that
   $I_{\mathcal{B}_{j-1}}(E,E')\subset\mathrm{rad}\mathcal{C}(E,E')$.

Conversely, assume $\mathrm{rad}\mathcal{C}(E,E')=I_{\mathcal{B}_{j-1}}(E,E')$. Let $f\in I_{\mathcal{B}}(X,Y)\mathrm{rad}\mathcal{C}(Y,Z)I_{\mathcal{B}}(Z,W)$. If
 $f=rst$ with $r\in I_{\mathcal{B}_j}(X,Y)$, $t\in\mathrm{rad}\mathcal{C}(Y,Z)$ and $s\in I_{\mathcal{B}_j}(Z,W)$,
 then we can write $r:X\rightarrow \coprod B_i\rightarrow Y$ and  $s:Z\rightarrow \coprod B_j\rightarrow W$,  where  the terms in the middle are finite 
 sums of indecomposable objects in $\mathcal{B}$. It follows that the induced maps $B_i\rightarrow Y\xrightarrow{t}Z\rightarrow B_j$ are all
 in $\mathrm{rad}\mathcal{C}(B_i,B_j)=I_{\mathcal{B}_{j-1}}(B_i,B_j)$, and therefore $f\in I_{\mathcal{B}_{j-1}}(X,W)$.
 


 (ii)  First we prove the sufficiency by induction.  The case $j=1$ has been proved in Lemma \ref{IdempotentIdeal} because $I_{{\mathcal B}_0}$=0.

 Assume that the statement is true for $j-1$ and  that $\frac{I_{\mathcal{ B}_j}(-,X)}{I_{\mathcal{ B}_{j-1}}(-,X)}$ is a finitely
 projective  $\mathcal{C}/I_{\mathcal{B}_{j-1}}$-
 module. We have then an isomorphism  
 $\frac{\mathcal{C}(-,E_j)}{I_{\mathcal{ B}_{j-1}}(-,E_j)}\cong\frac{I_{\mathcal{ B}_j}(-,X)}{I_{\mathcal{ B}_{j-1}}(-,X)}$ with $E_j\in\mathcal{B}_j$ by Lemma \ref{quotient01}. Therefore, we have a pushout diagram:
\[
\dgARROWLENGTH=1em
 \begin{diagram}
\node{0}\arrow{s,l}{}
    \node{0}\arrow{s,l}{}\\
  \node{I_{\mathcal{B}_{j-1}}(-,E_j)}\arrow{e,t}{}\arrow{s,l}{}
   \node{I_{\mathcal{B}_{j-1}}(-,X)}\arrow{s,l}{}\\
   \node{\mathcal{C}(-,E_j)}\arrow{e,t}{}\arrow{s,l}{}
    \node{I_{\mathcal{B}_{j}}(-,X)}\arrow{s,l}{}\\
  \node{\frac{\mathcal{C}(-,E_j)}{I_{\mathcal{B}_{j-1}}(-,E_j)}} \arrow{e,t}{\cong} \arrow{s,l}{}
   \node{\frac{I_{\mathcal{B}_{j}}(-,X)}{I_{\mathcal{B}_{j-1}}(-,X)}}\arrow{s,l}{}\\
   \node{0}
    \node{0}
 \end{diagram}
\]
which induces  the following exact sequence
\[
 I_{\mathcal{B}_{j-1}}(-,E_j)\rightarrow \mathcal{C}(-,E_j)\coprod I_{\mathcal{B}_{j-1}}(-,X)
 \xrightarrow{\pi} I_{\mathcal{B}_{j}}(-,X)\rightarrow 0
\]

On the other hand, by induction hypothesis, there are exact sequences
\begin{eqnarray}\label{projher}
 \mathcal{C}(-,E_{j-2}')\rightarrow \mathcal{C}(-,E_{j-1}')\xrightarrow{p} I_{\mathcal{B}_{j-1}}(-,X)\rightarrow 0,\\
 \mathcal{C}(-,E_{j-2}'')\rightarrow \mathcal{C}(-,E_{j-1}'')\xrightarrow{q} I_{\mathcal{B}_{j-1}}(-,E_j)\rightarrow 0,\notag
\end{eqnarray}
with $E_{j-1}, E_{j-1}'\in\mathcal{B}_{j-1}$ and $E_{j-2}, E_{j-2}''\in\mathcal{B}_{j-2}$.

 Thus, we have the following commutative diagram:
 \[
 \dgARROWLENGTH=.3em
\dgTEXTARROWLENGTH=.5em
  \begin{diagram}
\node{}
 \node{0}\arrow{s,l}{}
  \node{0}\arrow{s,l}{}
   \node{}
    \node{}
     \node{}\\
\node{}
 \node{Ker(1\coprod p)}\arrow{e,t,=}{}\arrow{s,l}{}
  \node{Ker(1\coprod p)}\arrow{s,l}{}
   \node{}
    \node{}
     \node{}\\
\node{0}\arrow{e,t}{}
 \node{Ker(\pi (1 \small{\coprod} p))}\arrow{e,t}{}\arrow{s,l}{}
  \node{\mathcal{C}(-,E_j)\coprod \mathcal{C}(-,E_{j-1}') }\arrow{e,t}{\pi (1\coprod p)}\arrow{s,l}{1\coprod p}
   \node{I_{\mathcal{B}_{j}}(-,X)}\arrow{e,t}{}\arrow{s,l,=}{}
    \node{0}
     \node{}\\
  \node{0}\arrow{e,t}{}
 \node{Ker(\pi)}\arrow{e,t}{}\arrow{s,l}{}
  \node{\mathcal{C}(-,E_j)\coprod I_{\mathcal{B}_{j-1}}(-,X) }\arrow{e,t}{\pi}\arrow{s,l}{}
   \node{I_{\mathcal{B}_{j}}(-,X)}\arrow{e,t}{}
    \node{0}
    \node{}\\
  \node{}
   \node{0}
    \node{0}
     \node{}
      \node{}
      \node{}
  \end{diagram}
 \]

From  the epimorphisms $I_{\mathcal{B}_{j-1}}(-,E_{j})\rightarrow Ker(\pi)\rightarrow 0$  and
$\mathcal{C}(-,E_{j-2}')\rightarrow Ker(p)=Ker(1\coprod p)\rightarrow 0$, the exact sequences (\ref{projher}) and the horseshoe lemma, we have
an epimorphism $\mathcal{C}(-,E_{j-2}')\coprod \mathcal{C}(-,E_{j-1}'')\rightarrow Ker(\pi (1\coprod p))\rightarrow 0$. In this way, we have 
an exact sequence:
\[
 \mathcal{C}(-,E_{j-2}')\coprod \mathcal{C}(-,E_{j-1}'')\rightarrow \mathcal{C}(-,E_j)\coprod \mathcal{C}(-,E_{j-1}') 
 \xrightarrow{\pi(1\coprod p)}I_{\mathcal{B}_{j}}(-,X)\rightarrow 0
\]
\bigskip

 It only remains to prove necessity. First, observe that $\frac{\mathcal{C}(?,E_{j-1})}{I_{\mathcal{B}_{j-1}}(?,E_{j-1})}\cong 0$ because  $E_{j-1}$ lies in $\mathcal{B}_{j-1}$ and therefore $\mathcal{C}(-,E_{j-1})=I_{\mathcal{B}_{j-1}}(-,E_{j-1})$.  After applying $-\otimes\frac{\mathcal{C}(?,-)}{I_{\mathcal{B}_{j-1}}(?,-)}$ to the
 exact sequence:
  \[
 \mathcal{C}(-,E_{j-1})\rightarrow \mathcal{C}(-,E_{j})\rightarrow I_{\mathcal{B}_j}(-,X)\rightarrow 0
\]
and using the isomorphisms $\mathcal{C}(-,E_{j-1})\otimes\frac{\mathcal{C}(?,-)}{I_{\mathcal{B}_{j-1}}(?,-)}\cong \frac{\mathcal{C}(?,E_{j-1})}{I_{\mathcal{B}_{j-1}}(?,E_{j-1})}\cong 0$ and  $\mathcal{C}(-,E_j)\otimes\frac{\mathcal{C}(?,-)}{I_{\mathcal{B}_{j-1}}(?,-)}\cong \frac{\mathcal{C}(?,E_j)}{I_{\mathcal{B}_{j-1}}(?,E_j)}$, we obtain the isomorphism:
\[
\frac{\mathcal{C}(?,E_j)}{I_{\mathcal{B}_{j-1}}(?,E_j)}\cong  I_{\mathcal{B}_j}(-,X)\otimes\frac{\mathcal{C}(?,-)}{I_{\mathcal{B}_{j-1}}(?,-)}
\]
\bigskip

Finally, we prove that $ I_{\mathcal{B}_j}(-,X)\otimes\frac{\mathcal{C}(?,-)}{I_{\mathcal{B}_{j-1}}(?,-)}\cong 
\frac{I_{\mathcal{B}_j}(?,X)}{I_{\mathcal{B}_{j-1}}(?,X)} $. First, we prove the following isomorphism:
\begin{eqnarray}\label{qher3}
I_{\mathcal{B}_{j-1}}(?,X)\cong  I_{\mathcal{B}_{j}}(-,X)\otimes
 I_{\mathcal{B}_{j-1}}(?,-).
\end{eqnarray}
Indeed after appliying $ I_{\mathcal{B}_{j}}(-,X)\otimes -$ to the exact sequence of $\mathcal C ^{op}$-modules
\[
 \mathcal{C}(E_{j-2}',-)\rightarrow \mathcal{C}(E_{j-1}',-)\rightarrow I_{\mathcal{B}_{j-1}}(?,-) \rightarrow 0
\]
and
using the isomorphisms 
$ I_{\mathcal{B}_{j}}(-,X)\otimes \mathcal{C}(E_{j-2}',-)\cong I_{\mathcal{B}_{j}}(E_{j-2}',X)=\mathcal{C}(E_{j-2}',X)$
and $I_{\mathcal{B}_{j}}(-,X)\otimes\mathcal{C}(E_{j-1}',-)  \cong I_{\mathcal{B}_{j}}(E_{j-1}',X)=\mathcal{C}(E_{j-1}',X)$ (
 because $E_{j-1}', E_{j-2}'\in\mathcal{B}_j$), 
we conclude that   $I_{\mathcal{B}_{j}}(-,X)\otimes I_{\mathcal{B}_{j-1}}(?,-) \cong I_{\mathcal{B}_{j-1}}(?,X)$.
\bigskip

On the other hand, after appliying $ I_{\mathcal{B}_{j}}(-,X)\otimes-$  to the exact sequence
$  0\rightarrow I_{\mathcal{B}_{j-1}}(?,-)\rightarrow \mathcal{C}(?,-)\rightarrow
\frac{\mathcal{C}(?,-)}{I_{\mathcal{B}_{j-1}}(?,-)}\rightarrow 0$
and by using the isomorphism (\ref{qher3}) and $ I_{\mathcal{B}_{j}}(-,X)\otimes \mathcal{C}(?,-)\cong I_{\mathcal{B}_{j}}(?,X)$, we
obtain the following commutative diagram: 
\[
\dgARROWLENGTH=.3em
\dgTEXTARROWLENGTH=.5em
 \begin{diagram}
 \node{}
  \node{I_{\mathcal{B}_{j-1}}(?,X)}\arrow{e,t}{}\arrow{s,l}{=}
   \node{I_{\mathcal{B}_{j}}(?,X)}\arrow{e,t}{}\arrow{s,l}{=}
    \node{I_{\mathcal{B}_{j}}(-,X)\otimes  \frac{\mathcal{C}(?,-)}{I_{\mathcal{B}_{j-1}}(?,-)}} \arrow{e,t}{}\arrow{s,l}{}
     \node{0}\\
     \node{0}\arrow{e,t}{}
  \node{I_{\mathcal{B}_{j-1}}(?,X)}\arrow{e,t}{}
   \node{I_{\mathcal{B}_{j}}(?,X)}\arrow{e,t}{}
    \node{  \frac{I_{\mathcal{B}_{j}}(?,X)}{I_{\mathcal{B}_{j-1}}(?,X)}} \arrow{e,t}{}
     \node{0,}
 \end{diagram}
\]
 and the isomorphism follows.
\end{proof}

\subsection{ The standard and costandard subcategories of $\mathcal{C}$-modules}

      Let $\{\mathcal{B}_j\}_{j\ge 0}$  be a  exhaustive filtration of $\mathcal{C}$ into subcategories which are closed under direct summands: 
\begin{equation}\label{filtration}
   0=\mathcal{B}_0\subset \mathcal{B}_1\subset \cdots\subset \mathcal{B}_{j-1}\subset \mathcal{B}_j\subset\cdots\subset \mathcal{C}
\end{equation}
\bigskip

In the rest of this section, all the filtrations  $\{\mathcal{B}_j\}_{j\ge 0}$ we consider are filtrations  into subcategories that are closed under direct summands.

\bigskip

 We introduce the concept of standard and costandard   subcategories of $\mathcal C$-modules with respect to the given filtration $\{\mathcal{B}_{j\ge 0}\}_{j\ge 0}$.
\begin{definition}\cite{Dlab}
 The sequence
 \[
  \Delta=\Delta_{\mathcal C}=\{\Delta(j):j\ge 1\}
 \]
  of (contravariant) \emph{ standard subcategories} with respect to  a given
   filtration  $\{\mathcal{B}_j\}_{j\ge 0}$  is given by 
  \[
   \Delta(j)=\{\frac{\mathcal{C}(-,E)}{I_{\mathcal{B}_{j-1}}(-,E)}|E\in\mathcal{B}_j\}
  \]

  Similarly, there is a sequence 
  \[
  \Delta^{\circ}=\Delta_{\mathcal C}^{\circ}=\{\Delta^{\circ}(j):j\ge 1\}
 \]
  of covariant standard subcategories  $ \Delta^{\circ}(j)=\{\frac{\mathcal{C}(E,-)}{I_{\mathcal{B}_{j-1}}(E,-)}|E\in\mathcal{B}_j\}$ 
  and the sequence of its
  duals, the (contravariant) costandard subcategories $\nabla=\nabla_{\mathcal{C}}=\{D\Delta^{\circ}(j):j\ge 1\}$ where
  \[
   D\Delta^{\circ}(j)=\{D \frac{\mathcal{C}(E,-)}{I_{\mathcal{B}_{j-1}}(E,-)}|E\in\mathcal{B}_j\}.
     \]
     
  Call the sequence $\Delta$  \emph{Schurian} if  every $\Delta_X(j)$  is Schurian, in other words  $\mathrm{End}(\Delta_X(j))$ is a division
  algebra for all $j\ge 1$, with $X\in\mathcal{B}_j$ indecomposable. 
  \end{definition}
   
  Let us describe some of the basic propierties of the standard and costandard subcategories. In the rest of this subsection we asumme that $\mathcal{C}$ is a 
  quasi-hereditary category with respect to the filtration (\ref{filtration}).

  \begin{lemma}\label{DeltaLemma1}
1) If $\mathrm{Hom}(\Delta(i),\Delta(j))\neq 0$ implies $i\ge j$. 2) $\mathrm{Ext}^1(\Delta(i),\Delta(j))\neq 0$ implies $i>j$. 3)  The sequence  $\Delta$ is Schurian
\end{lemma}

\begin{proof}

Set $\Delta_Y(j)=\mathcal{C}(-,Y)/I_{\mathcal{B}_{j-1}}(-,Y)$, $Y\in\mathcal{B}_j$ and 
$\Delta_X(i)=\mathcal{C}(-,X)/I_{\mathcal{B}_{i-1}}(-,X)$, $ X\in\mathcal{B}_i$. 
After applying $(-,\Delta_E(j))$ to the sequence   $0\rightarrow I_{\mathcal{B}_{i-1}}(-,X)\rightarrow\mathcal{C}(-,X)\rightarrow \Delta_X(i)\rightarrow 0$, we get by the long homology sequence the exact sequence 
\begin{eqnarray}\label{Delta1}
0\rightarrow (\Delta_X(i),\Delta_Y(j))\rightarrow (\mathcal{C}(-,X),\Delta_Y(j))\rightarrow (I_{\mathcal{B}_{i-1}}(-,X),\Delta_Y(j))\\
\rightarrow \mathrm{Ext}^1(\Delta_X(i),\Delta_Y(j))\rightarrow   \mathrm{Ext}^1((\mathcal{C}(-,X),\Delta_Y(j))=0.\notag 
\end{eqnarray}
 \textbf{1)} Let us assume $i < j$, then $i\le j-1$ and    $\mathcal{B}_i\subset \mathcal{B}_{j-1}$.  It is enough to prove that
 $(\mathcal{C}(-,X),\Delta_Y(j))=0$. Since $X$ is an object in $\mathcal{B}_i$,   it follows that $X$ is an object in $\mathcal{B}_{j-1}$
and 
 \[
(\mathcal{C}(-,X),\Delta_Y(j))=\frac{\mathcal{C}(X,Y)}{I_{\mathcal{B}_{j-1}}(X,Y) }=0.
\]
Therefore, (\ref{Delta1}) implies $\mathrm{Hom}(\Delta_X(i),\Delta_Y(j))=0$.  

\textbf{2)}  Let us assume $i \le  j$, then $\mathcal{B}_{i-1}\subset \mathcal{B}_{j-1}$. It is enough to prove that
$(I_{\mathcal{B}_i}(-,X),\Delta_Y(j))=0$. There is an exact sequence 
\begin{eqnarray}\label{Delta2}
(-,E_{i-2})\rightarrow (-,E_{i-1})\rightarrow I_{\mathcal{B}_{i-1}}(-,X)\rightarrow 0, 
\end{eqnarray}
with $E_{i-1}\in\mathcal{B}_{i-1},  E_{i-2}\in\mathcal{B}_{i-2}$.

After applying $(-,\Delta_Y(j))$ to the sequence (\ref{Delta2}),  we get the exact sequence
\[
0\rightarrow (I_{\mathcal{B}_i}(-,X),\Delta_Y(j))\rightarrow ((-,E_{i-1}),\Delta_Y(j))
\rightarrow  ((-,E_{i-2}),\Delta_Y(j))
\] 
Since $E_{i-1}\in\mathcal{B}_{i-1}\subset \mathcal{B}_{j-1}$, we have  $(\mathcal{C}(-,E_{i-1}),\Delta_Y(j))=0$, which
 implies
$(I_{\mathcal{B}_i}(-,X),\Delta_Y(j))=0$.  Thus (\ref{Delta1}) implies $\mathrm{Ext}^1(\Delta_X(i),\Delta_Y(j))=0$.

\textbf{3)} Let $X\in\mathcal{B}_i-\mathcal{B}_{i-1}$ indescomposable. It follows from the exact sequence (\ref{Delta1}) and  part 2) of Lemma \ref{DeltaLemma1} that 
 $(\Delta_X(i),\Delta_X(i))\cong  (\mathcal{C}(-,X),\Delta_X(i))=\frac{\mathcal{C}(X,X)}{I_{\mathcal{B}_{i-1}}(X,X)}$, however, by  Theorem \ref{TheoremQh} we have $I_{\mathcal{B}_{i-1}}(X,X)\cong \mathrm{rad}(X,X)$. 
 Thus $\mathrm{End}(\Delta_X(i))=\mathrm{End}_{\mathcal{C}}(X)/\mathrm{rad}(\mathrm{End}_{\mathcal{C}}(X))$ is 
 division algebra.
 \end{proof}

\begin{lemma}\label{DeltaNabla}
1) If $\mathrm{Hom}(\Delta(i),\nabla(j))\neq 0$ implies $i= j$. 2) $\mathrm{Ext}^1(\Delta(i),\nabla(j))= 0$  for all  $i,j$. 
  2') $\mathrm{Tor}_1(\Delta(i),\Delta^\circ (j))= 0$ for all  $i,j$.
\end{lemma}
\begin{proof}

Set  $\nabla_Y(j)=D\left(\mathcal{C}(Y,-)/I_{\mathcal{B}_{j-1}}(Y,-)\right)$, $Y\in\mathcal{B}_j$ and 
$\Delta_X(i)=\mathcal{C}(-,X)/I_{\mathcal{B}_{i-1}}(-,X)$, $ X\in\mathcal{B}_i$.

After applying $(-,\nabla_Y(j))$ to the sequence   $0\rightarrow I_{\mathcal{B}_{i-1}}(-,X)\rightarrow\mathcal{C}(-,X)\rightarrow \Delta_X(i)\rightarrow 0$, we get by the long homology sequence, the exact sequence 
\begin{eqnarray}\label{Delta3}
0\rightarrow (\Delta_X(i),\nabla_Y(j))\rightarrow (\mathcal{C}(-,X),\nabla_Y(j))\rightarrow (I_{\mathcal{B}_{i-1}}(-,X),\nabla_Y(j))\\
\rightarrow \mathrm{Ext}^1(\Delta_X(i),\nabla_Y(j))\rightarrow   \mathrm{Ext}^1((\mathcal{C}(-,X),\nabla_Y(j))=0.\notag
\end{eqnarray}

After applying $(\Delta_X(i),-)$ to the sequence $0\rightarrow \nabla_Y(j)\rightarrow D\mathcal{C}(Y,-)\rightarrow DI_{\mathcal{B}_{j-1}}(Y,-)\rightarrow 0$, we get by the long homology sequence the exact sequence
\begin{eqnarray}\label{Delta4}
0\rightarrow (\Delta_X(i),\nabla_Y(j))\rightarrow (\Delta_X(i), D\mathcal{C}(Y,-))\rightarrow (\Delta_X(i),DI_{\mathcal{B}_{i-1}}(Y,-))\\
\rightarrow\mathrm{Ext}^1(\Delta_X(i),\nabla_Y(j))\rightarrow \mathrm{Ext}^1 (\Delta_X(i), D\mathcal{C}(Y,-))=0.\notag
\end{eqnarray}

\textbf{1)} First, let us assume $i < j$, then $i\le j-1$ and    $\mathcal{B}_i\subset \mathcal{B}_{j-1}$. Since $X\in\mathcal{B}_{i}\subset\mathcal{B}_{j-1}$, we have  
\begin{eqnarray}\label{Delta5}
(\mathcal{C}(-,X),\nabla(j))\cong D(\frac{\mathcal{C}(Y,X)}{I_{{\mathcal B}_{j-1}}(Y,X)})=0
\end{eqnarray}

It follows from (\ref{Delta3})  and (\ref{Delta5}) that $\mathrm{Hom}(\Delta_X(i),\nabla_Y(j))=0$ when  $i<j$.

On the other hand, let us assume  $j<i$, then $j\le i-1$ and $\mathcal{B}_j\subset\mathcal{B}_{i-1}$. Since $Y\in\mathcal{B}_{j}\subset\mathcal{B}_{i-1}$, we have  
\begin{eqnarray}\label{Delta6}
 (\Delta(i), D\mathcal{C}(Y,-))\cong(\mathcal{C}(Y,-),D\Delta(i))\cong D\frac{\mathcal{C}(Y,X)}{I_{\mathcal{B}_{i-1}}(Y,X)}\cong 0
\end{eqnarray}

It follows from (\ref{Delta4})  and (\ref{Delta6}) that $\mathrm{Hom}(\Delta_X(i),\nabla_Y(j))=0$ when  $j<i$

\textbf{2)}  Assume that $i\le j$, then $\mathcal{B}_i\subset\mathcal{B}_j$. There is an exact sequence 
\[
(-,E_{i-2})\rightarrow (-,E_{i-1})\rightarrow I_{\mathcal{B}_{i-1}}(-,X)\rightarrow 0,
\]
with $ E_{i-1}\in\mathcal{B}_{i-1},  E_{i-2}\in\mathcal{B}_{i-2}$. 
After applying $(-,\Delta(j))$ to the above sequence  we get a monomorphism
\[
0\rightarrow(I_{\mathcal{B}_{i-1}}(-,X),\nabla_Y(j))\rightarrow(\mathcal{C}(-,E_{i-1}),\nabla_Y(j))\cong D\frac{\mathcal{C}(Y,E_{i-1})}{I_{\mathcal{B}_{j-1}}(Y,E_{i-1})}=0
\]
It follows from  (\ref{Delta3}) that  $\mathrm{Ext}^1(\Delta_X(i),\nabla_Y(j))=0$ when $i\le j$. 

Similarly, by using the exact sequence (\ref{Delta4}), we can prove $\mathrm{Ext}^1(\Delta_X(i),\nabla_Y(j))=0$ when $j\le i$.
\end{proof}

\subsection{Trace filtrations. The categories  $\mathcal{F} (\Delta$) and $\mathcal{F} (\nabla$)}
Throughout this subsection $\mathcal{C}$  will be  a $\mathrm{Hom}$-finite  Krull-Schmidt quasi-hereditary category $K$-category $\mathcal{C}$ with an exhaustive filtration $\{\mathcal{B}_i\}_{i\ge 0}$
\[
 0=\mathcal{B} _0\subset \mathcal{B}_1\subset\cdots\subset \mathcal{B}_{j-1}\subset\mathcal{B}_j\subset\cdots\subset 
\mathcal{C}.
\]
\bigskip

 We remember that under these conditions finitely presented functors have projective covers \cite{MVO2}. In this part we introduce some special subcategories related to the described filtrations  in Theorem \ref{TheoremQh}.  First we assume that the filtration $\{\mathcal{B}_i\}_{i\ge 0}$ is  not necessarily finite and we  obtain  general results. After we add the condition that the filtration is finite, obtaining in this way results similar  to those appearing  in \cite{Rin2}.
\bigskip

 For every $\mathcal{C}$-module $F$, there is 
an associated filtration.
\begin{definition}
 Given a $\mathcal{C}$-module $F$, define its trace filtration (with respect to  $\{\mathcal{B}_j\}_{j\ge 0}$)  by
 \[
  0=F^{(0)}\subset F^{(1)}\subset \cdots F^{(j-1)}\subset F^{(j)}\subset\cdots,
 \]
where $F^{(j)}=\mathrm{Tr}_{\{\mathcal{C}(-,E)\}_{E\in\mathcal{B}_j}}F$ and $F=\bigcup_{j\ge 0 }F^{(j)}$.
\end{definition}
\bigskip

In this part, we shall focus on the $\mathcal C$-modules  $F$ whose trace filtrations satisfy the condition that 
$F^{(i)}/F^{(i-1)}$  is a finite direct sum of objects from the category $\Delta(i)$ for every $i\ge 1$. In this case we say that
these $\mathcal{C}$-modules possess a $\Delta$-filtration. Denote the full subcategory of all $\mathcal{C}$-modules with $\Delta$-filtration
by  $\mathcal{F}(\Delta)$. 
\bigskip

First we give some descriptions of the categories of $\mathcal{C}$-modules with a $\Delta$-filtration in the case where the filtration $\{\mathcal{B}_i\}_{i\ge 0}$. Finally, we  consider the finite  case,  obtaining a characterization as that given in \cite{Dlab}. In order to do this, we start with a list of lemmas. 
\bigskip


\begin{lemma}\label{traceFunctor}
 Let $F\in\mathcal{F}(\Delta)$. Then

\begin{itemize}
 \item[(i)]  For all $i\ge 0$, $F^{(i)}$ has a presentation
\begin{eqnarray}\label{resFilt}
 \mathcal{C}(-,E_{i-1})\rightarrow\mathcal{C}(-,E_{i})\rightarrow F^{(i)}\rightarrow 0
\end{eqnarray}
with $E_{i-1}\in\mathcal{B}_{i-1}$ and  $E_{i}\in\mathcal{B}_{i}$.

\item[(ii)] $F^{(i)}\cong \mathcal{C}\otimes_{\mathcal{B}_i}(F|_{\mathcal{B}_i})$, (see Proposition A.3.2 in \cite{Dlab}).
\end{itemize}

\end{lemma}
\begin{proof}
We prove (i) by induction. Since  $I_{\mathcal{B}_0}(-,E)=0$ for all $E\in\mathcal{B}_1$, we have 
$\Delta(1)=\{\mathcal{C}(-,E)|E\in\mathcal{B}_1\}$.  Let $F\in\mathcal{F}(\Delta)$, then $F^{(0)}=0$, and  $F^{(1)}/F^{(0)}=F^{(1)}$ is a finite direct
 sum 
of objects of $\Delta(1)$.
Assume that there exists an exact sequence for $F^{(i-1)}$:
\begin{eqnarray}\label{FDelta11}
\mathcal{C}(-,E_{i-2})\rightarrow \mathcal{C}(-,E_{i-1})\rightarrow F^{(i-1)}\rightarrow 0.
\end{eqnarray}
On the other hand we have an exact sequence 
\begin{eqnarray}\label{FDelta12}
0\rightarrow F^{(i-1)}\rightarrow  F^{(i)}\rightarrow\frac{ F^{(i)}} { F^{(i-1)}}\rightarrow 0.
\end{eqnarray}

Since $F\in\mathcal{F}(\Delta)$, it follows that   $\frac{ F^{(i)}} { F^{(i-1)}}$ is  a  finite direct sum of objects $ \Delta_{E}(i)$, ${E\in\mathcal B_i}$. We
write $\Delta_{E}(i)= \frac{\mathcal{C}(-,E)}{I_{\mathcal{B}_{i-1}}(-,E)}$. Thus, we have an exact sequence for all $E\in\mathcal{B}_i$
\[
0\rightarrow I_{\mathcal{B}_{i-1}}(-,E)\rightarrow\mathcal{C}(-,E)\rightarrow \frac{\mathcal{C}(-,E)}{I_{\mathcal{B}_{i-1}}(-,E)}\rightarrow 0
\]
Furthermore, there is an exact sequence
 \[
 \mathcal{C}(-,E_{i-1}')\rightarrow \mathcal{C}(-,E_{i-1}')\rightarrow I_{\mathcal{B}_{i-1}}(-,E)\rightarrow 0,
 \]
  by Theorem \ref{TheoremQh}.
Thus, we have an exact sequence
\begin{eqnarray}\label{FDelta13}
\mathcal{C}(-,E_{i-1}')\rightarrow\mathcal{C}(-,E)\rightarrow  \Delta_{E}(i)\rightarrow 0
\end{eqnarray}
 for all $E\in\mathcal{B}_i$. 
The desired resolution is obtained from the exact sequences (\ref{FDelta11}), (\ref{FDelta12}) and (\ref{FDelta13}).

(ii) First observe that $(\frac{ F} {F^{(i)}})^{(i)}=0$. Thus $((-,E),\frac{ F} {F^{(i)}})=0$ for all $E\in\mathcal{B}_i$, and therefore
$\frac{ F} {F^{(i)}}(E)=\frac{ F} {F^{(i)}}|_{\mathcal{B}_i}(E)=((-,E),\frac{ F} {F^{(i)}})=0$, for all $E\in\mathcal{B}_i$, i.e, 
$\frac{ F} {F^{(i)}}|_{\mathcal{B}_i}=0$. After then appliying the exact functor $|_{\mathcal{B}_i}$ to the exact sequece 
 $ 0\rightarrow F^{(i)}\rightarrow F\rightarrow \frac{F}{F^{(i)}}\rightarrow 0$, we obtain
 \begin{eqnarray}\label{FDelta14}
 F^{(i)}|_{\mathcal{B}_i}\cong F|_{\mathcal{B}_i}
 \end{eqnarray}

On the other hand, it follows from  (\ref{resFilt}) that $F^{(i)}$ is projectively presented over $\mathcal{B}_i$. 
Thus, $\mathcal{C}\otimes_{\mathcal{B}_i}(F^{(i)}|_{\mathcal{B}_i})\cong F^{(i)}$ (see [Au])  and by using (\ref{FDelta14})  we obtain (ii).
\end{proof}

\begin{lemma}
Assume that $F\in\mathcal{F}(\Delta)$. Then the following holds:
\begin{itemize}
\item[(i)] $F$ is locally finite.
\item[(ii)] $F$ is finitely presented if and only if   $F=F^{(i)}$ for some $i\in\mathbb{N}$.
\item[(iii)] If the filtartion $\{0\}=\mathcal{B}_0\subset\cdots \mathcal{B}_n=\mathcal{C}$  is finite, $\mathcal{F}(\Delta)$ consists of finitely presented functors.

\end{itemize}
\end{lemma}

\begin{proof}
(i) Let $B\in\mathcal{C}$, then $B\in\mathcal{B}_i$ for some $i\in\mathbb{N}$. By Lemma \ref{traceFunctor} there exists an exact sequence 
$\mathcal{C}(-,E_{i-1})\rightarrow\mathcal{C}(-,E_{i})\rightarrow F^{(i)}\rightarrow 0$. 
If follows that $\dim_K F^{(i)}(B)<\infty$.  Since $F|_{\mathcal{B}_i}\cong F^{(i)}|_{\mathcal{B}_i}$ by (\ref{FDelta14}), we have 
$F(B)=F|_{\mathcal{B}_i}(B) =F^{(i)}(B)$, i.e, $\dim_K F(B)<\infty$.
 
(ii) Assume that there exists an exact sequence
\[
\mathcal{C}(-,C')\rightarrow\mathcal{C}(-,C')\rightarrow F\rightarrow 0,
\]
with $C,C'\in\mathcal{C}$. Then, there exists some $i\ge 0$ such that  $C,C'\in\mathcal{B}_i$. In this way, $F$ is projectively presented
over $\mathcal{B}_i$ and $\mathcal{C}\otimes_{\mathcal B_i}F|_{\mathcal B_i}\cong F$, 
but $F^{(i)}\cong \mathcal{C}\otimes_{\mathcal B_i}(F|_{\mathcal B_i})$ by part (i) of Lemma (\ref{traceFunctor}). 

The other implication  is a consequence  of  part (ii) of (\ref{traceFunctor}).  
(iii)  If  the filtration is finite: $0=\mathcal{B}_0\subset\mathcal{B}_1\subset\cdots\mathcal{B}_n=\mathcal{C}$, we have $F^{(n)}=F$, and the assertion follows from (ii).
\end{proof}
\bigskip

Now we want to have a more specific characterization of $\mathcal{F}(\Delta)$ when the filtration  $0=\mathcal{B}_0\subset\mathcal{B}_1\subset\cdots\mathcal{B}_n=\mathcal{C}$  is finite.  For this we have the following lemma. 
\bigskip

 Assume that $\mathcal{C}$ is a quasi-hereditary with respect to a filtration $\{\mathcal{B}_i\}_{i\ge 0}$ which is not necessarialy finite.

\begin{lemma}\label{Traceexact}
Let $f:F\rightarrow G$ be an epimorphism of finitely presented $\mathcal{C}$-modules, for which there are
 projective presentations  $ \mathcal{C}(-,E_{i-1})\rightarrow \mathcal{C}(-,E_i)\rightarrow F^{(i)}\rightarrow 0$,  $ \mathcal{C}(-,E_{i-1}')\rightarrow \mathcal{C}(-,E_i')\rightarrow G^{(i)}\rightarrow 0$, with  $E_{i-1},E_{i-1}'\in\mathcal{B}_{i-1}$ and $E_{i}, E_{i}'\in\mathcal{B}_i$, for all $i\ge 0$. Then,
 $f$ induces an epimorphism $f^{(i)}:F^{(i)}\rightarrow G^{(i)}$ and the following
 short exact sequence
 \begin{eqnarray}\label{Traceexact11}
  0\rightarrow (\mathrm{Ker}f)^{(i)}/(\mathrm{Ker}f)^{(i-1)}\rightarrow F^{(i)}/(F)^{(i)} \rightarrow G^{(i)}/G^{(i-1)} \rightarrow 0
 \end{eqnarray}
 for all $i\ge 1$.
\end{lemma}
\begin{proof}
The morphism $F^{(i)}\rightarrow G^{(i)}$ induced by $f$ is surjective because every map from 
$\mathcal{C}(-,E)\rightarrow G$, $E\in\mathcal{B}_i$, lifts to $F$.  In this way, we  have a commutative diagram
\[
\dgARROWLENGTH=.3em
\dgTEXTARROWLENGTH=.5em
\begin{diagram}
\node{0}\arrow{e,t}{}
 \node{K\cap F^{(i)}}\arrow{e,t}{}\arrow{s,l}{}
  \node{F^{(i)}}\arrow{e,t}{}\arrow{s,l}{}
   \node{G^{(i)}}\arrow{e,t}{}\arrow{s,l}{}
   \node{0}\\
   \node{0}\arrow{e,t}{}
 \node{K}\arrow{e,t}{}
  \node{F}\arrow{e,t}{}
   \node{G}\arrow{e,t}{}
   \node{0.}
\end{diagram}
\]
We show that $K^{(i)}=K\cap F^{(i)}$. Indeed let $E\in\mathcal{B}_i$ and consider any map 
$\psi:\mathcal{C}(-,E)\rightarrow K$. Thus,  the image of $\mathcal{C}(-,E)\xrightarrow{ \psi}K\rightarrow F$
is contained in $F$ and therefore in $F^{(i)}$. Thus $K^{(i)}\subset K\bigcap F^{(i)}$.

On the other hand, since $K\bigcap F^{(i)}\subset K$, we have $(K\bigcap F^{(i)})^{(i)}\subset K^{(i)}$. It is
enough to prove $(K\bigcap F^{(i)})^{(i)}=K\bigcap F^{(i)}$. Consider the following minimal projective presentations
\begin{eqnarray*}
 \mathcal{C}(-,E_{i-1})\rightarrow \mathcal{C}(-,E_i)\rightarrow F^{(i)}\rightarrow 0,\\
 \mathcal{C}(-,E_{i-1}')\rightarrow \mathcal{C}(-,E_i')\rightarrow G^{(i)}\rightarrow 0,
\end{eqnarray*}
with $E_{i-1},E_{i-1}'\in\mathcal{B}_{i-1}$ and $E_{i}, E_{i}'\in\mathcal{B}_i$.
Thus, we have the following commutative diagram:
\[
\dgARROWLENGTH=.3em
\dgTEXTARROWLENGTH=.5em
\begin{diagram}
\node{}
 \node{}
  \node{0}\arrow{s,l}{}
   \node{0}\arrow{s,l}{}
   \node{}
    \node{}\\
\node{}
 \node{0}\arrow{e,t}{}
  \node{\Omega F^{(i)}}\arrow{e,t}{}\arrow{s,l}{}
   \node{L}\arrow{e,t}{}\arrow{s,l}{}
   \node{K\bigcap F^{(i)}} \arrow{e,t}{}
    \node{0}\\
\node{}
 \node{}
  \node{\mathcal{C}(-,E_i)}\arrow{e,t,=}{}\arrow{s,l}{}
   \node{\mathcal{C}(-,E_i)}\arrow{s,l}{}
   \node{}
    \node{}\\
   \node{0}\arrow{e,t}{}
 \node{K\bigcap F^{(i)}}\arrow{e,t}{}
  \node{F^{(i)}}\arrow{e,t}{}
   \node{G^{(i)}}\arrow{e,t}{}
   \node{0.}
    \node{}
\end{diagram}
\]

Thus, $\mathcal{C}(-,E_i')$ is a direct summand of $\mathcal{C}(-,E_i )$, i.e, 
$\mathcal{C}(-,E_i )\cong \mathcal{C}(-,E_i ')\coprod Q$ for some $\mathcal{C}$-module $Q$, and as a consequence  $L\cong Q\coprod \Omega G^{(i)}$. In this
way $L$ can be covered  by $\mathcal{C}(-,E_i )\coprod \mathcal{C}(-,E_{i-1} ')$. Thus,  there is an epimorphism 
\[
 \mathcal{C}(-,E_i )\coprod \mathcal{C}(-,E_{i-1} ')\rightarrow K\bigcap F^{(i)},
\]
and therefore $ (K\bigcap F^{(i)})^{(i)}= K\bigcap F^{(i)}$.

\end{proof}

Assume $\mathcal{C}$ is Noetherian and  $F$ is  a finitely presented functor. Thus, the subfunctor  $F^{(i)}\subset F$ is finitely generated, and therefore finitely presented.  In this way, the existence of a epimorphism $\coprod_{E\in\mathcal{B}_i}(-,B_i)\rightarrow F^{(i)}$ implies the existence of a epimorphism  
$(-,B)\rightarrow F^{(i)}$ with $B\in\mathcal{B}_i$;  morever $F^{(i)}$ has a  presentation   as that required in the conditions of Lemma \ref{Traceexact}. On the other hand if $F$ is in $\mathcal{F}(\Delta)$  we have the same presentation by Lemma \ref{traceFunctor}.
\bigskip

  In the next results, we  have more explicit characterizations of $\mathcal{F}(\Delta)$ when $0=\mathcal{B}_0\subset\mathcal{B}_1\subset\cdots\mathcal{B}_n=\mathcal{C}$  is a finite filtration.

\begin{theorem}\label{FDelta}
\begin{itemize}
\item[(i)]  Assume that $\mathcal{C}$ has a finite filtration 
 $0=\mathcal{B}_0\subset\mathcal{B}_1\subset\cdots\mathcal{B}_n=\mathcal{C}$. Let $F\in\mathcal{F}(\Delta)$, then
 $\mathrm{Ext}^1(F,\nabla)=\mathrm{Tor}_1(F,\Delta^\circ)=0$.
\item[(ii)] $\mathcal{F}(\Delta)$ is closed under kernels of epimorphisms. In adition, if $0=\mathcal{B}_0\subset\mathcal{B}_1\subset\cdots\mathcal{B}_n=\mathcal{C}$ is finite we have
\[
 \mathrm{Ext}^t(F,\nabla)=\mathrm{Tor}_t(F,\Delta^\circ)=0, \ \text{for all } F\in \mathcal{F}(\Delta) \text{ and }t\ge 1.
\]
\end{itemize}
\end{theorem}
\begin{proof}
(i) It can be proved by induction that $\mathrm{Ext}^1(F^{(i)},\nabla)=0$. The  case $i=1$ is trivial since
 $F^{(1)}$ is projective. Assume that $\mathrm{Ext}^1(F^{(i-1)},\nabla)=0$; hence,  after applying
 $\mathrm{Ext}^1(-,\nabla)$ to the exact sequence 
 $0\rightarrow F^{(i-1)}\rightarrow F^{(i)}\rightarrow F^{(i)}/F^{(i)}\rightarrow 0$ we have 
 $\mathrm{Ext}^1(F^{(i)},\nabla)=0$ by Lemma \ref{DeltaNabla}.
 
 (ii) Let $0\rightarrow\mathrm{Ker} f\rightarrow F\xrightarrow{f} G\rightarrow 0$  be an epimorphism with $F, G\in\mathcal{F}(\Delta)$. For each $i\ge 1$, Lemma \ref{Traceexact}  gives the (split) exact sequence (\ref{Traceexact11}) of projective $\mathcal{C}/I_{\mathcal{B}_i}$-modules. Hence, 
$(\mathrm{Ker}f)^{(i)}/(\mathrm{Ker}f)^{(i-1)}$ is a direct summand of $F^{(i)}/F^{(i-1)}$. Given $F\in\mathcal{F}(\Delta)$
 and an exact sequence  $0\rightarrow F'\rightarrow \mathcal{C}(-,C)\rightarrow F\rightarrow 0$, we have 
 $F'\in\mathcal{F}(\Delta)$ because $\mathcal{C}(-,C)\in\mathcal{F}(\Delta)$ since $\mathcal{C}$ is quasi-hereditary. Therefore,
  $\mathrm{Ext}^{t+1}(F,\nabla)=\mathrm{Ext}^t(F',\nabla)=0$ for all $t\ge 1$, and the proof follows by induction and part (i).
\end{proof}
\bigskip

To obtain a result similar to $\mathrm{mod}(\Lambda)$  for  finite-dimensional algebras  $\Lambda$,  we add the condition that $\mathcal{C}$ is Noetherian;
in this way  if we take $F\in\mathrm{mod}(\mathcal{C})$, then $F^{(i)}\in \mathrm{mod}(\mathcal{C})$ for every $i\ge 1$.
\bigskip

Now we give one of the main results of this section (see Proposition A.2.3 in  \cite{Dlab}).

\begin{theorem}
Assume $\mathcal{C}$ is Noetherian and  $(\mathcal{B})$ is a finite filtration for $\mathcal{C}$. Then,
\[
 \mathcal{F}(\Delta)=\{F|\mathrm{Tor}_1(F,\Delta^{\circ})=0\}.
\]
\end{theorem}
\begin{proof}
By Lemma $\mathcal{F}(\Delta)\subset \{F|\mathrm{Tor}_1(F,\Delta^{\circ})=0\}$.

 Now we will prove by reverse induction that
 \[
  \{F|\mathrm{Tor}_1(F,\Delta^{\circ})=0\text{ and } F^{(j)}=0\}\subset \mathcal{F}(\Delta),
 \]
 for $ j=1,\ldots,n$.
 
The case $j=n$ is trivial. Assume that the statement is true for $j=i$.  Let $F\in\{X|\mathrm{Tor}_1(X,\Delta^{\circ})=0\text{ and } X^{(i-1)}=0\}$. We have the
following exact sequences

\begin{eqnarray}
0\rightarrow F^{(i)}\rightarrow F\rightarrow G\rightarrow 0,\ \ G^{(i)}=0,\label{Fdelta}\\
 0\rightarrow H\rightarrow \coprod \Delta_{E'}(i)\rightarrow F^{(i)}\rightarrow 0, \ \ \text{with\ } H^{(i)}=0 \label{Fdelta1},
\end{eqnarray}
where $\coprod \Delta_{E'}(i)$ is a finite coproduct of $\mathcal{C}$-modules 
$ \Delta_{E'}(i)=\mathcal{C}(-,E')/I_{\mathcal{B}_{i-1}}(-,E')$,  with $E'\in\mathcal{B}_i$.

\textbf{a}) First, we show the existence of the exact  sequence (\ref{Fdelta1}).
 
  Let $\mathcal{C}(-,E_i)\xrightarrow{\varphi} F^{(i)}\rightarrow 0,\ E_i\in\mathcal{B}_i$, be a projective cover. Thus,  we have a 
  commutative diagram
  \[
  \dgARROWLENGTH=.8em
\dgTEXTARROWLENGTH=.5em
   \begin{diagram}
    \node{}
     \node{}
     \node{I_{\mathcal{B}_{i-1}}(-,E_i)}\arrow{e,t}{}\arrow{s,l,J}{}
      \node{F^{(i-1)}=0}\arrow{s,l,J}{}
       \node{}\\
   \node{0}\arrow{e,t}{}
    \node{K}\arrow{e,t}{u}
    \node{\mathcal{C}(-,E_i)}\arrow{e,t}{\varphi}
     \node{F^{(i)}}\arrow{e,t}{}
      \node{0.}
   \end{diagram}
  \]
  
  In this way, we have a commutative diagram
  \[
  \dgARROWLENGTH=.3em
\dgTEXTARROWLENGTH=.5em
   \begin{diagram}
    \node{}
     \node{0}\arrow{s,l}{}
      \node{0}\arrow{s,l}{}
       \node{}
        \node{}\\
   \node{}
    \node{I_{\mathcal{B}_{i-1}}(-,B_i)}\arrow{e,t,=}{}\arrow{s,l}{}
     \node{I_{\mathcal{B}_{i-1}}(-,B_i)}\arrow{s,l}{}
      \node{}
       \node{}\\
   \node{0}\arrow{e,t}{}
    \node{K}\arrow{e,t}{u}\arrow{s,l}{\pi}
     \node{\mathcal{C}(-,E_i)}\arrow{e,t}{\varphi}\arrow{s,l}{\pi'}
      \node{F^{(i)}}\arrow{e,t}{}\arrow{s,l,=}{}
       \node{0}\\
    \node{0}\arrow{e,t}{}
    \node{H}\arrow{e,t}{u'}\arrow{s,l}{}
     \node{\coprod\Delta_{E'}(i)}\arrow{e,t}{\varphi}\arrow{s,l}{}
      \node{F^{(i)}}\arrow{e,t}{}\arrow{s,l}{}
       \node{0}\\
      \node{}
     \node{0}
      \node{0}
       \node{0,}
        \node{}  
      \end{diagram}
  \]
where $\coprod\Delta_{E'}(i)$ is a finite sum, $E'\in\mathcal{B}_i$.

It is clear that $H^{(i-1)}=(\frac{K}{I_{\mathcal{B}_{i-1}}(-,E_i)})^{(i-1)}=0$. Now we show $H^{(i)}=0$. Indeed we put
$E_i=\coprod_{t=1}^sE_t'$ as a sum of indescomposable sumands.
Let $\psi:\mathcal{C}(-,\tilde{E})\rightarrow H$
 any morphism, with $\tilde{E}\in\mathcal{B}_i$ an indecomposable object.  Since $\mathcal{C}(-,\tilde{E})$ is projective, there exists a map 
 $p:\mathcal{C}(-,\tilde{E})\rightarrow K$ such that  $\pi p=\psi$. 
 Let $p_i:\coprod_{t=1}^sE_t'\rightarrow E_i'$ be a projection. It follows
 that the composition
 \[
  \mathcal{C}(-,\tilde{E})\xrightarrow{p} K\xrightarrow{u} \mathcal{C}(-,\coprod_{t=1}^sE_t')\xrightarrow{p_i}\mathcal{C}(-,E_i')
 \]
is an isomorphism or it  lies in $\mathrm{rad}\mathcal{C}(\tilde{E},E_i')=I_{\mathcal{B}_{i-1}}(\tilde{E},E_i')$. Assume the above
morphism is an isomorphism, then we get a contradiction with  the fact that $\varphi$ is a projective cover. Therefore, the above morphism
lies in $I_{\mathcal{B}_{i-1}}(\tilde{E},E_i')$, and we get $qup=u'\psi=0$; finally $\psi=0$ because $u'$ is a monomorphism.
\bigskip

\textbf{b}) Assume that $i<j$, then $F^{(i)}\otimes \Delta^{\circ}(j)=0$. Indeed since $i<j$, we have $i\le j-1$ and 
$\mathcal{B}_{i}\subset \mathcal{B}_{j-1}$. Set $\Delta_{E}^\circ(j)=\frac{\mathcal{C}(E,-)}{I_{\mathcal{B}_{j-1}}(E,-)}\in\Delta^\circ(j)$, 
then there exists an  epimorphism 
$\mathcal{C}(-, E_i)\xrightarrow{\varphi} F^{(i)}\rightarrow 0$, with $E_i\in\mathcal{B}_i$,  which implies the epimorphism 
\[
 0=\frac{\mathcal{C}(E,E_i)}{I_{\mathcal{B}_{j-1}}(E,E_i)}=\mathcal{C}(-,E_i)\otimes \Delta_E^{\circ}(j)\rightarrow F^{(i)}\otimes\Delta_E^{\circ}(j) \rightarrow 0,
\]
since $E\in\mathcal{B}_j$.
\bigskip

\textbf{c}) Assume that $j\le i$, then $\mathrm{Tor}_1(G,\Delta(j))=0$. Set 
$\Delta_{E}(j)=\frac{\mathcal{C}(E,-)}{I_{\mathcal{B}_{j-1}}(E,-)}$. Thus, there exist exact sequences 
\begin{eqnarray}
\mathcal{C} (E_{j-1},-)\rightarrow I_{\mathcal{B}_{j-1}}(E,-)\rightarrow 0, \ E_{j-1}\in\mathcal{B}_{j-1},\label{Fdelta01}\\
0\rightarrow I_{\mathcal{B}_{j-1}}(E,-)\rightarrow\mathcal{C}(E,-)\rightarrow \Delta^{\circ}_E(j)\rightarrow 0. \label{Fdelta02}
\end{eqnarray}
On the other hand, the fact   $j\le i$ implies $\mathcal{B}_{j-1}\le \mathcal{B}_{i-1}$ and
 $G^{(j-1)}\subset G^{(i-1)}\subset G{(i)}=0$. Therefore  
 $G\otimes\mathcal{C}(E_{j-1},-)=G(E_{j-1})=(\mathcal{C}(-,E_{j-1}),G)=0$ because $G^{(j-1)}=0$. Thus, the
 exact sequence (\ref{Fdelta01}) implies the following exact sequence:
 $
  0=G\otimes\mathcal{C}(E_{j-1},-)\rightarrow G\otimes I_{\mathcal{B}_{j-1}}(E,-)\rightarrow 0.
 $
Thus, 
\begin{equation}\label{Fdelta03}
 G\otimes I_{\mathcal{B}_{j-1}}(E,-)=0.
\end{equation}

After applying $G\otimes -$ to the exact sequence (\ref{Fdelta02}) and by using (\ref{Fdelta03}), we have the following exact sequence:
\begin{eqnarray*}
 0\rightarrow\mathrm{Tor}_1(G,\Delta_E^\circ(j))\rightarrow G\otimes I_{\mathcal{B}_{j-1}}(E,-)
 \rightarrow G\otimes \mathcal{C}(E,-)\rightarrow G\otimes\Delta_E^\circ(j)\rightarrow 0,
\end{eqnarray*}
which implies $\mathrm{Tor}_1(G,\Delta_E^\circ(j))=0$, by (\ref{Fdelta03}).
 \bigskip
 
 \textbf{d}) $\mathrm{Tor}_1(G,\Delta^{\circ}(j))=0$ for all $j\ge 0$.  First, observe that after applying $-\otimes \Delta^{\circ}(j)$ to the exact 
 sequence (\ref{Fdelta}) and by using the fact $\mathrm{Tor}_1(F,\Delta^\circ(j))=0$, we obtain  by the long homology sequence the following
 exact sequence  
 \begin{eqnarray}\label{Fdelta07}
  \cdots\rightarrow\mathrm{Tor}_2(G,\Delta^{\circ}(j))\rightarrow \mathrm{Tor}_1(F^{(i)},\Delta^{\circ}(j))\rightarrow 0\rightarrow\mathrm{Tor}_1(G,\Delta^{\circ}(j))\rightarrow \\
    F^{(i)}\otimes\Delta^{\circ}(j)\rightarrow F\otimes\Delta^{\circ}(j)\rightarrow G\otimes\Delta^{\circ}(j)\rightarrow 0,\notag
 \end{eqnarray}
 for all $j\ge 0$.
 \bigskip
 
 By part (c), it only remains to prove  $\mathrm{Tor}_1(G,\Delta^{\circ}(j))=0$ when  $i<j$. By part (b), however, we
 have  $G^{(i)}\otimes \Delta^{\circ}(j)=0$, and 
 therefore we get $\mathrm{Tor}_1(G,\Delta^{\circ}(j))=0$ by the sequence (\ref{Fdelta07}). 
 \bigskip
 
 \textbf{e})  $G\in F(\Delta)$ and $\mathrm{Tor}_1(G,\Delta^{\circ})=\mathrm{Tor}_1(F^{(i)},\Delta^{\circ})=0$.  In  part (d), we 
 proved 
 $\mathrm{Tor}_1(G,\Delta^{\circ}(j))=0$, and since $G^{(i)}=0$ it follows that $G\in F(\Delta)$  by induction hypothesis. Thus,  $\mathrm{Tor}_2(G,\Delta^{\circ})=0$ by Lemma \ref{FDelta}. Finally, it follows from 
 (\ref{Fdelta07}) that $\mathrm{Tor}_1(F^{(i)},\Delta^{\circ})=0$.
 \bigskip
 
\textbf{f}) $H\otimes \Delta^\circ (j)=0$ for all $j\ge 0$.  The exact sequence (\ref{Fdelta1}) implies the following long exact sequence:
\begin{eqnarray}\label{FDelta111}
 \hspace{1.2cm} \mathrm{Tor}_1(F^{(i)},\Delta^\circ(j))\rightarrow 
 H\otimes\Delta^{\circ}(j)\rightarrow  \Delta(i)\otimes\coprod \Delta_{E'}^{\circ}(j)\rightarrow F^{(i)}\otimes \Delta^{\circ}(j)\rightarrow 0.
\end{eqnarray}

 First assume that $i\neq j$. Then, $\Delta(i)\otimes \Delta^{\circ}(j)=0$ by Lemma \ref{DeltaNabla}; therefore 
the exact sequence (\ref{FDelta111}) and part (e) imply that $H\otimes \Delta^\circ (j)=0$. 

Assume $j=i$, then there exist exact sequences
\[
 \mathcal{C}(E',-)\rightarrow \Delta_{E'}^\circ (i)\rightarrow 0, E'\in\mathcal{B}_i
\]
which implies $H\otimes\mathcal{C}(E',-)\rightarrow H\otimes\Delta_{E'}^\circ (i)\rightarrow 0$. Since
$H^{(i)}=0$, we have $H\otimes\mathcal{C}(E',-)=H(E')=(\mathcal{C}(-,E'),H)=0$.  It follows that $H\otimes \Delta^\circ (i)=0$.
\bigskip

\textbf{g}) $H\cong 0$. By induction, let $E_1\in\mathcal{B}_1$ and put $\Delta_{E_1}^\circ(1)=\frac{\mathcal{C}(E_1,-)}{I_{\mathcal{B}_0}(E_1,-)}=
\mathcal{C}(E_1,-)$. It follows that $H(E_1)=H\otimes \mathcal{C}(E_1,-)=H\otimes\Delta_{E_1}^{\circ}(1)=0$. Therefore
$H(E_1)=0$ for all $E_1\in\mathcal{B}_1$.

Assume  $H(E_{i-1})=0$ for all $E_{i-1}\in\mathcal{B}_{i-1}$. Let $E_i\in\mathcal{B}_i$, then $H\otimes I_{\mathcal{B}_{i-1}}(E_i,-)=0$. 
Indeed we have an exact sequence
\begin{eqnarray}
  \hspace{1.2cm} \mathcal{C}(E_{i-2},-)\rightarrow \mathcal{C}(E_{i-1},-)\rightarrow I_{\mathcal{B}_{i-1}}(E_i,-)\rightarrow 0, 
 E_{i-2}\in\mathcal{B}_{i-2},  E_{i-1}\in\mathcal{B}_{i-1}.
\end{eqnarray}
which implies 
\[
 0=H(E_{i-1})=H\otimes \mathcal{C}(E_{i-1},-)\rightarrow H\otimes I_{\mathcal{B}_{i-1}}(E_i,-)\rightarrow 0.
\]
As a result  $H\otimes I_{\mathcal{B}_{i-1}}(E_i,-)=0$.

In this way, there exists an exact sequence 
\[
 0\rightarrow I_{\mathcal{B}_{i-1}}(E_1,-)\rightarrow\mathcal{C}(E_1,-)\rightarrow 
\frac{\mathcal{C}(E_i,-)}{I_{\mathcal{B}_{i-1}}(E_i,-)}=\Delta_{E_i}^\circ(i)\rightarrow 0,
\]
which implies
 \[
 \dgARROWLENGTH=.3em
\dgTEXTARROWLENGTH=.5em
  \begin{diagram}
     \node{H\otimes I_{\mathcal{B}_{i-1}}(E_1,-)}\arrow{e,t}{}\arrow{s,l}{\cong}
      \node{H\otimes\mathcal{C}(E_i,-)}\arrow{e,t}{}\arrow{s,l}{}
       \node{H\otimes\Delta_{E_i}^\circ(i)}\arrow{e,t}{}\arrow{s,l}{\cong}
        \node{0}\\   
  \node{0}\arrow{e,t}{}
   \node{H(E_i)}\arrow{e,t}{}
    \node{0}\arrow{e,t}{}
     \node{0}
      \node{}
  \end{diagram}
 \]
by part (f). Therefore $H(E_i)=0$ for all $E_i\in\mathcal{B}_i$.

\textbf{h}) Since $H\cong 0$ and $G\in\mathcal{F}(\Delta)$, the sequences (\ref{Fdelta}) and (\ref{Fdelta1}) imply that $F\in \mathcal{F}(\Delta)$.

\end{proof}

\subsection{Quasi-hereditary categories and their subcategories}
Assume that $\mathcal{C}$ is  a $\mathrm{Hom}$-finite  Krull-Schmidt quasi-hereditary category $K$-category with an exhaustive filtration $\{\mathcal{B}_i\}_{i\ge 0}$:
\[
 0=\mathcal{B}_0\subset\cdots\subset \mathcal{B}_{i-1}\subset\mathcal{B}_{i}\subset\cdots\subset \mathcal{C}.
\]

 According with the definition, each $\mathcal{B}_i$ is a quasi-hereditary category. We denote by 
\[
 \Delta_{\mathcal{B}_i}=\{\hat{\Delta}(j)|0\le j\le i\},
\]
the corresponding sequence of standard subcategories with respect to the
filtration 
\[
 0=\mathcal{B}_0\subset\cdots \subset \mathcal{B}_{i-1}\subset\mathcal{B}_i
\]
where 
\[
 \hat{\Delta}(j)=\{\mathcal{B}_i(-,E)/\hat{I}_{\mathcal{B}_{j-1}}(-,E)|E\in\mathcal{B}_j\},
\]
 and
 \[
  \hat{I}_{\mathcal{B}_{j-1}}(-,E)=\mathrm{Tr}_{\{\mathcal{B}_i(-,E')|E'\in\mathcal{B}_{j-1}\}}\mathcal{B}_i(-,E).
 \]
\bigskip
 
Thus, there is  an associated category
$\mathcal{F}(\Delta_{\mathcal{B}_i})\subset \mathrm{mod}(\mathcal{B}_i)$  consisting of
$\mathcal{B}_i$-modules which have a $\Delta_{\mathcal{B}_i}$-filtration.
\bigskip

Here we study a relationship between 
$\mathcal{F}(\Delta_{\mathcal{B}_i})$ and $\mathcal{F}(\Delta_{\mathcal{C}})$. Remember that there is a pair of functors
\begin{eqnarray}\label{functors1}
|_{\mathcal {B}_i}:\mathrm{mod}(\mathcal C)\rightarrow \mathrm{mod}(\mathcal B_i), \;\mathcal C\otimes_{\mathcal{B}_i}-: \mathrm{mod}(\mathcal B_i)\rightarrow \mathrm{mod}(\mathcal C).
\end{eqnarray}
\bigskip 
	
We denote by  $ \mathrm{mod}(\mathcal C)^{(i)}$ the full subcategory of all  
$F\in\mathcal{F}(\Delta_{\mathcal{C}})$ for which $F=F^{(i)}$
\bigskip

\begin{theorem}
Let $\mathcal{C}$ be a quasi-hereditary category with respect to $\{\mathcal{B}_i\}_{i\ge 0}$. The restrictions of the functors (\ref {functors1}) define 
an equivalence of  $\mathrm{mod}(\mathcal{C})^{(i)}$ and $\mathcal{F}(\Delta_{{\mathcal B}_i})$
\end{theorem}
\begin{proof}
\textbf{1}) Let $G\in\mathcal{F}(\Delta_{{\mathcal B}_i})$ with filtration
\[
G^{(0)}\subset G^{(1)}\subset\cdots \subset G^{(i-1)}\subset G^{(i)}=G.
\]
We will show that   $\mathcal{C}\otimes_{\mathcal{B}_i}G$ lies in  $ \mathrm{mod}(\mathcal C)^{(i)}$.
\bigskip

Consider the following  exact sequence 
\begin{eqnarray}\label{ext0}
\mathcal{B}_i(-,E\rq{})\rightarrow\mathcal{B}_i(-,E)\rightarrow G\rightarrow 0,
\end{eqnarray}
with $\ E\rq{},E\in\mathcal{B}_i$. Thus,   after applying $\mathcal{C}\otimes_{\mathcal{B}_i}$  to (\ref{ext0}), we obtain
\begin{eqnarray}\label{ext4}
\mathcal{C}(-E\rq{})\rightarrow\mathcal{C}(-,E)\rightarrow \mathcal{C}\otimes_{\mathcal{B}_i}G\rightarrow 0. 
\end{eqnarray}
\bigskip

\textbf{a)} First we prove that 
 \[
(\mathcal{C}\otimes_{\mathcal{B}_i} G)^{(j)}=\begin{cases} \mathcal{C}\otimes_{\mathcal{B}_i} G^{(j)}& \text{if } j< i;\\
                                                                                \mathcal{C}\otimes_{\mathcal{B}_i} G& \text{if } j\ge i.
                                                                                            
 \end{cases}
 \]

a.1) Assume that $j< i$.  Given  $E\in B_i$, there is a pair of functors 
\begin{eqnarray*}
 I_{\mathcal B_j}(-,E)\subset \mathcal{C}(-,E):\mathrm{mod}(\mathcal C)\rightarrow \mathbf{Ab},\\
 \hat{I}_{\mathcal B_j}(-,E)\subset \mathcal{B}_i(-,E):\mathrm{mod}(\mathcal B_i)\rightarrow \mathbf{Ab}.
\end{eqnarray*}

 Clearly, we have 
 \[
 I_{\mathcal B_j}(-,E)|_{\mathcal{B}_i}=\hat{I}_{\mathcal B_j}(-,E).
\]

Now consider the exact sequence   
\[
\mathcal{C}(-,E_{j-1})\rightarrow\mathcal{C}(-,E_{j})\rightarrow I_{\mathcal B_j}(-,E)\rightarrow 0,
\]
with $ E_j\in \mathcal{B}_j,E_{j-1}\in\mathcal{B}_{j-1}$.   Since $j< i$, we have $E_{j},E_{j-1}\in B_i$.  In this way,  $I_{\mathcal{B}_j}(-,E)$ is projectively  presented over $B_i$; thus we have
\begin{eqnarray}\label{ext2}
 I_{\mathcal B_j}(-,E)&=&\mathcal{C}\otimes_{\mathcal{B}_i}( I_{\mathcal B_j}(-,E)|_{\mathcal{B}_i})\\
                               &=&\mathcal{C}\otimes_{\mathcal{B}_i}\hat{I}_{\mathcal B_j}(-,E).\notag
\end{eqnarray}

 After applying  $|_{\mathcal{B}_j}$ followed by $\mathcal{B}_i\otimes_{\mathcal{B}_j}$ to the exact sequence (\ref{ext0}),  we obtain 
 by  Lemma \ref{traceFunctor} the following exact sequence:
\begin{eqnarray}\label{ext1}
 \hat{I}_{\mathcal B_j}(-,E\rq{})\rightarrow \hat{I}_{\mathcal B_j}(-,E)\rightarrow G^{(j)}\rightarrow 0.
\end{eqnarray} 

After applying $\mathcal{C}\otimes_{\mathcal{B}_i}$ to the exact sequence (\ref{ext1}) and by using  (\ref{ext2})  we have the following exact sequence
\begin{eqnarray}\label{ext3}
 {I}_{\mathcal B_j}(-,E)\rightarrow {I}_{\mathcal B_j}(-,E)\rightarrow \mathcal{C}\otimes_{\mathcal{B}_i}G^{(j)}\rightarrow 0
\end{eqnarray}

After   applying $|_{\mathcal{B}_j}$ followed by $\mathcal{C}\otimes_{\mathcal{B}_j}$ to the exact sequence (\ref{ext4}) and by using 
Lemma \ref{traceFunctor},  we obtain the following exact sequence
\begin{eqnarray}\label{ext5}
 {I}_{\mathcal B_j}(-,E\rq{})\rightarrow {I}_{\mathcal B_j}(-,E)\rightarrow( \mathcal{C}\otimes_{\mathcal{B}_i}G)^{(j)}\rightarrow 0.
\end{eqnarray}
Therefore, we obtain
\[
\mathcal{C}\otimes_{\mathcal{B}_i}G^{(j)}\cong ( \mathcal{C}\otimes_{\mathcal{B}_i}G)^{(j)}
\]
by  (\ref{ext5}) and (\ref{ext3}).

a.2)  Assume $j\ge i$. Thus, after applying  $|_{\mathcal{B}_j}$ followed by $\mathcal{C}\otimes_{\mathcal{B}_j}$ to the exact sequence  (\ref{ext4}), we have
\begin{eqnarray}\label{ext6}
I_{\mathcal{B}_j}(-E\rq{})\rightarrow I_{\mathcal{B}_j}(-,E)\rightarrow( \mathcal{C}\otimes_{\mathcal{B}_i}G)^{(j)}\rightarrow 0.
\end{eqnarray}
Since $E\rq{},E\in \mathcal{B}_i\subset\mathcal{B}_j$, we conclude  that $I_{\mathcal{B}_j}(-E\rq{})=\mathcal{C}(-,E\rq{})$ and  $I_{\mathcal{B}_j}(-E)=\mathcal{C}(-,E)$. Thus (\ref{ext6}) can be written as 
 \[
\mathcal{C}(-E\rq{})\rightarrow \mathcal{C}(-,E)\rightarrow( \mathcal{C}\otimes_{\mathcal{B}_i}G)^{(j)}\rightarrow 0.
 \] 
 This implies that  $( \mathcal{C}\otimes_{\mathcal{B}_i}G)^{(j)}=\mathcal{C}\otimes_{\mathcal{B}_i}G$ by (\ref{ext4}).
\bigskip

\textbf{ b)} Let $E\in \mathcal{B}_j$, and assume that $\hat{\Delta}_E(j)=\mathcal{B}_i(-,E)/\hat{I}_{{\mathcal B}_{j-1}}(-,E)$ is in $\hat{\Delta}(j)$. 
 Then $\mathcal{C}\otimes \hat{\Delta}_E(j)$ lies in $\Delta(j)$. Indeed, consider the exact sequence
 \[
 0\rightarrow \hat{I}_{\mathcal{B}_{j-1}}(-,E)\rightarrow\mathcal{B}_{i}(-,E)\rightarrow \hat{\Delta}_E(j)\rightarrow 0.
 \]
Then,  after applying $\mathcal{C}\otimes_{\mathcal{B}_i}$ and using (\ref{ext2}), we have a commutative diagram with exact arrows:
\[
\begin{diagram}\dgARROWLENGTH=1em
\node{}\arrow{e,!}
 \node{\mathcal{C}\otimes_{\mathcal{B}_i}\hat{I}_{\mathcal{B}_{j-1}}(-,E)}\arrow{e,t}{}\arrow{s,l}{\cong}
  \node{\mathcal{C}\otimes_{\mathcal{B}_i}\mathcal{B}_i(-,E)} \arrow{e,t}{}\arrow{s,l}{\cong}
   \node{\mathcal{C}\otimes_{\mathcal{B}_i}\hat{\Delta}_E(j)}\arrow{e,t}{}\arrow{s,l}{}
 \node{0}\\
\node{0}\arrow{e,t}{}
 \node{{I}_{\mathcal{B}_{j-1}}(-,E)}\arrow{e,t}{}
  \node{\mathcal{B}(-,E)} \arrow{e,t}{}
   \node{{\Delta}_E(j)}\arrow{e,t}{}
 \node{0}
\end{diagram} 
\]
Therefore  $\mathcal{C}\otimes \hat{\Delta}_E(j)\cong  \Delta_E(j)\in\Delta(j)$.
\bigskip

\textbf{c)} We have a $\Delta_\mathcal{C}$-filtration 
\[
(\mathcal{C}\otimes G)^{(0)}\subset (\mathcal{C}\otimes G)^{(1)}\subset (\mathcal{C}\otimes G)^{(2)}\subset \cdots
\]
such that $(\mathcal{C}\otimes G)^{(j)}=\mathcal{C}\otimes G$ for $j\ge i$. Indeed  after applying  $\mathcal{C}\otimes_{\mathcal{B}_i}$
to the exact sequence 
\[
0\rightarrow G^{(j-1)}\rightarrow G^{(j)}\rightarrow\frac{G^{(j)}}{G^{(j-1)}}\rightarrow 0,
\]
we obtain the following  commutative diagram with exact arrows:
\begin{eqnarray}\label{Delta6}
\begin{diagram}\dgARROWLENGTH=2em
\node{}\arrow{e,!}{}
 \node{\mathcal{C}\otimes_{\mathcal{B}_i}G^{(j-1)}}\arrow{e,t}{}\arrow{s,l}{\cong}
  \node{\mathcal{C}\otimes_{\mathcal{B}_i}G^{(j)}}\arrow{e,t}{}\arrow{s,l}{\cong}
   \node{\mathcal{C}\otimes_{\mathcal{B}_i}\frac{G^{(j)}}{Y^{(j-1)}}}\arrow{e,t}{}\arrow{s,l}{\cong}
    \arrow{e,t}{}
     \node{0}\\
\node{0}\arrow{e,t}{}
 \node{(\mathcal{C}\otimes_{\mathcal{B}_i}G)^{(j-1)}}\arrow{e,t}{}
  \node{(\mathcal{C}\otimes_{\mathcal{B}_i}G)^{(j)}}\arrow{e,t}{}
   \node{\frac{(\mathcal{C}\otimes_{\mathcal{B}_i}G)^{(j)}}{(\mathcal{C}\otimes_{\mathcal{B}_i}G)^{(j-1)}}}\arrow{e,t}{}
    \arrow{e,t}{}
     \node{0}
\end{diagram}
\end{eqnarray}

Thus, the  map $\mathcal{C}\otimes_{\mathcal B_i}G^{(j-1)}\rightarrow \mathcal{C}\otimes_{\mathcal B_i}G^{(j)}$ is  a monomorphism. Morever, we have
\[
\frac{\mathcal{C}\otimes_{\mathcal{B}_i}G^{(j)}}{\mathcal{C}\otimes_{\mathcal{B}_i}G^{(j-1)}}\cong \frac{(\mathcal{C}\otimes_{\mathcal{B}_i}G)^{(j)}}{(\mathcal{C}\otimes_{\mathcal{B}_i}G)^{(j-1)}}\cong \mathcal{C}\otimes_{\mathcal{B}_i}(\frac{G^{(j)}}{G^{(j-1)}}).
\]
Since $ \frac{G^{(j)}}{G^{(j-1)}}$ is a direct sum of objects from the category $\hat{\Delta}(j)$, we conclude that 
$\mathcal{C}\otimes_{\mathcal{B}_i}\frac{G^{(j)}}{G^{(j-1)}}$  is a direct sum of objects from the category ${\Delta}(j)$ by the part (b).
\bigskip

\textbf{2}) Let $F\in \mathrm{mod}(\mathcal{C})^{(i)}$, and assume that $F$ has a trace filtration
\[
 0=F^{(0)}\subset F^{(1)}\subset\cdots \subset F^{(i-1)}\subset F^{(i)}=F
\]

Since the functor restriction  $|_{\mathcal {B}_i}$ is exact, we have a filtration for  $F|_{\mathcal {B}_i}$:

\begin{eqnarray}\label{traceRest}
 0=F^{(0)}|_{\mathcal {B}_i}\subset F^{(1)}|_{\mathcal {B}_i}\subset\cdots
 \subset F^{(i-1)}|_{\mathcal {B}_i}\subset F^{(i)}|_{\mathcal {B}_i}=F|_{\mathcal {B}_i}.
\end{eqnarray}
We will prove that (\ref{traceRest}) is a $\Delta_{\mathcal{B}_i}$-filtration for $F|_{\mathcal {B}_i}$.
\bigskip

\textbf{a}) Assume that $j\le i$. First we prove that $( F|_{\mathcal{B}_i})^{(j)}\cong F^{(j)}|_{\mathcal{B}_i}$.
\bigskip

By Lemma \ref{traceFunctor} the functor $F^{(i)}$, and therefore $F$  has a presentation
\begin{eqnarray}\label{***}
  \mathcal{C}(-,E_{i-1})\rightarrow\mathcal{C}(-,E_i)\rightarrow F\rightarrow 0, 
\end{eqnarray}
$E_{i-1}\in\mathcal{B}_{i-1}, \ E_i\in\mathcal{B}_{i}$

After applying $|_{\mathcal{B}_i}$ to the exact sequence (\ref{***}), we obtain the following exact sequence:
\begin{eqnarray}\label{****}
 \mathcal{B}_i(-,E_{i-1})\rightarrow\mathcal{B}_i(-,E_i)\rightarrow F|_{\mathcal{B}_i}\rightarrow 0,
\end{eqnarray}
with $ E_{j-1}\in\mathcal{B}_{j-1},E_j\in\mathcal{B}_{j}$.
\bigskip

After applying $\mathcal{C}\otimes_{\mathcal{B}_j}(-|_{\mathcal{B}_j})$ and 
$\mathcal{B}_i\otimes_{\mathcal{B}_j}(-|_{\mathcal{B}_j})$ to the exact sequences (\ref{***}) and
(\ref{****}) respectively, we obtain the following exact sequences:
\begin{eqnarray}
  {I}_{\mathcal B_j}(-,E_{i-1})\rightarrow {I}_{\mathcal B_j}(-,E_i)\rightarrow  F^{(j)}\rightarrow 0, \label{Deltarest}\\
  \hat{I}_{\mathcal B_j}(-,E_{i-1})\rightarrow \hat{I}_{\mathcal B_j}(-,E_i)\rightarrow ( F|_{\mathcal{B}_i})^{(j)}\rightarrow 0\label{tracerest1},
\end{eqnarray}
by Lemma \ref{traceFunctor}.

After applying  $|_{\mathcal{B}_i}$  to the exact sequence (\ref{Deltarest}), we obtain the exact sequence
\begin{eqnarray}\label{tracerest2}
  \hat{I}_{\mathcal B_j}(-,E_{i-1})\rightarrow \hat{I}_{\mathcal B_j}(-,E_i)\rightarrow F^{(j)}|_{\mathcal{B}_i}\rightarrow 0.
\end{eqnarray}
It follows from (\ref{tracerest1}) and (\ref{tracerest2}) that  $( F|_{\mathcal{B}_i})^{(j)}\cong F^{(j)}|_{\mathcal{B}_i}$.
\bigskip

\textbf{b})  $(F|_{\mathcal{B}_i})^{(j)}/(F|_{\mathcal{B}_i})^{(j-1)}$ is a finite sum of
 objects from $\hat{\Delta}(j)$.

Let $\Delta_{E}(j)=\mathcal{C}(-,E)/I_{\mathcal{B}_{j-1}}(-,E)$ with $E\in\mathcal{B}_j$. Thus, after applying
$|_{\mathcal{B}_i}$ to the exact sequence 
$0\rightarrow I_{\mathcal{B}_{j-1}}(-,E)\rightarrow\mathcal{C}(-,E)\rightarrow \Delta_{E}(j)\rightarrow 0$, we obtain  the
following commutative diagram: 
\[
\begin{diagram}\dgARROWLENGTH=1em
\node{0}\arrow{e,t}{}
 \node{{I}_{\mathcal{B}_{j-1}}(-,E)|_{\mathcal{B}_i}}\arrow{e,t}{}\arrow{s,l}{\cong}
  \node{\mathcal{C}(-,E)|_{\mathcal{B}_i}} \arrow{e,t}{}\arrow{s,l}{\cong}
   \node{\Delta_{E}(j)|_{\mathcal{B}_i}}\arrow{e,t}{}\arrow{s,l}{}
 \node{0}\\
\node{0}\arrow{e,t}{}
 \node{\hat{I}_{\mathcal{B}_{j-1}}(-,E)}\arrow{e,t}{}
  \node{\mathcal{B}_i(-,E)} \arrow{e,t}{}
   \node{\hat{\Delta}_{E}(j)}\arrow{e,t}{}
 \node{0}
\end{diagram} 
\]
Thus, $\Delta_{E}(j)|_{\mathcal{B}_i}\cong\hat{\Delta}_{E}(j)$.

Since $F^{(j)}/F^{(j-1)}$ is a finite sum of objects from $\Delta(j)$, it follows that  the quotient
\[
 \frac{ F^{(j)}|_{\mathcal{B}_i}}{ F^{(j-1)}|_{\mathcal{B}_i}}\cong\frac{F^{(j)}}{F^{(j-1)}}|_{\mathcal{B}_i}
\]
is a finite sum of objects from $\hat{\Delta}(j)$.

\end{proof}

\section{$\mathcal{F}(\Delta)$  is functorially finite }

In order to  to have another analogous result closer to the theory of  finite-dimensional quasi-hereditary algebras, we  add more restrictions to our categories; in  particular, we need the existence of duality. We will assume in this section that  $\mathcal{C} $ is a dualizing Krull-Schmidt   $K$-variety.  In this way, finitely presented functors have projective covers (see Theorem 2 in \cite{MVO1}), and the category of finitely presented functors $\mathrm{mod}(\mathcal{C})$  has enough projective and injective objects. In the rest of this work, all the $\mathcal{C}$-modules we are considering are finitely presented.
\bigskip

The main result of this section is to prove that  $\mathcal{F}(\Delta)$ is functorially finite in $\mathrm{mod}(\mathcal{C})$  if we add the additional condition that $\mathcal{C}$ is Noetherian \cite{Rin2}. We begin the subsection recalling some definitions from \cite{AB, AuRe}.
\bigskip

Let $\mathcal{X}$ be  a full subcategory of $\mathrm{mod}(\mathcal{C})$; a morphism $f : X\rightarrow M$ in $\mathrm{mod}(\mathcal{C})$, with $X$ in
 $\mathcal{X}$, is a right $\mathcal{X}$-approximation of $M$ if $(-,X)_{\mathcal{X}}\xrightarrow{ (-,f)_{\mathcal{X}}} (-,M)_\mathcal{X} $ is an exact sequence.
 Dually, let $\mathcal{Y}$  be a subcategory of $\mathrm{mod}(\mathcal{C})$, a morphism $g: M \rightarrow Y$ is a left $\mathcal{Y}$-approximation of $M$ if $(Y,-)_{\mathcal{Y}}\xrightarrow{(g,-)_{\mathcal{Y}}}(M,-)_{\mathcal{Y}}\rightarrow 0$  is exact.

 A subcategory $X$  of $\mathrm{mod}(\mathcal{C})$ is  contravariantly (covariantly) finite in $\mathcal{C}$ if every object 
$ M$  in $\mathrm{mod}(\mathcal{C})$  has a right (left) $\mathcal{X}$-approximation.

A subcategory $X$  of $\mathrm{mod}(\mathcal{C})$   is  resolving (coresolving) if it satisfies the following three conditions: (a) it is closed under extensions; (b) it is closed under kernels (cokerneles)  of epimorphisms (monomorphisms) 
and (c) it contains the projective (injective) objects.
\bigskip

A full subcategory $\mathcal{X}$ in $\mathrm{mod}(\mathcal{C})$ is said to be functorially finite in $\mathrm{mod}(\mathcal{C})$  provided every $\mathcal{C}$-module has both  a right $\mathcal{X}$-approximation and  a left $\mathcal{X}$-approximation.
Assume that $\mathcal{C}$ is  quasi-hereditary  with respect to a finite filtration $0=\mathcal{B}_0\subset\cdots\subset \mathcal{B}_n=\mathcal{C}$. We follow closely the arguments given in  \cite{Rin2},  so we obtain the following

\begin{theorem}[\cite{Rin2}, Theorem 1]\label{RingT} Let $\mathcal{C}$ be    a dualizing Krull-Schmidt Noetherian  $K$-category. Assume $\mathcal{C}$ is quasi-hereditary with respect to a finite filtration  $0=\mathcal{B}_0\subset\cdots\subset \mathcal{B}_n=\mathcal{C}$, then  $\mathcal{F}(\Delta)$ is functorially finite in $\mathrm{mod}(\mathcal{C})$.
\end{theorem}

We start with an arbitrary subcategory $\mathcal{X}$ of $\mathrm{mod}(\mathcal{C})$, and we denote by $\mathcal{Y}$
 the full subcategory of $\mathrm{mod}(\mathcal{C})$ of all modules $Y$ satisfying $\mathrm{Ext}^1_{\mathcal{C}}(X,Y)=0$.
 
 \begin{lemma}\label{waka1}
  Let $0\rightarrow Y\rightarrow X\xrightarrow{\gamma} M\rightarrow 0$ be exact, with $X\in\mathcal{X}$ and 
  $Y\in\mathcal{Y}$. Then $\gamma$ is a right $\mathcal{X}$-approximation of $M$.
 \end{lemma}
 \begin{proof}
  This is the converse of Wakamatsu's lemma, and it is proved in [Ring].
 \end{proof}

Now let $\mathcal{X}=\mathcal{F}(\Delta)$. Then $\mathcal{Y}=\mathcal{Y}(\Delta)$ may be characterized alternatively as the full subcategory of $\mathcal{C}$-modules $\mathcal{Y}$ satisfiying $\mathrm{Ext}^1_{\mathcal{C}}(\Delta(i),Y)=0$  for $1\le i\le n$.
\bigskip

In the  rest of this section  $\mathcal{C}$, will be   a dualizing Krull-Schmidt Noetherian  $K$ quasi-hereditary category with respect to a finite filtration $(\mathcal{B})$.

\begin{lemma}\label{aprox} 
 Assume that $\mathcal{X}$ is closed under extensions and that for every $\mathcal{C}$-module $N$ there exists an exact sequence
 $0\rightarrow N\rightarrow  Y^N\rightarrow X^N\rightarrow  0$ with $X^N\in \mathcal{X}$ and $Y^N\in\mathcal{Y}$. Then every $\mathcal{C}$-module
 $M$ has a right $\mathcal{X}$-approximation.
\end{lemma}
\begin{proof}
 \textbf{a}) Let $M$ be a finitely generated $\mathcal{C}$-module. Then, there is an
 epimorphism $\pi:X\rightarrow M$ with $X\in\mathcal{X}$. Let $M'$  be the submodule of $M$ generated by the images of maps $X'\rightarrow M$ with $X'\in \mathcal{X}$. In this way, there  exists 
 an epimorphism $\coprod_{X\in\mathcal{F}(\Delta)}X\xrightarrow{\psi} M'$. Since $M'$ is a submodule of the finitely generated 
 $\mathcal{C}$-module $M$, there exists an epimorphism $\varphi:\mathcal{C}(-,E)\rightarrow M'$, since $\mathcal{C}$ is Noetherian. Therefore, there exists
  $i\ge 1$ such that $E\in\mathcal{B}_i$ and $\varphi$ lifts to $\psi$, in other words,   the exists $f:\mathcal{C}(-,E)\rightarrow \coprod_{X\in\mathcal{F}(\Delta)}X$ such that $\psi f=\varphi$. Then,  there is
   a finite family $\{X_j\}_{j=1}^n $ of objects in  $\mathcal{F}(\Delta)$ for which the image of $f$ is contained in $\coprod_{j=1}^n X_j$. Thus, there exists
   an epimorphism $\coprod_{j=1}^n X_j^{(i)}\rightarrow M'$. Since $\mathcal{F}(\Delta)$ is closed  under direct sums, we can take 
   $X=\coprod_{j=1}^n X_j^{(i)}$.
   \bigskip
   
 \textbf{b}) By part (a), there exists an exact sequence $0\rightarrow K\rightarrow X\xrightarrow{\pi} M'\rightarrow 0$ with $X\in\mathcal{X}$:
 \[
  \dgARROWLENGTH=.3em
\dgTEXTARROWLENGTH=.5em
   \begin{diagram}
    \node{}
    \node{0}\arrow{s,l}{}
     \node{0}\arrow{s,l}{}
      \node{}
       \node{}\\
   \node{0}\arrow{e,t}{}
    \node{K}\arrow{e,t}{}\arrow{s,l}{}
     \node{Y^K}\arrow{e,t}{}\arrow{s,l}{}
      \node{X^K}\arrow{e,t}{}\arrow{s,l,=}{}
       \node{0}\\
    \node{0}\arrow{e,t}{}
    \node{X}\arrow{e,t}{}\arrow{s,l}{\pi}
     \node{Z}\arrow{e,t}{}\arrow{s,r}{\gamma}
      \node{X^K}\arrow{e,t}{}
       \node{0}\\
      \node{}
     \node{M}\arrow{e,t,=}{}\arrow{s,l}{}
      \node{M}\arrow{s,l}{}
       \node{}
        \node{}  \\
     \node{}
    \node{0}
     \node{0}
      \node{}
       \node{}   
      \end{diagram}
  \]

 Since $X$ and $ X^{K}$  belong to $\mathcal{X}$, and $\mathcal{X}$ is closed under extensions, $Z\in\mathcal{X}$. Since $Y^{K}\in\mathcal{Y}$,  we use Lemma \ref{waka1}
  for the exact sequence which appears as middle column and conclude that $\gamma:Z\rightarrow M'$ is a right approximation.
  We denote by $\mu:M'\rightarrow M$ the inclusion map. It is clear that $\mu\gamma$ is a right approximation.
 \end{proof}

Of course $\mathcal{X}=\mathcal{F}(\Delta)$ is closed under extensions and coproducts by Lemma  \ref{FDelta} since  $0=\mathcal{B}_0\subset\cdots\subset \mathcal{B}_n=\mathcal{C}$ is finite. Morever it contains the projective objects of $\mathrm{mod}(\mathcal C)$ because $\mathcal{C}$ is quasi-hereditary; thus, it is a coresolving category of $\mathrm{mod}(\mathcal C)$.
 
 \begin{lemma}\label{lemmaF}
   Let $N\in\mathrm{mod}({\mathcal C})$ and $t\in\{1,\ldots,n\}$.
 \begin{itemize}
  \item [(i)] Assume that $N$ is a finitely generated $\mathcal{C}$-module. Then $\mathrm{Ext}^1(\Delta_E(t),N)$ is
   finitely generated for all $E\in\mathcal{B}_t$.
  \item[(ii)] Assume $N$ has finite length. Then, there is a finite number of $E$'s in $\mathcal{B}_t$ for which 
  $\mathrm{Ext}^1(\Delta_E(t),N)\neq 0$ (see [A3 Prop. 3.10]).
 \end{itemize}
  \end{lemma}
 
\begin{proof}
i) Consider the folowing exact senquences:
 \begin{eqnarray*}
  0\rightarrow I_{\mathcal{B}_t}(-,E)\rightarrow \mathcal{C}(-,E)\rightarrow \Delta_E(t)\rightarrow 0,\\
  \mathcal{C}(-,E_{t-1})\rightarrow \mathcal{C}(-,E_{t})\rightarrow  I_{\mathcal{B}_t}(-,E)\rightarrow 0.
\end{eqnarray*}
After applying $(-,N)$  to the  above the exact sequences, we get the following exact sequences:
\begin{eqnarray*}
0\rightarrow(\Delta_E(t),N)\rightarrow(\mathcal{C}(-,E),N)\rightarrow (I_{\mathcal{B}_t}(-,E),N)\rightarrow 
\mathrm{Ext}^1(\Delta_E(t),N)\rightarrow 0,\\
0\rightarrow ( I_{\mathcal{B}_t}(-,E),N)\rightarrow (\mathcal{C}(-,E_{t}),N)\cong N(E_t).
\end{eqnarray*}
 Thus,  
 $( I_{\mathcal{B}_t}(-,E),N)$ is finitely  generated,  and it  follows that $\mathrm{Ext}^1(\Delta_E(t),N)$ is finitely
 generated.

ii)  First, we prove that if $N$ is   of finite length, there is then a finite number of $E$'s in $\mathcal{B}_t$
for which $\mathrm{Hom}(\Delta_E(t),N)\neq 0$. By induction on the length 
 $l(N)$ of $N$. If  $l(N)=1$, $N$ is simple and the claim is true in this case. Assume that the fact is true for 
 modules with lenght $<l(N)$, and consider  the 
 exact sequence $0\rightarrow S\rightarrow N\rightarrow N/S\rightarrow 0$ which implies 
 \[
  0\rightarrow (\Delta_E(t),S)\rightarrow (\Delta_E(t),N)\rightarrow (\Delta_E(t),N/S)
 \]
Then there is a finite number of $E$'s in $\mathcal{B}_t$
for which $(\Delta_E(t),S)\neq 0$ and $(\Delta_E(t),N/S)\neq 0$; therefore 
there is a finite number of $E$'s in $\mathcal{B}_t$ for which $(\Delta_E(t),N)\neq 0$.

Now consider the exact sequence $0\rightarrow N\rightarrow I\rightarrow \Omega^{-1}N\rightarrow 0$ where $I$ is the injective envelope of $N$. It implies
\begin{eqnarray*}
 0\rightarrow (\Delta_E(t),N)\rightarrow (\Delta_E(t),N)\rightarrow (\Delta_E(t),\Omega^{-1}N)\rightarrow \mathrm{Ext}^1(\Delta_E(t),N)
 \rightarrow 0.
\end{eqnarray*}

Since $(\Delta_E(t),\Omega^{-1}N)\neq 0 $ only for  a finite number of $E$'s  in $\mathcal{B}_t$, it follows
that $\mathrm{Ext}^1(\Delta_E(t),N)\neq 0$ only for  a finite number of $E$'s  in $\mathcal{B}_t$.

\end{proof}

\begin{lemma}\label{Contravariant}
Let $1\le t\le n$. Let $N$  be a $\mathcal{C}$-module with $\mathrm{Ext}^1(\Delta(j),N)=0$ for all $j<t$. Then  
there exists an exact sequence $0\rightarrow N\rightarrow N'\rightarrow Q\rightarrow 0$ with $Q$ a direct sum of objects 
from $\Delta(t)$ and $ \mathrm{Ext}^1(\Delta(j),N')=0$ for all $j\le t$.
\end{lemma}

Let us fix an integer $t\ge 0$ and set $\Delta_{E}(t)=\mathcal{C}(-,E)/I_{\mathcal{B}_{t-1}}(-,E)$. By  Lemma \ref{lemmaF}, the family
\[ 	
B_t=\{E\in\mathcal{B}_t|\mathrm{Ext}(\Delta_E(t),N)\neq 0\}
\]
has a finite number of objects, say $m_E$.  For all $E\in B_t$, let
\[
\{\xi_E^i:0\rightarrow N\rightarrow W_E^i\rightarrow \Delta_E(t)\rightarrow 0|1\le i\le m_E\}
\]  be a $K$-basis for 
$\mathrm{Ext}^1(\Delta_E(t),N)$. Thus, we can  take  $\{\xi_E^i|1\le i\le m_E\}_{E\in B_t}$ as a $K$-basis for  $\mathrm{Ext}^1(\coprod_{E\in B_t}\Delta_E(t),N)$.

In this way, we have  a push-out diagram:
\[
\begin{diagram}
 \dgARROWLENGTH=.3em
\dgTEXTARROWLENGTH=.5em
\node{0}\arrow{e,t}{}
 \node{\coprod_{E\in B_t}N^{m_E}}\arrow{e,t}{}\arrow{s,l}{[1,\ldots,1]}
  \node{\coprod_{E\in B_t}\coprod _{i=1}^{m_E}W_E^i}\arrow{e,t}{}\arrow{s,l}{}
   \node{\coprod_{E\in B_t}\Delta_E(t)^{m_E}}\arrow{e,t}{}\arrow{s,l,=}{}
    \node{0}\\
    \node{0}\arrow{e,t}{} 
     \node{N}\arrow{e,t}{}
     \node{N'}\arrow{e,t}{}
      \node{\coprod_{E\in B_t}\Delta_E(t)^{m_E}}\arrow{e,t}{}
       \node{0.}
\end{diagram}
\]
Let ${u_E ^i}'':N\rightarrow\coprod_{E\in B_t}N^{m_E}$,  ${u_E ^i}':W_E^i\rightarrow\coprod_{E\in B_t}\coprod _{i=1}^{m_E}W_E^i$, 
 $u_E ^i:\Delta_E(t)\rightarrow\coprod_{E\in B_t}\Delta_E(i)^{m_E}$ the corresponding inclusions. Thus, we have the following commutative diagram
 \[
\begin{diagram}
 \dgARROWLENGTH=.3em
\dgTEXTARROWLENGTH=.5em
\node{0}\arrow{e,t}{}
 \node{N}\arrow{e,t}{}\arrow{s,l}{{u_E^i}''}
  \node{W_E^i}\arrow{e,t}{}\arrow{s,l}{{u_E^i}'}
   \node{\Delta_E(i)}\arrow{e,t}{}\arrow{s,l}{{u_E^i}}
    \node{0}\\
\node{0}\arrow{e,t}{}
 \node{\coprod_{E\in B_t}N^{m_E}}\arrow{e,t}{}\arrow{s,l}{[1,\ldots,1]}
  \node{\coprod_{E\in B_t}\coprod _i^{m_E}W_E^i}\arrow{e,t}{}\arrow{s,l}{}
   \node{\coprod_{E\in B_t}\Delta_E(i)^{m_E}}\arrow{e,t}{}\arrow{s,l,=}{}
    \node{0}\\
    \node{0}\arrow{e,t}{} 
     \node{N}\arrow{e,t}{}
     \node{N'}\arrow{e,t}{}
      \node{\coprod_{E\in B_t}\Delta_E(i)^{m_E}}\arrow{e,t}{}
       \node{0.}
\end{diagram}
\]


By  Lemma  \ref{DeltaLemma1},  we have that  $\mathrm{Ext}^1(\Delta(j),\coprod_{E\in B_t}\Delta_E(t)^{m_E})=0$  for all $j\le t$. In this way, after applying $\mathcal{C}(\Delta(j),-)$ to the exact sequence at the bottom, we have

\begin{eqnarray*}
0\rightarrow (\Delta(j),N)\rightarrow (\Delta(j),N')\rightarrow (\Delta(j),\coprod_{E\in B_t}\Delta_E(t)^{m_E})\xrightarrow{\delta}\mathrm{Ext}^1(\Delta(j),N)\rightarrow\\
\mathrm{Ext}^1(\Delta(j),N')
\rightarrow\mathrm{Ext}^1(\Delta(j),\coprod_{E\in B_t}\Delta_E(t)^{m_E})=0
\end{eqnarray*}

Assume that $j<t$. Therefore, by  hypothesis we have $\mathrm{Ext}^1(\Delta(j),N)=0$, and therefore  $\mathrm{Ext}^1(\Delta(j),N')=0$. It only remains to verify the case $j=t$.
Let $\Delta_{E'}(t)=\mathcal{C}(-,E')/I_{\mathcal{B}_{t-1}}(-,E')\in\Delta(t)$. If $E'\notin B_t$, we have $\mathrm{Ext}^1(\Delta_{E'}(t),N)=0$; thus we have
$\mathrm{Ext}^1(\Delta_{E'}(t),N')=0$. If $E'\in B_t$  the connection morphism $\delta$ is an epimorphism, and we conclude $\mathrm{Ext}^1( \Delta_{E'}(t),N')=0$.

\begin{lemma}
Let $1\le t\le n$. Let $N$  be a $\mathcal{C}$-module with $\mathrm{Ext}^1(\Delta(j),N)=0$ for all $j<t$. Then  
there exists an exact sequence $0\rightarrow N\rightarrow Y\rightarrow X\rightarrow 0$ with 
$X\in\mathcal{F}(\{\Delta(1),\ldots,\Delta(t)\})$ and $Y\in\mathcal{Y}(\Delta)$. 
\end{lemma}
\begin{proof}
 By Lemma \ref{Contravariant} we can construct  exact sequences $0\rightarrow N_{i+1}\rightarrow N_i\rightarrow Q_i \rightarrow 0$,
 $1\le i\le t$,  with $\mathrm{Ext}^1(\Delta(j),N_i)=0$ for all $j\le i$, and $Q_i$ is 
   a direct finite  sum of objects from $\Delta(i)$. Set $N=N_{t+1}$, $N_1=Y$ and $X=Y/N$.

   Thus, $Y\in\mathcal{Y}$ and there is a filtration for $X$
 \[
  0\subset N_t/N\subset\cdots\subset N_1/N=Y/N=X
 \]
such that $(N_{i+1}/N)/(N_{i}/N)\cong N_{i+1}/N_{i}\cong Q_i $.
\end{proof}

Since $\Delta(1)$ consists only of projective objects, of particular interest is when $t=2$ in the above theorem. As a consequence, we have the following.

\begin{lemma}
 For every $\mathcal{C}$-module $N$, there exists an exact sequence 
 $0\rightarrow N\rightarrow Y\rightarrow X\rightarrow 0$ with $X\in\mathcal{F}(\Delta)$ and $Y\in\mathcal{Y}(\Delta)$.
\end{lemma}
The proof of  Theorem \ref{RingT} follows as in [ring2], and it follows from  Lemma  \ref{aprox}. Furthermore,  we may  use duality in order also obtain left $\mathcal{F}(\Delta)$-approximations.
\bigskip

Since $\mathcal{C}$ is dualizing, we can define analogously as in \cite{Dlab} and \cite{DR} the category $\mathcal{F}(\nabla)$  as follows. Denote
  $\mathcal{F}(\Delta^{\circ})$ as the full subcategory of covariant $\mathcal{C}$-modules  with $\Delta^{\circ}$-filtration and by
  $\mathcal{F}(\nabla)\cong \mathcal{F}(\nabla^\circ)^\circ$ the category of all contravariant $\nabla$-filtered $\mathcal{C}$-modules. In this way we could try in the future to define  the concept of characteristic category such as the category $\omega=\mathcal{F}(\Delta)\cap\mathcal{F}(\nabla)$ and to link  the theorem obtained by Ringel (see Theorem 5 in \cite{Rin2})  and Theorem 12 in \cite{MVO2}.

\section{Examples}
  In this section, we exhibit some filtrations for the different Auslander-Reiten components to consider them   as quasi-hereditary categories. We have already mentioned the importance of studying these components at the beginning of this work in relation to \cite{MVS3}. Later we obtain  a tilting category  for $\mathbb{Z}A_\infty$ motivated by the theory developed in the third section. Finally,  we show that tensor product of quasi-hereditary categories is again a quasi-hereditary category; in this way we can build more examples from others already given.

\subsection{Auslander Reiten components seem like quasi-hereditary categories}
Let $\Gamma=(\Gamma_0,\Gamma_1, \tau)$ be a translation quiver with translation $\tau:\Gamma_0'\rightarrow \Gamma_0$ defined on a subset $\Gamma_0'\subset\Gamma_0$. The vertices in $\Gamma_0$ which do not belong to $\Gamma_0'$ are called projective. Given  a  translation quiver 
 $\Gamma=(\Gamma_0,\Gamma_1, \tau)$, a  semitranslation can be defined  $\sigma:\Gamma_1'\rightarrow \Gamma_1$, where $\Gamma_1'\subset\Gamma_1$ is the set  of all  arrows $\alpha:a\rightarrow b$ with  $b$ not projective, such that $\sigma(\alpha):\tau b\rightarrow a$ for $\alpha:a\rightarrow b$ \cite{Rin1,ASS}. We can consider the mesh category  $K(\Gamma,\sigma)$, i.e., the quotient category of the path category of $(\Gamma_0,\Gamma_1)$ modulo the mesh ideal. Remember that the mesh ideal in the path category of $(\Gamma_0,\Gamma_1)$ is the one generated by the elements $m_x=\Sigma_{\{\alpha\in\Gamma_1:t(\alpha)=x\}}$, where $t(\alpha)$ means the target of $\alpha$ and $x$ is a non-projective vertex.

 In this part, we take $\Gamma$ as $\mathbb{Z}\Sigma$ where $\Sigma$ is one of the following: $D_\infty$, $A_\infty$ or $A_\infty^\infty$, or $\Gamma=\mathbb{N}\Sigma$,  the full translation  subquiver of $\mathbb{Z}\Sigma$, where $\Sigma$ is an extended Dynkin diagram. In this way, we will see that we think of the mesh category $K(\Gamma,\sigma)$  as a quasi-hereditary category; simultaneously,  we can study the category of representations
 $Rep(K\Gamma,\sigma)$  as the category of $K(\Gamma,\sigma)$-modules.

  Now we show some of the filtrations  for different mesh categories  mentioned above. In order to do this, we pick out full subcategories
  $B_i\subset \mathcal{C}$, and after we consider the filtration $0=\mathcal{B}_0\subset\mathcal{B}_1\subset\cdots\cdots\subset\mathcal{B}_i\subset\cdots\mathcal{C}$, with $\mathcal{B}_i=\mathrm{add}B_i$:  the full subcategory of $\mathcal{C}$ consisting of directed summands of finite sums. Thus we obtain a filtration  of $\mathcal C$ into closely additive subcategories. 
  \bigskip

  \textbf{The category $K(\mathbb{Z}A_\infty^\infty,\sigma)$}.  Let $B_1=\{E^1_1\}$, and $B_i=B_{i-1}\cup\{E^i_j\}_{1\le j\le 4(i-1)}$ as is shown in the next picture.
  
  \begin{center}
\begin{tikzpicture}
\foreach \x in {1,...,7}  
\draw[->] (\x+.1,.1) to (\x+0.4,0.4);
\foreach \x in {1,...,7}  
\draw[->] (\x +.6,0.6) to (\x+.9,.9);
\foreach \x in {1,...,7}  
\draw[->] (\x+.1,1.1) to (\x+0.4,1.4);
\foreach \x in {1,...,7}  
\draw[->] (\x+.6,1.6) to (\x+.9,1.9);
\foreach \x in {1,...,7}  
\draw[->] (\x+.1,2.1) to (\x+0.4,2.4);
\foreach \x in {1,...,7}  
\draw[->] (\x+.6,2.6) to (\x+.9,2.9);

\foreach \x in {1,...,7}  
\draw[->] (\x+.6,0.4) to (\x+.9,.1);
\foreach \x in {1,...,7}  
\draw[->] (\x+.1,.9) to (\x+0.4,0.6);
\foreach \x in {1,...,7}  
\draw[->] (\x+.6,1.4) to (\x+.9,1.1);
\foreach \x in {1,...,7}  
\draw[->] (\x+.1,1.9) to (\x+0.4,1.6);
\foreach \x in {1,...,7}  
\draw[->] (\x+.6,2.4) to (\x+.9,2.1);
\foreach \x in {1,...,7}  
\draw[->] (\x+.1,2.9) to (\x+.4,2.6);

\draw (4.5,1.5) node{\tiny{$E^1_1$}}; 

\draw (4,2) node{\tiny{$E^2_1$}};
 \draw (5,2) node{\tiny{$E^2_2$}};
   \draw (5,1) node{\tiny{$E^2_3$}};
    \draw (4,1) node{\tiny{$E^2_4$}};
    
     \draw (3.5,2.5) node{\tiny{$E^3_1$}};
 \draw (4.5,2.5) node{\tiny{$E^3_2$}};
   \draw (5.5,2.5) node{\tiny{$E^3_3$}};
    \draw (5.5,1.5) node{\tiny{$E^3_4$}};
    \draw (5.5,0.5) node{\tiny{$E^3_5$}};
     \draw (4.5,0.5) node{\tiny{$E^3_6$}};
      \draw (3.5,0.5) node{\tiny{$E^3_7$}};
       \draw (3.5,1.5) node{\tiny{$E^3_8$}};

     \draw (3,3) node{\tiny{$E^4_1$}};
 \draw (4,3) node{\tiny{$E^4_2$}};
   \draw (5,3) node{\tiny{$E^4_3$}};
    \draw (6,3) node{\tiny{$E^4_4$}};
    \draw (6,2) node{\tiny{$E^4_5$}};
     \draw (6,1) node{\tiny{$E^4_6$}};
      \draw (6,0) node{\tiny{$E^4_7$}};
       \draw (5,0) node{\tiny{$E^4_8$}};
 \draw (4,0) node{\tiny{$E^4_9$}};
   \draw (3,0) node{\tiny{$E^4_{10}$}};
    \draw (3,1) node{\tiny{$E^4_{11}$}};
    \draw (3,2) node{\tiny{$E^4_{12}$}};

\draw (4.5,-.5) node{Labeled vertices in the graph $\mathbb{Z}\times A_{\infty}^\infty$.};
  \end{tikzpicture}
\end{center}

   Now we will  see how the   conditions of  Theorem \ref{TheoremQh}  are satisfied on $K(\mathbb{Z}A_\infty^\infty,\sigma)$. We only explain this case.
   \bigskip
   
    i)  Indeed, by commutativity of the diagram,  every path between  a  pair of non isomorphic indecomposable objects  $E,E'\in\mathcal{B}_i$  factors throughout an object $E''\in\mathcal{B}_{i-1}$. Therefore, $\mathrm{rad}_{\mathcal{C}}(E,E')=I_{\mathcal{B}_{i-1}}(E,E')$.
  
 ii)  Let  $X$ be an indecomposable object  in  $K(\mathbb{N}A_{\infty}^{\infty},\sigma)$.  If  there is a path from $E_1$ to $X$, then $I_{\mathcal B_1}(-,X)=\mathcal{C}(-,E_1)$,  else $I_{\mathcal B_1}(-,X)$ is the zero functor. This implies that $I_{\mathcal B_1}(-,X)$ is a projective
  $K(\mathbb{N}A_{\infty}^{\infty},\sigma)$-module. For $i> 1$, if $\mathcal{C}(-,X)|_{\mathcal{B}_i}=0$, then   there is no a path from $B$ to $X$ for any $B\in \mathcal{B}_i$, and  $I_{\mathcal{B}_i}(-,X)=0$; if $X\in  \mathcal{B}_i$ then $I_{\mathcal B_1}(-,X)=\mathcal{C}(-,X)$. Assume that
   $\mathcal{C}(-,X)|_{\mathcal{B}_i}\neq 0$ and $X\notin\mathcal{B}_i$, then there is a finite set of indecomposable objects  $\{ E'_1, ..., E'_r\}\subset \mathcal{B}_i$ which are  vertically aligned for which there is a path  from $E'_j$ to $X$ for $j=1,\ldots,r$. Thus,   $I_{\mathcal B_i}(-,X)$ can be covered by     $\coprod_{j=1}^r \mathcal{C}(-,E'_j)$; therefore, there is an epimorphism $  \mathcal{C}(-,E')\rightarrow I_{\mathcal B_i}(-,X)\rightarrow 0$ with
     $E'=\coprod_{j=1}^r E'_j\in\mathcal{B}_i$, which has a kernel that is a finite sum of copies of $\coprod\mathcal{C}(-,E'')$ with $E''\in\mathcal{B}_{i-1}$. Moreover, it can be covered by a projective module $\mathcal{C}(-,E_{i-1})$, $E_{i-1}\in\mathcal{B}_{i-1}$. As a result,  we have   an exact sequence
     
    \[
    \mathcal{C}(-,E''_{i-1})\rightarrow \mathcal{C}(-,E'_i)\rightarrow I_{\mathcal B_i}(-,X)\rightarrow 0.
    \]

  Now, we illustrate the above situation for $ I_{\mathcal{B}_4}(-,X)$ with $X=E_7^5$.

\begin{center}
\begin{tikzpicture}
\filldraw[fill=gray!20,draw=gray!5!](1,3)-- (5,3)--(6,2)--(5.5,1.5)--(6,1)--(5,0)--(1,0)--cycle;
\foreach \x in {1,...,7}  
\draw[->] (\x+.1,.1) to (\x+0.4,0.4);
\foreach \x in {1,...,7}  
\draw[->] (\x +.6,0.6) to (\x+.9,.9);
\foreach \x in {1,...,7}  
\draw[->] (\x+.1,1.1) to (\x+0.4,1.4);
\foreach \x in {1,...,7}  
\draw[->] (\x+.6,1.6) to (\x+.9,1.9);
\foreach \x in {1,...,7}  
\draw[->] (\x+.1,2.1) to (\x+0.4,2.4);
\foreach \x in {1,...,7}  
\draw[->] (\x+.6,2.6) to (\x+.9,2.9);

\foreach \x in {1,...,7}  
\draw[->] (\x+.6,0.4) to (\x+.9,.1);
\foreach \x in {1,...,7}  
\draw[->] (\x+.1,.9) to (\x+0.4,0.6);
\foreach \x in {1,...,7}  
\draw[->] (\x+.6,1.4) to (\x+.9,1.1);
\foreach \x in {1,...,7}  
\draw[->] (\x+.1,1.9) to (\x+0.4,1.6);
\foreach \x in {1,...,7}  
\draw[->] (\x+.6,2.4) to (\x+.9,2.1);
\foreach \x in {1,...,7}  
\draw[->] (\x+.1,2.9) to (\x+.4,2.6);

\foreach \x in {1,...,5}  
\draw (\x,3) node{\tiny{$K$}};
\foreach \x in {6,...,8}  
\draw (\x,3) node{\tiny{$0$}};

 \foreach \x in {1,...,5}  
\draw (\x+.5,2.5) node{\tiny{$K$}};
 \foreach \x in {6,7}  
\draw (\x+.5,2.5) node{\tiny{$0$}};

\foreach \x in {1,...,6}  
\draw (\x,2) node{\tiny{$K$}};
 \foreach \x in {7,8}  
\draw (\x,2) node{\tiny{$0$}};

\foreach \x in {1,...,5}  
\draw (\x+.5,1.5) node{\tiny{$K$}};
 \foreach \x in {6,7}  
\draw (\x+.5,1.5) node{\tiny{$0$}};

\foreach \x in {1,...,6}  
\draw (\x,1) node{\tiny{$K$}};
 \foreach \x in {7,8}  
\draw (\x,1) node{\tiny{$0$}};

\foreach \x in {1,...,5}  
\draw (\x+.5,0.5) node{\tiny{$K$}};
 \foreach \x in {6,7}  
\draw (\x+.5,0.5) node{\tiny{$0$}};

\foreach \x in {1,...,5}  
\draw (\x,0) node{\tiny{$K$}};
 \foreach \x in {6,...,8}  
\draw (\x,0) node{\tiny{$0$}};

\draw (4.5,-.5) node{The functor $I_{\mathcal{B}_4}(-,X)$ with $X=E^5_7$.};

\end{tikzpicture}
\end{center}

\begin{center}
\begin{tikzpicture}
\filldraw[fill=gray!20,draw=gray!5!](1,3)-- (5,3)--(6,2)--(5.5,1.5)--(6,1)--(5,0)--(1,0)--cycle;
\filldraw[fill=gray!45,draw=gray!5!](1,3)-- (4,3)--(5.5,1.5)--(4,0)--(1,0)--cycle;
\foreach \x in {1,...,7}  
\draw[->] (\x+.1,.1) to (\x+0.4,0.4);
\foreach \x in {1,...,7}  
\draw[->] (\x +.6,0.6) to (\x+.9,.9);
\foreach \x in {1,...,7}  
\draw[->] (\x+.1,1.1) to (\x+0.4,1.4);
\foreach \x in {1,...,7}  
\draw[->] (\x+.6,1.6) to (\x+.9,1.9);
\foreach \x in {1,...,7}  
\draw[->] (\x+.1,2.1) to (\x+0.4,2.4);
\foreach \x in {1,...,7}  
\draw[->] (\x+.6,2.6) to (\x+.9,2.9);

\foreach \x in {1,...,7}  
\draw[->] (\x+.6,0.4) to (\x+.9,.1);
\foreach \x in {1,...,7}  
\draw[->] (\x+.1,.9) to (\x+0.4,0.6);
\foreach \x in {1,...,7}  
\draw[->] (\x+.6,1.4) to (\x+.9,1.1);
\foreach \x in {1,...,7}  
\draw[->] (\x+.1,1.9) to (\x+0.4,1.6);
\foreach \x in {1,...,7}  
\draw[->] (\x+.6,2.4) to (\x+.9,2.1);
\foreach \x in {1,...,7}  
\draw[->] (\x+.1,2.9) to (\x+.4,2.6);

\foreach \x in {1,...,4}  
\draw (\x,3) node{\tiny{$K^2$}};
\foreach \x in {6,...,8}  
\draw (\x,3) node{\tiny{$0$}};

 \foreach \x in {1,...,4}  
\draw (\x+.5,2.5) node{\tiny{$K^2$}};
 \foreach \x in {6,7}  
\draw (\x+.5,2.5) node{\tiny{$0$}};

\foreach \x in {1,...,5}  
\draw (\x,2) node{\tiny{$K^2$}};
 \foreach \x in {7,8}  
\draw (\x,2) node{\tiny{$0$}};

\foreach \x in {1,...,5}  
\draw (\x+.5,1.5) node{\tiny{$K^2$}};
 \foreach \x in {6,7}  
\draw (\x+.5,1.5) node{\tiny{$0$}};

\foreach \x in {1,...,5}  
\draw (\x,1) node{\tiny{$K^2$}};
 \foreach \x in {7,8}  
\draw (\x,1) node{\tiny{$0$}};

\foreach \x in {1,...,4}  
\draw (\x+.5,0.5) node{\tiny{$K^2$}};
 \foreach \x in {6,7}  
\draw (\x+.5,0.5) node{\tiny{$0$}};

\foreach \x in {1,...,4}  
\draw (\x,0) node{\tiny{$K^2$}};

\draw (5,3) node{\tiny{$K$}};
 \draw (5,0) node{\tiny{$K$}};
  \draw (5.5,2.5) node{\tiny{$K$}};
   \draw (5.5,0.5) node{\tiny{$K$}};
    \draw (6,2) node{\tiny{$K$}};
     \draw (6,1) node{\tiny{$K$}};

\draw (4.5,-.5) node{$I_{\mathcal{B}_4}(-,X)$ with $X=E^5_7$ can be covered by $\mathcal{C}(-,E^4_5)\coprod \mathcal{C}(-, E^4_6).$};

\end{tikzpicture}
\end{center}

\begin{center}
\begin{tikzpicture}

\filldraw[fill=gray!45,draw=gray!5!](1,3)-- (4,3)--(5.5,1.5)--(4,0)--(1,0)--cycle;
\foreach \x in {1,...,7}  
\draw[->] (\x+.1,.1) to (\x+0.4,0.4);
\foreach \x in {1,...,7}  
\draw[->] (\x +.6,0.6) to (\x+.9,.9);
\foreach \x in {1,...,7}  
\draw[->] (\x+.1,1.1) to (\x+0.4,1.4);
\foreach \x in {1,...,7}  
\draw[->] (\x+.6,1.6) to (\x+.9,1.9);
\foreach \x in {1,...,7}  
\draw[->] (\x+.1,2.1) to (\x+0.4,2.4);
\foreach \x in {1,...,7}  
\draw[->] (\x+.6,2.6) to (\x+.9,2.9);

\foreach \x in {1,...,7}  
\draw[->] (\x+.6,0.4) to (\x+.9,.1);
\foreach \x in {1,...,7}  
\draw[->] (\x+.1,.9) to (\x+0.4,0.6);
\foreach \x in {1,...,7}  
\draw[->] (\x+.6,1.4) to (\x+.9,1.1);
\foreach \x in {1,...,7}  
\draw[->] (\x+.1,1.9) to (\x+0.4,1.6);
\foreach \x in {1,...,7}  
\draw[->] (\x+.6,2.4) to (\x+.9,2.1);
\foreach \x in {1,...,7}  
\draw[->] (\x+.1,2.9) to (\x+.4,2.6);

\foreach \x in {1,...,4}  
\draw (\x,3) node{\tiny{$K$}};
\foreach \x in {5,...,8}  
\draw (\x,3) node{\tiny{$0$}};

 \foreach \x in {1,...,4}  
\draw (\x+.5,2.5) node{\tiny{$K$}};
 \foreach \x in {5,6,7}  
\draw (\x+.5,2.5) node{\tiny{$0$}};

\foreach \x in {1,...,5}  
\draw (\x,2) node{\tiny{$K$}};
 \foreach \x in {6,7,8}  
\draw (\x,2) node{\tiny{$0$}};

\foreach \x in {1,...,5}  
\draw (\x+.5,1.5) node{\tiny{$K$}};
 \foreach \x in {6,7}  
\draw (\x+.5,1.5) node{\tiny{$0$}};

\foreach \x in {1,...,5}  
\draw (\x,1) node{\tiny{$K$}};
 \foreach \x in {6,7,8}  
\draw (\x,1) node{\tiny{$0$}};

\foreach \x in {1,...,4}  
\draw (\x+.5,0.5) node{\tiny{$K$}};
 \foreach \x in {5,6,7}  
\draw (\x+.5,0.5) node{\tiny{$0$}};

\foreach \x in {1,...,4}  
\draw (\x,0) node{\tiny{$K$}};
\foreach \x in {5,6,7,8}  
\draw (\x,0) node{\tiny{$0$}};

\draw (4.5,-.5) node{The kernel of $\mathcal{C}(-,E^4_5\coprod E^4_6)\rightarrow I_{\mathcal{B}_4}(-,X)$  is $\mathcal{C}(-,E^3_4)$.};

\end{tikzpicture}
\end{center}

In this way, we have an exact sequence $\mathcal{C}(-,E^3_4)\rightarrow \mathcal{C}(-,E^4_5\coprod E^4_6)\rightarrow I_{\mathcal{B}_4}(-,X)\rightarrow 0$  with $E_4^3\in\mathcal{B}_3$ and $E^4_5\coprod E^4_6\in\mathcal{B}_4$.
\bigskip

 \textbf{The category $K(\mathbb{Z}A_\infty,\sigma)$}. Let $B_1=\{E^1_j\}_{j\in\mathbb{Z}}$, and $B_i=B_{i-1}\cup\{E^i_j\}_{j\in\mathbb{Z}}$ as they appear in the next picture.

  \begin{center}
\begin{tikzpicture}
\foreach \x in {1,...,7}  
\draw[->] (\x+.1,.1) to (\x+0.4,0.4);
\foreach \x in {1,...,7}  
\draw[->] (\x +.6,0.6) to (\x+.9,.9);
\foreach \x in {1,...,7}  
\draw[->] (\x+.1,1.1) to (\x+0.4,1.4);
\foreach \x in {1,...,7}  
\draw[->] (\x+.6,1.6) to (\x+.9,1.9);
\foreach \x in {1,...,7}  
\draw[->] (\x+.1,2.1) to (\x+0.4,2.4);
\foreach \x in {1,...,7}  
\draw[->] (\x+.6,2.6) to (\x+.9,2.9);

\foreach \x in {1,...,7}  
\draw[->] (\x+.6,0.4) to (\x+.9,.1);
\foreach \x in {1,...,7}  
\draw[->] (\x+.1,.9) to (\x+0.4,0.6);
\foreach \x in {1,...,7}  
\draw[->] (\x+.6,1.4) to (\x+.9,1.1);
\foreach \x in {1,...,7}  
\draw[->] (\x+.1,1.9) to (\x+0.4,1.6);
\foreach \x in {1,...,7}  
\draw[->] (\x+.6,2.4) to (\x+.9,2.1);
\foreach \x in {1,...,7}  
\draw[->] (\x+.1,2.9) to (\x+.4,2.6);

\foreach \x in {-1,...,4}  
\draw (\x+3,3) node{\tiny{$E^1_\x$\ \ }};

\foreach \x in {-2,...,4}  
\draw (\x+3.5,2.5) node{\tiny{$E^2_\x$\ \ }};

\foreach \x in {-2,...,3}  
\draw (\x+4,2) node{\tiny{$E^3_\x$\ \ }};

\foreach \x in {-3,...,3}  
\draw (\x+4.5,1.5) node{\tiny{$E^4_\x$\ \ }};

\foreach \x in {-3,...,2}  
\draw (\x+5,1) node{\tiny{$E^5_\x$\ \ }};

\foreach \x in {-4,...,2}  
\draw (\x+5.5,.5) node{\tiny{$E^6_\x$\ \ }};

\draw (4.5,-.5) node{Labeled vertices in the graph $\mathbb{Z}\times A_{\infty}$};

\end{tikzpicture}
\end{center}
\bigskip

   \textbf{The category $K(\mathbb{Z}D_\infty,\sigma)$}.  Let $B_1=\{E^1_1\}$ and 
   \[
   B_i= \begin{cases}
   B_{i-1}\cup \{E^i_j\}_{1\le j\le 2(i+1)},\  \ \mbox{if\; \ } i \mbox{ is even};\\
B_{i-1}\cup \{E^i_j\}_{1\le j\le 2i-1},\  \ \mbox{if\; \  } i \mbox{ is odd}.  
\end{cases} 
\]

\begin{center}
\begin{tikzpicture}
\foreach \x in {1,...,7}  
\draw[->] (\x+.1,.1) to (\x+0.4,0.4);
\foreach \x in {1,...,7}  
\draw[->] (\x +.6,0.6) to (\x+.9,.9);
\foreach \x in {1,...,7}  
\draw[->] (\x+.1,1.1) to (\x+0.4,1.4);
\foreach \x in {1,...,7}  
\draw[->] (\x+.6,1.6) to (\x+.9,1.9);
\foreach \x in {1,...,7}  
\draw[->] (\x+.1,2.1) to (\x+0.4,2.4);
\foreach \x in {1,...,7}  
\draw[->] (\x+.6,2.6) to (\x+.9,2.9);

\foreach \x in {1,...,7}  
\draw[->] (\x+.6,0.4) to (\x+.9,.1);
\foreach \x in {1,...,7}  
\draw[->] (\x+.1,.9) to (\x+0.4,0.6);
\foreach \x in {1,...,7}  
\draw[->] (\x+.6,1.4) to (\x+.9,1.1);
\foreach \x in {1,...,7}  
\draw[->] (\x+.1,1.9) to (\x+0.4,1.6);
\foreach \x in {1,...,7}  
\draw[->] (\x+.6,2.4) to (\x+.9,2.1);
\foreach \x in {1,...,7}  
\draw[->] (\x+.1,2.9) to (\x+.4,2.6);

 \draw[->] (1.1,2.5) to (1.3,2.5);
 \draw[->] (1.6,2.5) to (1.8,2.5);
   \draw[->] (2.1,2.5) to (2.3,2.5);
    \draw[->] (2.6,2.5) to (2.8,2.5);
 \draw[->] (3.1,2.5) to (3.3,2.5);
   \draw[->] (3.6,2.5) to (3.8,2.5);
   \draw[->] (4.1,2.5) to (4.3,2.5);
 \draw[->] (4.6,2.5) to (4.8,2.5);
   \draw[->] (5.1,2.5) to (5.3,2.5);
           \draw[->] (5.6,2.5) to (5.8,2.5);
               \draw[->] (6.1,2.5) to (6.3,2.5);
                   \draw[->] (6.6,2.5) to (6.8,2.5);
                       \draw[->] (7.1,2.5) to (7.3,2.5);
                           \draw[->] (7.6,2.5) to (7.8,2.5);
                           
  \draw (4.5,2.5) node{\tiny{$E^1_1$}}; 
  
    \draw (4,3) node{\tiny{$E^2_1$}};
      \draw (4,2.5) node{\tiny{$E^2_2$}};
        \draw (4,2) node{\tiny{$E^2_3$}};
          \draw (5,2) node{\tiny{$E^2_4$}};
            \draw (5,2.5) node{\tiny{$E^2_5$}};
              \draw (5,3) node{\tiny{$E^2_6$}};
              
      \draw (3.5,2.5) node{\tiny{$E^3_1$}};
      \draw (3.5,1.5) node{\tiny{$E^3_2$}};
        \draw (4.5,1.5) node{\tiny{$E^3_3$}};
          \draw (5.5,1.5) node{\tiny{$E^3_4$}};
            \draw (5.5,2.5) node{\tiny{$E^3_5$}};

         \draw (3,3) node{\tiny{$E^4_1$}};
          \draw (3,2.5) node{\tiny{$E^4_2$}};
      \draw (3,2) node{\tiny{$E^4_3$}};
        \draw (3,1) node{\tiny{$E^4_4$}};
          \draw (4,1) node{\tiny{$E^4_5$}};
            \draw (5,1) node{\tiny{$E^4_6$}};     
             \draw (6,1) node{\tiny{$E^4_7$}};
             \draw (6,2) node{\tiny{$E^4_8$}};
                \draw (6,2.5) node{\tiny{$E^4_9$}};
              \draw (6,3) node{\tiny{$E^4_{10}$}};

           \draw (2.5,2.5) node{\tiny{$E^5_1$}};
      \draw (2.5,1.5) node{\tiny{$E^5_2$}};
        \draw (2.5,0.5) node{\tiny{$E^5_3$}};
          \draw (3.5,0.5) node{\tiny{$E^5_4$}};
            \draw (4.5,0.5) node{\tiny{$E^5_5$}};    
             \draw (5.5,0.5) node{\tiny{$E^5_6$}};
      \draw (6.5,0.5) node{\tiny{$E^5_7$}};
        \draw (6.5,1.5) node{\tiny{$E^5_8$}};
          \draw (6.5,2.5) node{\tiny{$E^5_9$}};

\draw (4.5,-.5) node{Labeled vertices in the graph $\mathbb{Z}\times D_{\infty}$};

\end{tikzpicture}
\end{center}
\bigskip

 \textbf{The categories $K(\mathbb{N}\tilde{D}_m,\sigma)$, $K(\mathbb{N}\tilde{A}_m,\sigma)$, $K(\mathbb{N}\tilde{E}_6,\sigma)$, $K(\mathbb{N}\tilde{E}_7,\sigma)$ and $K(\mathbb{N}\tilde{E}_8,\sigma)$}. 
  
 We start with  $K(\mathbb{N}\tilde{D}_m,\sigma)$. First we label the vertices of $\tilde{D}_n$  as
 $(\tilde{D}_n)_0=\{1,2,\ldots, n-1,n-1,n,n+1\}$. Let $A=\{x\in (\tilde{D}_n)_0: x \textrm{  is a source vertex}\}$ and
   $B=\{x\in (\tilde{D}_n)_0: x \textrm{  is a sink vertex}\}$. Now we  label the indecomposable objects of  $K(\mathbb{N}\tilde{D}_m,\sigma)$; first we  label the vertices of $\{1\}\times \tilde{D}_m=\{E^1_j\}_{j\in A}\cup \{E^2_j\}_{j\in B}$. For $i\ge 3$, let $E_{j}^{i}=\tau^{-1}(E^{i-2}_j)$ and define  $B_1=\{E^1_j\}_{j\in A}$, $B_2=B_1\cup \{E^2_j\}_{j\in B}$. For $i\ge 3$,
  
  \[
B_i= \begin{cases} B_{i-1}\cup \{E^i_j\}_{j\in A},\ \ \mbox{if $i$ is odd}; \\
B_{i-1}\cup \{E^i_j\}_{j\in B},\ \ \mbox{if $i$ is even} .
 \end{cases} 
\]

 \begin{center}
\begin{tikzpicture}
\foreach \x in {1,...,7}  
\draw[->] (\x+.1,.1) to (\x+0.4,0.4);
\foreach \x in {1,...,7}  
\draw[->] (\x +.6,0.6) to (\x+.9,.9);
\foreach \x in {1,...,7}  
\draw[->] (\x+.1,1.1) to (\x+0.4,1.4);
\foreach \x in {1,...,7}  
\draw[->] (\x+.6,1.6) to (\x+.9,1.9);
\foreach \x in {1,...,7}  
\draw[->] (\x+.1,2.1) to (\x+0.4,2.4);
\foreach \x in {1,...,7}  
\draw[->] (\x+.6,2.6) to (\x+.9,2.9);

\foreach \x in {1,...,7}  
\draw[->] (\x+.6,0.4) to (\x+.9,.1);
\foreach \x in {1,...,7}  
\draw[->] (\x+.1,.9) to (\x+0.4,0.6);
\foreach \x in {1,...,7}  
\draw[->] (\x+.6,1.4) to (\x+.9,1.1);
\foreach \x in {1,...,7}  
\draw[->] (\x+.1,1.9) to (\x+0.4,1.6);
\foreach \x in {1,...,7}  
\draw[->] (\x+.6,2.4) to (\x+.9,2.1);
\foreach \x in {1,...,7}  
\draw[->] (\x+.1,2.9) to (\x+.4,2.6);

\foreach \x in {1,...,7}  
\draw (\x+.5,1.5) node{\tiny{$\cdots$}};

 \draw[->] (1.1,2.5) to (1.3,2.5);
 \draw[->] (1.6,2.5) to (1.8,2.5);
   \draw[->] (2.1,2.5) to (2.3,2.5);
    \draw[->] (2.6,2.5) to (2.8,2.5);
 \draw[->] (3.1,2.5) to (3.3,2.5);
   \draw[->] (3.6,2.5) to (3.8,2.5);
   \draw[->] (4.1,2.5) to (4.3,2.5);
 \draw[->] (4.6,2.5) to (4.8,2.5);
   \draw[->] (5.1,2.5) to (5.3,2.5);
           \draw[->] (5.6,2.5) to (5.8,2.5);
               \draw[->] (6.1,2.5) to (6.3,2.5);
                   \draw[->] (6.6,2.5) to (6.8,2.5);
                       \draw[->] (7.1,2.5) to (7.3,2.5);
                           \draw[->] (7.6,2.5) to (7.8,2.5);

  \draw[->] (1.1,0.5) to (1.3,0.5);
 \draw[->] (1.6,0.5) to (1.8,0.5);
   \draw[->] (2.1,0.5) to (2.3,0.5);
    \draw[->] (2.6,0.5) to (2.8,0.5);
 \draw[->] (3.1,0.5) to (3.3,0.5);
   \draw[->] (3.6,0.5) to (3.8,0.5);
   \draw[->] (4.1,0.5) to (4.3,0.5);
 \draw[->] (4.6,0.5) to (4.8,0.5);
   \draw[->] (5.1,0.5) to (5.3,0.5);
           \draw[->] (5.6,0.5) to (5.8,0.5);
               \draw[->] (6.1,0.5) to (6.3,0.5);
                   \draw[->] (6.6,0.5) to (6.8,0.5);
                       \draw[->] (7.1,0.5) to (7.3,0.5);
                           \draw[->] (7.6,0.5) to (7.8,0.5);

              \draw (1,3) node{\tiny{$E^1_1$}};
              \draw (1,2.5) node{\tiny{$E^1_2$}};                                     
              \draw (1,2) node{\tiny{$E^1_4$}};
              \draw (1,1) node{\tiny{\ \ $E^1_{n- 2}$}};
              \draw (1,0.5) node{\tiny{$E^1_n$}};
              \draw (1,0) node{\tiny{$E^1_{n+1}$}};
              
               \draw (1.5,2.5) node{\tiny{$E^2_3$}};
              \draw (1.5,0.5) node{\tiny{$\ \ \ E^2_{n-1}$}};

                \draw (2,3) node{\tiny{$E^3_1$}};
              \draw (2,2.5) node{\tiny{$E^3_2$}};                                     
              \draw (2,2) node{\tiny{$E^3_4$}};
              \draw (2,1) node{\tiny{\ \ $E^3_{n- 2}$}};
              \draw (2,0.5) node{\tiny{$\ \ E^3_n$}};
              \draw (2,0) node{\tiny{$\ \  E^3_{n+1}$}};
              
               \draw (2.5,2.5) node{\tiny{$E^4_3$}};
              \draw (2.5,0.5) node{\tiny{$\ \ \ \  E^4_{n-1}$}};

              \draw (4.5,-.5) node{Labeled vertices in the graph $\mathbb{N}\times \tilde{D}_m$};

                           \end{tikzpicture}
\end{center}

 Analogously, we   can give a  filtration for the rest of the categories: $K(\mathbb{N}\tilde{A}_m,\sigma)$, $K(\mathbb{N}\tilde{E}_6,\sigma)$, $K(\mathbb{N}\tilde{E}_7,\sigma)$ and $K(\mathbb{N}\tilde{E}_8,\sigma)$ into subcategories. For example, for  $K(\mathbb{N}\tilde{E}_6,\sigma)$,  let $(\tilde{E}_6)_0=\{1,2,...,7\}$, and let $A=\{1,3,5,7\}$, $B=\{2,4,6\}$. Define   $B_1=\{E^1_1, E^1_3, E^1_5, E^1_7\}$, $B_2=B_1\cup \{E^2_2, E^2_4, E^2_6\},$
  $B_i= B_{i-1}\cup \{E^i_1,  E^i_3, E^i_5, E^i_7\}$, if $i$ is odd, and $B_i= B_{i-1}\cup \{E^i_2, E^i_4, E^i_6\}$, if $i$ is even.


\begin{center}
\begin{tikzpicture}
\foreach \x in {0.5,2,3.5,5}  
\draw[->] (\x+.1,.1) to (\x+0.4,0.4);

 \foreach \x in {1,2.5,4}  
\draw[->] (\x+.3,.3) to (\x+0.9,0.1);

 \foreach \x in {0.5,2,3.5,5}  
\draw[->] (\x+.2,1.3) to (\x+0.4,0.4);

\foreach \x in {1,2.5,4}  
\draw[->] (\x+0.2,.7) to (\x+0.9,1.4);

\foreach \x in {0.5,2,3.5,5}  
\draw[->,thick] (\x+.1,1.6) to (\x+0.4,1.9);

 \foreach \x in {1,2.5,4}  
\draw[->,thick] (\x+.2,1.8) to (\x+0.9,1.6);

\foreach \x in {0.5,2,3.5,5}  
\draw[->,thin, dashed] (\x+.1,1.6) to (\x+0.4,2.4);

\foreach \x in {0,1.5,3,4.5}  
\draw[->,thick] (\x+.1,2.4) to (\x+0.9,2.1);

\foreach \x in {1,2.5,4}  
\draw[->,thick] (\x+.2,2) to (\x+0.4,2.4);

\foreach \x in {0.5,2,3.5,5}  
\draw[->] (\x+.1,2.9) to (\x+0.4,2.6);

\foreach \x in {1,2.5,4}  
\draw[->] (\x+0.1,2.6) to (\x+0.9,2.9);

\foreach \x in {1,2.5,4}  
\draw[->, thin, dashed] (\x+.1,2.4) to (\x+0.9,1.4);

              \draw (0.5,3) node{\tiny{$ E^1_1$}};   
                             \draw (0,2.5) node{\tiny{$ E^1_7\ \ $}};   
                                            \draw (0.5,1.5) node{\tiny{$ E^1_3\ \  $}};   
                                                           \draw (0.5,0) node{\tiny{$ E^1_5$}};

              \draw (1,2.5) node{\tiny{$ E^2_2$}};                                                              
                            \draw (1,2) node{\tiny{$ E^2_6$}};   
                                          \draw (1,0.5) node{\tiny{$ \ \ E^2_4$}};

              \draw (2,3) node{\tiny{$ E^3_1$}};   
                             \draw (1.5,2.5) node{\tiny{$ E^3_7\ \ $}};   
                                            \draw (2,1.5) node{\tiny{$\   E^3_3\   $}};   
                                                           \draw (2,0) node{\tiny{$ E^3_5$}};

              \draw (2.5,2.5) node{\tiny{$ E^2_2$}};                                                              
                            \draw (2.5,2) node{\tiny{$ E^2_6$}};   
                                          \draw (2.5,0.5) node{\tiny{$ \ \ E^2_4$}};

          \draw (6,0.5) node{\tiny{$\cdots$}};                                                                                                                                                                                                                      
           \draw (6,1.5) node{\tiny{$\cdots$}};                                                                                                                                                                                                                      
            \draw (6,2) node{\tiny{$\cdots$}};                                                                                                                                                                                                                      
             \draw (6,3) node{\tiny{$\cdots$}};                                                                                                                                                                                                                      
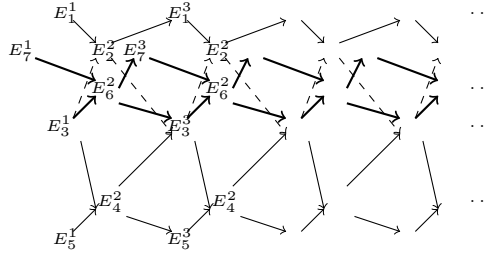
                     \draw (4.5,-.5) node{Labeled vertices in the graph $\mathbb{N}\times \tilde{E}_6$};         
 \end{tikzpicture}
  \end{center}

   \subsection{A tilting subcategory in $\mathcal{C}=K(\mathbb{Z}A_\infty,\sigma)$}

According with to filtration given above, although  it is not finite,  we could try to find tilting category $\mathcal{T}\subset\mathrm{Mod}(\mathcal C)$  in the category $\mathcal{F}(\Delta)\cap\mathcal{F}(\nabla)$ (see \cite{Rin2}). In this part, we found a tilting category in $\mathrm{Mod}(\mathcal{C})$
for this special case. First, we label the vertices of $\mathbb{Z}A_{\infty}$ as follows.

\begin{center}
\begin{tikzpicture}
\foreach \x in {1,...,7}  
\draw[->] (\x+.1,.1) to (\x+0.4,0.4);
\foreach \x in {1,...,7}  
\draw[->] (\x +.6,0.6) to (\x+.9,.9);
\foreach \x in {1,...,7}  
\draw[->] (\x+.1,1.1) to (\x+0.4,1.4);
\foreach \x in {1,...,7}  
\draw[->] (\x+.6,1.6) to (\x+.9,1.9);
\foreach \x in {1,...,7}  
\draw[->] (\x+.1,2.1) to (\x+0.4,2.4);
\foreach \x in {1,...,7}  
\draw[->] (\x+.6,2.6) to (\x+.9,2.9);

\foreach \x in {1,...,7}  
\draw[->] (\x+.6,0.4) to (\x+.9,.1);
\foreach \x in {1,...,7}  
\draw[->] (\x+.1,.9) to (\x+0.4,0.6);
\foreach \x in {1,...,7}  
\draw[->] (\x+.6,1.4) to (\x+.9,1.1);
\foreach \x in {1,...,7}  
\draw[->] (\x+.1,1.9) to (\x+0.4,1.6);
\foreach \x in {1,...,7}  
\draw[->] (\x+.6,2.4) to (\x+.9,2.1);
\foreach \x in {1,...,7}  
\draw[->] (\x+.1,2.9) to (\x+.4,2.6);

\foreach \x in {-1,...,4}  
\draw (\x+3,3) node{\tiny{(1,\x)}};

\foreach \x in {-2,...,4}  
\draw (\x+3.5,2.5) node{\tiny{(2,\x)}};

\foreach \x in {-2,...,3}  
\draw (\x+4,2) node{\tiny{(3,\x)}};

\foreach \x in {-3,...,3}  
\draw (\x+4.5,1.5) node{\tiny{(4,\x)}};

\foreach \x in {-3,...,2}  
\draw (\x+5,1) node{\tiny{(5,\x)}};

\foreach \x in {-4,...,2}  
\draw (\x+5.5,.5) node{\tiny{(6,\x)}};

\draw (4.5,-.5) node{Labeled vertices in the graph $\mathbb{Z}\times A_{\infty}$};

\end{tikzpicture}
\end{center}
  
  In this way, we can indentify the representations assigning a $K$-vector  space $V_{ij}$ to each vertice $(i,j)\in\mathbb{N}\times\mathbb{Z}$.
 
Let $r\in\mathbb{Z}$ and consider the  representation $T(r,1)=\{V_{ij}\}$ defined as 
\[
V_{ij}= \begin{cases} K,\ \ \mbox{if } i\ge r \mbox{ and } -(i-r-1)\le j\le 1; \\
0 \ \ \mbox{ in other case}, \end{cases} 
\]
and a map between two adjacent $K$-vectorial spaces $V_{rs}$ and $V_{uv}$ is $1_K$ if $V_{rs}=V_{uv}=K$ and $0$ in other case.
\bigskip

In the same way,  for any $s\in\mathbb{Z}$ we can define the \emph{moved}  representations of  $T(r,1)$ as follows.  We define $T(r,s)=\{V_{ij}\}$  as
\[
V_{ij}= \begin{cases} K,\  \ \mbox{if } i\ge r \mbox{ and } -(i-r-s)\le j\le s; \\
0\ \ \mbox{ in other case}.  \end{cases} 
\]

\begin{lemma}
For all pairs $(r,s)\in\mathbb{N}\times\mathbb{Z}$ the representation $T(r,s)$ lies in $\mathcal{F}(\Delta)$. 
\end{lemma}
\begin{proof}
Indeed for simplicity we only  consider the representation  $T(1,1)$.  First, we observe that 
$T(1,1)^{(i)}=\mathcal{C}(-,E^i_1)$.  On the other hand, let $X$ be an idecomposable object. By the mesh relations then each path in $I_{\mathcal{B}_{i-1}}(X,E^i_1)$ can be written in a unique way as $X\rightarrow  E^{i-1}_1\xrightarrow{\alpha} E^i_1$ where $\alpha$ is the unique arrow from 
$E^{i-1}_1$ to $ E^i_1$;  therefore $I_{\mathcal{B}_{i-1}}(-,E^i_1)\cong \mathcal{C}(-,E^{i-1}_1)$. It follows that we have a chain
\[
0=T(1,1)^{(0)}\subset T(1,1)^{(1)}\subset T(1,1)^{(2)}\subset \cdots,
\]
such that $T(1,1)^{(i)}/T(1,1)^{(i-1)}\cong \mathcal{C}(-,E^i_1)/\mathcal{C}(-,E^{i-1}_1)\cong \mathcal{C}(-,E^i_1)/I_{\mathcal{B}_{i-1}}(-,E^i_1)\in\Delta(i)$.
Therefore $T(1,1)\in\mathcal{F}(\Delta)$.
\end{proof}

\begin{lemma}\label{tilting1}
\begin{itemize}
\item[(i)] $T(1,s)$ is a projective $\mathcal C$-module for all $s\in\mathbb{Z}$.
\item[(ii)] For all $(r,s)\in\mathbb{N}\times\mathbb{Z}$, there is an exact sequence
\begin{eqnarray}\label{triangle1}
0\rightarrow \mathcal{C}(-, E^{r-1}_s)\xrightarrow{j} T(1,s)\xrightarrow{p} T(r,s)\rightarrow 0.
\end{eqnarray}
\item[(iii)] Let $(r,s),(r',s')\in\mathbb{N}\times\mathbb{Z}$. Then $\mathrm{Hom}(\mathcal{C}(-,E^r_s),T(r',s'))=0$ if  $s\neq s'$ or 
$s=s'$ and $r<r'$.
\item[(iv)]  For all  $(r,s)\in\mathbb{N}\times\mathbb{Z}$, the  projective $\mathcal C$-module $\mathcal{C}(-,E^r_s)$ has a resolution
\begin{eqnarray*}
0\rightarrow \mathcal{C}(-, E^{r}_s)\rightarrow T(1,s)\rightarrow T(r+1,s)\rightarrow 0.
\end{eqnarray*}
\end{itemize}

\end{lemma}
\begin{proof}
To prove (i),  consider the directed system of projective $\mathcal{C}$-modules: $\mathcal{C}(-,E^1_s)\subset \mathcal{C}(-,E^2_s)\subset\cdots$. Clearly, it follows that
$T(1,s)=  \displaystyle\lim_{ \rightarrow } \mathcal{C}(-,E^i_s) $. Hence, $T(1,s)$ is a flat $\mathcal{C}$-module because it is a limite of projective
$\mathcal{C}$-modules. Let be $0\rightarrow A\rightarrow B\rightarrow C \rightarrow  0$ an exact sequence of $\mathcal{C}$-modules. As a resultado, we have an exact sequence of $\mathcal{C}$-modules:
\[
0\rightarrow T\otimes_K A\rightarrow T\otimes_K B\rightarrow T\otimes_K C\rightarrow  0
\]
 
 If we denote $D=\mathrm{Hom}(-,K):L.F(\mathcal{C})\rightarrow L.F(\mathcal{C}^{op})$, $D(M)(X)=\mathrm{Hom}(M(X),K)$, the usual duality between the subcategories of locally finite $\mathcal{C}$-modules.
 \[
 \begin{diagram}
 \node{0}\arrow{e,t}{}
  \node{\mathcal{C}(C,D(T))}\arrow{e,t}{}\arrow{s,l}{\cong}
   \node{\mathcal{C}(B,D(T))}\arrow{e,t}{}\arrow{s,l}{\cong}
   \node{\mathcal{C}(A,D(T))}\arrow{e,t,!}{}\arrow{s,l}{\cong}\\
    \node{0}\arrow{e,t}{}
  \node{D(T\otimes_K C)}\arrow{e,t}{}
   \node{D(T\otimes_K B)}\arrow{e,t}{}
   \node{D(T\otimes_K A)}\arrow{e,t,}{}
   \node{0}
 \end{diagram}
\]
it follows that the exact sequence on the top is a short exact sequence. Thus,  $D(T)$ is injective and therefore $T$ is projective.
(ii)  is clear, and it follows from  the definition of $T(r,s)$.
(iii) follows from the mesh relations in $\mathcal{C}$, finally  (iv)  follows straightforward from the definition.
\end{proof}

\begin{theorem}
The full subcategory $\mathcal{T}$ of $\mathrm{Mod}( \mathcal{C})$  consisting of the family of  $\mathcal{C}$-modules $\{T(r,s) \}_{(r,s)\in\mathbb{N\times Z}}$ is a tilting category.
\end{theorem}
\begin{proof}
\textbf{i)} It is clear that $\mathrm{pdim}T(r,s)\le 1$ by (i) and (ii) in Lemma \ref{tilting1}.
\textbf{ii)} After applying $\mathrm{Hom}(-,T(r',s'))$  to the exact sequence (\ref{triangle1}), we obtain the following exact sequence:
\begin{eqnarray*}
 0\rightarrow\mathrm{Hom}(T(r,s),T(r',s'))\xrightarrow{p*} \mathrm{Hom}(T(1,s), T(r',s'))\xrightarrow{j*}\\
 \mathrm{Hom}( \mathcal{C}(-, E^{r-1}_s),T(r',s'))\xrightarrow{\partial}\mathrm{Ext}^1(T(r,s), T(r',s'))\rightarrow  0.
\end{eqnarray*}
First,  assume that $s\neq s'$ or  $s=s'$ and $r-1<r'$, then by Lemma \ref{tilting1}, we have  $\mathrm{Hom}( \mathcal{C}(-, E^{r-1}_s),T(r',s'))=0$ and 
$\mathrm{Ext}^1(T(r,s), T(r',s'))=0$. Assume that $s=s'$ and $r-1\ge r'$, and let $h\in  \mathrm{Hom}( \mathcal{C}(-, E^{r-1}_s),T(r',s'))$. If we write $h=\{h_{ij}:V_{ij}\rightarrow V_{ij}'\} $ as a map between representations, by commutativity, we can assume that each  $h_{ij}=\lambda 1_K$ for some $\lambda\in K$ or zero. in this way, we have the following commutative diagram:
\[
\begin{diagram}
\node{\mathcal{C}(-,E^{r-1}_s)}\arrow{e,t}{j}\arrow{s,l}{h}
 \node{T(1,s)} \arrow{sw,b}{\lambda p}\\
  \node{T(r',s)}
\end{diagram}
\]
 
 Therefore, $j*(\lambda p)=\lambda p j=h$. Hence,  $j*$ is an epimorphism, and $\mathrm{Ext}^1(T(r,s), T(r',s'))=0$.
 
 \textbf{iii)} This condition follows directly from  part (iv) of Lemma \ref{tilting1}.
\end{proof}

Finally, we illustrate the above theorem with an example 

\begin{center}
\begin{tikzpicture}

\filldraw[fill=gray!20,draw=gray!5!](1,0)-- (4,3)--(7,0)--cycle;
\foreach \x in {1,...,7}  
\draw[->] (\x+.1,.1) to (\x+0.4,0.4);
\foreach \x in {1,...,7}  
\draw[->] (\x +.6,0.6) to (\x+.9,.9);
\foreach \x in {1,...,7}  
\draw[->] (\x+.1,1.1) to (\x+0.4,1.4);
\foreach \x in {1,...,7}  
\draw[->] (\x+.6,1.6) to (\x+.9,1.9);
\foreach \x in {1,...,7}  
\draw[->] (\x+.1,2.1) to (\x+0.4,2.4);
\foreach \x in {1,...,7}  
\draw[->] (\x+.6,2.6) to (\x+.9,2.9);

\foreach \x in {1,...,7}  
\draw[->] (\x+.6,0.4) to (\x+.9,.1);
\foreach \x in {1,...,7}  
\draw[->] (\x+.1,.9) to (\x+0.4,0.6);
\foreach \x in {1,...,7}  
\draw[->] (\x+.6,1.4) to (\x+.9,1.1);
\foreach \x in {1,...,7}  
\draw[->] (\x+.1,1.9) to (\x+0.4,1.6);
\foreach \x in {1,...,7}  
\draw[->] (\x+.6,2.4) to (\x+.9,2.1);
\foreach \x in {1,...,7}  
\draw[->] (\x+.1,2.9) to (\x+.4,2.6);

\foreach \x in {1,...,7}  
\draw (\x,0) node{\tiny {K}};
\foreach \x in {1,...,6}  
\draw (\x+.5,.5) node{\tiny {K}};
\foreach \x in {1,...,5}  
\draw (\x+1,1) node{\tiny{K}};
\foreach \x in {1,...,4}  
\draw (\x+1.5,1.5) node{\tiny{K}};
\foreach \x in {1,...,3}  
\draw (\x+2,2) node{\tiny{K}};
\foreach \x in {1,...,2}  
\draw (\x+2.5,2.5) node{\tiny{K}};

\foreach \x in {5,...,8}  
\draw (\x,3) node{\tiny {0}};

\foreach \x in {5,...,7}  
\draw (\x+.5,2.5) node{\tiny {0}};

\draw (4,3) node{\tiny{K}};
\draw (8,0) node{\tiny{0}};
\draw (7.5,.5) node{\tiny{0}};

\draw (7,1) node{\tiny{0}};
\draw (8,1) node{\tiny{0}};

\draw (6.5,1.5) node{\tiny{0}};
\draw (7.5,1.5) node{\tiny{0}};

\draw (6,2) node{\tiny{0}};
\draw (7,2) node{\tiny{0}};
\draw (8,2) node{\tiny{0}};

\foreach \x in {1,...,3} 
\draw (\x,3) node{\tiny{0}};

\foreach \x in {1,...,2} 
\draw (\x+.5,2.5) node{\tiny{0}};

\foreach \x in {1,...,2} 
\draw (\x,2) node{\tiny{0}};

\draw (4,-.5) node{$T(1,1)$};

\end{tikzpicture}
\end{center}

\begin{center}
\begin{tikzpicture}
\filldraw[fill=gray!20,draw=gray!5!](3,0)-- (5,2)--(7,0)--cycle;
\foreach \x in {1,...,7}  
\draw[->] (\x+.1,.1) to (\x+0.4,0.4);
\foreach \x in {1,...,7}  
\draw[->] (\x +.6,0.6) to (\x+.9,.9);
\foreach \x in {1,...,7}  
\draw[->] (\x+.1,1.1) to (\x+0.4,1.4);
\foreach \x in {1,...,7}  
\draw[->] (\x+.6,1.6) to (\x+.9,1.9);
\foreach \x in {1,...,7}  
\draw[->] (\x+.1,2.1) to (\x+0.4,2.4);
\foreach \x in {1,...,7}  
\draw[->] (\x+.6,2.6) to (\x+.9,2.9);

\foreach \x in {1,...,7}  
\draw[->] (\x+.6,0.4) to (\x+.9,.1);
\foreach \x in {1,...,7}  
\draw[->] (\x+.1,.9) to (\x+0.4,0.6);
\foreach \x in {1,...,7}  
\draw[->] (\x+.6,1.4) to (\x+.9,1.1);
\foreach \x in {1,...,7}  
\draw[->] (\x+.1,1.9) to (\x+0.4,1.6);
\foreach \x in {1,...,7}  
\draw[->] (\x+.6,2.4) to (\x+.9,2.1);
\foreach \x in {1,...,7}  
\draw[->] (\x+.1,2.9) to (\x+.4,2.6);

\foreach \x in {1,...,8} 
\draw (\x,3) node{\tiny{0}};

\foreach \x in {1,...,7} 
\draw (\x+.5,2.5) node{\tiny{0}};

\foreach \x in {1,...,4} 
\draw (\x,2) node{\tiny{0}};

\foreach \x in {1,...,3} 
\draw (\x+.5,1.5) node{\tiny{0}};

\foreach \x in {1,...,3} 
\draw (\x,1) node{\tiny{0}};

\foreach \x in {1,...,2} 
\draw (\x+.5,.5) node{\tiny{0}};

\foreach \x in {1,...,2} 
\draw (\x,0) node{\tiny{0}};

\foreach \x in {3,...,7} 
\draw (\x,0) node{\tiny{K}};

\foreach \x in {3,...,6} 
\draw (\x+.5,0.5) node{\tiny{K}};

\foreach \x in {4,...,6} 
\draw (\x,1) node{\tiny{K}};

\foreach \x in {4,...,5} 
\draw (\x+.5,1.5) node{\tiny{K}};

\draw(5,2) node{\tiny{K}};

\foreach \x in {6,...,8} 
\draw (\x,2) node{\tiny{0}};

\foreach \x in {6,...,7} 
\draw (\x+.5,1.5) node{\tiny{0}};

\foreach \x in {7,...,8} 
\draw (\x,1) node{\tiny{0}};

\draw (7.5,.5) node{\tiny{0}};
\draw (8,0) node{\tiny{0}};

\draw (4,-.5) node{$T(3,1)$};

\end{tikzpicture}
\end{center}

\begin{center}
\begin{tikzpicture}
\filldraw[fill=gray!20,draw=gray!5!](1,0)-- (4,3)--(4.5,2.5)--(2,0)--cycle;
\foreach \x in {1,...,7}  
\draw[->] (\x+.1,.1) to (\x+0.4,0.4);
\foreach \x in {1,...,7}  
\draw[->] (\x +.6,0.6) to (\x+.9,.9);
\foreach \x in {1,...,7}  
\draw[->] (\x+.1,1.1) to (\x+0.4,1.4);
\foreach \x in {1,...,7}  
\draw[->] (\x+.6,1.6) to (\x+.9,1.9);
\foreach \x in {1,...,7}  
\draw[->] (\x+.1,2.1) to (\x+0.4,2.4);
\foreach \x in {1,...,7}  
\draw[->] (\x+.6,2.6) to (\x+.9,2.9);

\foreach \x in {1,...,7}  
\draw[->] (\x+.6,0.4) to (\x+.9,.1);
\foreach \x in {1,...,7}  
\draw[->] (\x+.1,.9) to (\x+0.4,0.6);
\foreach \x in {1,...,7}  
\draw[->] (\x+.6,1.4) to (\x+.9,1.1);
\foreach \x in {1,...,7}  
\draw[->] (\x+.1,1.9) to (\x+0.4,1.6);
\foreach \x in {1,...,7}  
\draw[->] (\x+.6,2.4) to (\x+.9,2.1);
\foreach \x in {1,...,7}  
\draw[->] (\x+.1,2.9) to (\x+.4,2.6);

\foreach \x in {0,.5,1,1.5,2} 
\draw (\x+1,\x+1) node{\tiny{0}};

\foreach \x in {0,.5,1} 
\draw (\x+1,\x+2) node{\tiny{0}};

\foreach \x in {0,.5,1,1.5,2,2.5,3} 
\draw (\x+1,\x) node{\tiny{K}};

\foreach \x in {0,.5,1,1.5,2,2.5} 
\draw (\x+2,\x) node{\tiny{K}};

\foreach \x in {0,.5,1,1.5,2,2.5,3} 
\draw (\x+3,\x) node{\tiny{0}};

\foreach \x in {0,.5,1,1.5,2,2.5,3} 
\draw (\x+4,\x) node{\tiny{0}};

\foreach \x in {0,.5,1,1.5,2,2.5,3} 
\draw (\x+4,\x) node{\tiny{0}};

\foreach \x in {0,.5,1,1.5,2,2.5,3} 
\draw (\x+5,\x) node{\tiny{0}};

\foreach \x in {0,.5,1,1.5,2} 
\draw (\x+6,\x) node{\tiny{0}};

\foreach \x in {0,.5,1} 
\draw (\x+7,\x) node{\tiny{0}};

\draw (8,0) node{\tiny{0}};
\draw (5,3) node{\tiny{0}};

\draw (1,3) node{\tiny{0}};

\draw (4,-.5) node{$\mathcal{C}(-,E^2_1)$};

\end{tikzpicture}
\end{center}

The above diagrams show how to obtain the exact sequence:
\begin{eqnarray*}
0\rightarrow \mathcal{C}(-,E^2_1)\rightarrow T(1,1)\rightarrow T(3,1)\rightarrow 0
\end{eqnarray*}

\bigskip
Of course many other filtrations for these categories can be given. For example, for the category $K(\mathbb{Z}A_\infty,\sigma)$, we can take $B_1=\{E^1_1\}$ and $B_i=B_{i-1}\cup\{E^i_j\}_{j\in\mathbb{N}}$ as shown in the next picture.

  \begin{center}
\begin{tikzpicture}
\foreach \x in {1,...,7}  
\draw[->] (\x+.1,.1) to (\x+0.4,0.4);
\foreach \x in {1,...,7}  
\draw[->] (\x +.6,0.6) to (\x+.9,.9);
\foreach \x in {1,...,7}  
\draw[->] (\x+.1,1.1) to (\x+0.4,1.4);
\foreach \x in {1,...,7}  
\draw[->] (\x+.6,1.6) to (\x+.9,1.9);
\foreach \x in {1,...,7}  
\draw[->] (\x+.1,2.1) to (\x+0.4,2.4);
\foreach \x in {1,...,7}  
\draw[->] (\x+.6,2.6) to (\x+.9,2.9);

\foreach \x in {1,...,7}  
\draw[->] (\x+.6,0.4) to (\x+.9,.1);
\foreach \x in {1,...,7}  
\draw[->] (\x+.1,.9) to (\x+0.4,0.6);
\foreach \x in {1,...,7}  
\draw[->] (\x+.6,1.4) to (\x+.9,1.1);
\foreach \x in {1,...,7}  
\draw[->] (\x+.1,1.9) to (\x+0.4,1.6);
\foreach \x in {1,...,7}  
\draw[->] (\x+.6,2.4) to (\x+.9,2.1);
\foreach \x in {1,...,7}  
\draw[->] (\x+.1,2.9) to (\x+.4,2.6);

\draw (5,3) node{\tiny{$E_1^1$}};
\draw (5,2) node{\tiny{$E_1^2$}};
\draw (5,1) node{\tiny{$E_2^2$}};
\draw (5,0) node{\tiny{$E_3^2$}};

\draw (4.5,2.5) node{\tiny{$E_1^3$}};
\draw (4.5,1.5) node{\tiny{$E_3^3$}};
\draw (4.5,0.5) node{\tiny{$E_5^3$}};

\draw (5.5,2.5) node{\tiny{$E_2^3$}};
\draw (5.5,1.5) node{\tiny{$E_4^3$}};
\draw (5.5,0.5) node{\tiny{$E_6^3$}};

\draw (4,3) node{\tiny{$E_1^4$}};
\draw (4,2) node{\tiny{$E_3^4$}};
\draw (4,1) node{\tiny{$E_5^4$}};
\draw (4,0) node{\tiny{$E_7^4$}};

\draw (6,3) node{\tiny{$E_2^4$}};
\draw (6,2) node{\tiny{$E_4^4$}};
\draw (6,1) node{\tiny{$E_6^4$}};
\draw (6,0) node{\tiny{$E_8^4$}};


\draw (3.5,2.5) node{\tiny{$E_1^5$}};
\draw (3.5,1.5) node{\tiny{$E_3^5$}};
\draw (3.5,0.5) node{\tiny{$E_5^5$}};

\draw (6.5,2.5) node{\tiny{$E_2^5$}};
\draw (6.5,1.5) node{\tiny{$E_4^5$}};
\draw (6.5,0.5) node{\tiny{$E_6^5$}};

\draw (4.5,-.5) node{Labeled vertices in the graph $\mathbb{Z}\times A_{\infty}$};

\end{tikzpicture}
\end{center}
\bigskip

Another filtration for this category  is given by 
$B_1=\{E_1^1\}$ and
  \[
B_i= \begin{cases} B_{i-1}\cup \{E^i_j\}_{1\le j\le 2i },\ \ \mbox{if $i$ is even}; \\
B_{i-1}\cup \{E^i_j\}_{1\le j\le 2i-1},\ \ \mbox{if $i$ is odd} .
 \end{cases} 
\]
  \begin{center}
\begin{tikzpicture}
\foreach \x in {1,...,7}  
\draw[->] (\x+.1,.1) to (\x+0.4,0.4);
\foreach \x in {1,...,7}  
\draw[->] (\x +.6,0.6) to (\x+.9,.9);
\foreach \x in {1,...,7}  
\draw[->] (\x+.1,1.1) to (\x+0.4,1.4);
\foreach \x in {1,...,7}  
\draw[->] (\x+.6,1.6) to (\x+.9,1.9);
\foreach \x in {1,...,7}  
\draw[->] (\x+.1,2.1) to (\x+0.4,2.4);
\foreach \x in {1,...,7}  
\draw[->] (\x+.6,2.6) to (\x+.9,2.9);

\foreach \x in {1,...,7}  
\draw[->] (\x+.6,0.4) to (\x+.9,.1);
\foreach \x in {1,...,7}  
\draw[->] (\x+.1,.9) to (\x+0.4,0.6);
\foreach \x in {1,...,7}  
\draw[->] (\x+.6,1.4) to (\x+.9,1.1);
\foreach \x in {1,...,7}  
\draw[->] (\x+.1,1.9) to (\x+0.4,1.6);
\foreach \x in {1,...,7}  
\draw[->] (\x+.6,2.4) to (\x+.9,2.1);
\foreach \x in {1,...,7}  
\draw[->] (\x+.1,2.9) to (\x+.4,2.6);

\draw (4.5,2.5) node{\tiny{$E_1^1$}};

\draw (4,3) node{\tiny{$E_1^2$}};
\draw (4,2) node{\tiny{$E_2^2$}};
\draw (5,2) node{\tiny{$E_3^2$}};
\draw (5,3) node{\tiny{$E_4^2$}};
\draw (3.5,2.5) node{\tiny{$E_1^3$}};
\draw (3.5,1.5) node{\tiny{$E_2^3$}};
\draw (4.5,1.5) node{\tiny{$E_3^3$}};
\draw (5.5,1.5) node{\tiny{$E_4^3$}};
\draw (5.5,2.5) node{\tiny{$E_5^3$}};
\draw (3,3) node{\tiny{$E_1^4$}};
\draw (3,2) node{\tiny{$E_2^4$}};
\draw (3,1) node{\tiny{$E_3^4$}};
\draw (4,1) node{\tiny{$E_4^4$}};
\draw (5,1) node{\tiny{$E_5^4$}};
\draw (6,1) node{\tiny{$E_6^4$}};
\draw (6,2) node{\tiny{$E_7^4$}};
\draw (6,3) node{\tiny{$E_8^4$}};

\draw (4.5,-.5) node{Labeled vertices in the graph $\mathbb{Z}\times A_{\infty}$};

\end{tikzpicture}
\end{center}

Observe that filtrations of the same style  for the above examples could be given. Therefore, a category $\mathcal{C}$  can be considered cuasi-hereditary with respect to different filtrations.

\subsection{Tensor product of quasi-hereditary categories}

Let $K$ be a field. Assume that  $\mathcal{C}_1$ and $\mathcal{C}_2$ are  quasi-hereditary $K$-categories. In this section, we show that the tensor product $\mathcal{C}_1\otimes_K \mathcal{C}_{2}$ is quasi-hereditary.  

First we  remember some facts about the  tensor product $\mathcal{C}_1\otimes\mathcal{C}_2$ of  two categories (see \cite{MB}).  This is the category whose
class of objects is  $|\mathcal{C}_1|\times |\mathcal{C}_2|$, where the abelian group of morphisms
from $(X_1,X_2)$ to $(Y_1,Y_2)$   is the ordinary tensor product of $K$-vectorial spaces $\mathcal{C}(X_1,Y_1)\otimes_K \mathcal{C}(X_1,Y_1)$

Consider additive functors $F:\mathcal{C}_1\rightarrow \mathrm{Mod}(K)$,  
$G:\mathcal{C}_2\rightarrow \mathrm{Mod}(K)$. They induce the additive bifunctor
 \[
 F\otimes G:\mathcal{C}_1\otimes \mathcal{C}_2\rightarrow \mathrm{Mod}(K)
 \]
wich is defined by $F\otimes G((X,Y))=F(X)\otimes G(Y)$. Assume  that 
$I(-,?)$ and $J(-,?)$ are two-sided ideals
in $\mathcal{C}_1$ and $\mathcal{C}_2$ respectively.  We can then  see that $J(-,?)\otimes \mathcal{C}_2(-,?)$ is a two-sided ideal in $\mathcal{C}_1\otimes\mathcal{C}_2$. 

\begin{theorem}
Let $\mathcal{C}_1$ and $\mathcal{C}_2$  be quasi-hereditary $K$-categories. Then the tensor product $\mathcal{C}_1\otimes \mathcal{C}_2$
is quasi-hereditary.
\end{theorem}
\begin{proof}
Let
\begin{eqnarray*}
 \mathcal{C}_1 (-,?)\supset I_1(-,?)\supset I_2(-,?)\supset I_3(-,?)\supset\cdots \\
 \mathcal{C}_2 (-,?)\supset J_1(-,?)\supset J_2(-,?)\supset J_3(-,?)\supset\cdots
 \end{eqnarray*}
 the corresponding heredity chains of ideals.
 
 We can then can form a chain of ideals:
 \begin{eqnarray*}
 \mathcal{C}_1 (-,?)\otimes\mathcal{C}_2(-,?)\supset \mathcal{C}_1 (-,?)\otimes J_1+I_1\otimes\mathcal{C}_2 (-,?)\supset
 \mathcal{C}_1 (-,?)\otimes J_2+I_1\otimes\mathcal{C}_2 (-,?)\supset\\
 \mathcal{C}_1 (-,?)\otimes J_2+I_2\otimes\mathcal{C}_2 (-,?)\supset\mathcal{C}_1 (-,?)\otimes J_3+I_2\otimes\mathcal{C}_2 (-,?)\supset\cdots.\ \ \ \ \ \ \ \ \ \ \ \  \ \ \
 \end{eqnarray*}
 We show that the above is a heredity chain.
 
  Consider the inclusion
 \[
 \frac{\mathcal{C}_1 (-,?)\otimes\mathcal{C}_2(-,?)}{\mathcal{C}_1 (-,?)\otimes J_2+I_1\otimes\mathcal{C}_2 (-,?)}\supset\frac{\mathcal{C}_1 (-,?)\otimes J_1+I_1\otimes\mathcal{C}_2 (-,?)}{\mathcal{C}_1 (-,?)\otimes J_2+I_1\otimes\mathcal{C}_2 (-,?)}.
 \]
 
 Observe that 
 \begin{eqnarray*}
 \frac{\mathcal{C}_1 (-,?)\otimes\mathcal{C}_2(-,?)}{\mathcal{C}_1 (-,?)\otimes J_2+I_1\otimes\mathcal{C}_2 (-,?)}&\cong&
 \frac{\frac{\mathcal{C}_1 (-,?)\otimes \mathcal{C}_2 (-,?)+I_1\otimes \mathcal{C}_2 (-,?)}{I_1\otimes \mathcal{C}_2 (-,?)}}{\frac{\mathcal{C}_1 (-,?)\otimes J_2+I_1\otimes \mathcal{C}_2 (-,?)}{I_1\otimes \mathcal{C}_2 (-,?)}}\\
 &\cong&\frac{\frac{\mathcal{C}_1 (-,?)\otimes \mathcal{C}_2 (-,?)}{I_1\otimes \mathcal{C}_2 (-,?)}}{\frac{\mathcal{C}_1 (-,?)\otimes J_2 }{I_1\otimes J_2 }}\\
 &\cong&\frac{\frac{\mathcal{C}_1(-,?)}{I_1}\otimes \mathcal{C}_2(-,?)} {\frac{\mathcal{C}_1(-,?)}{I_1}\otimes J_2}\\
 &\cong&\frac{\mathcal{C}_1(-,?)}{I_1}\otimes \frac{\mathcal{C}_2(-,?)}{J_2}.
 \end{eqnarray*}
 
 On the other hand,
 \begin{eqnarray*}
 \frac{\mathcal{C}_1 (-,?)\otimes J_1+I_1\otimes\mathcal{C}_2 (-,?)}{\mathcal{C}_1 (-,?)\otimes J_2+I_1\otimes\mathcal{C}_2 (-,?)}&\cong&
 \frac{\frac{\mathcal{C}_1 (-,?)\otimes J_1+I_1\otimes\mathcal{C}_2 (-,?)}{I_1\otimes J_1}}{\frac{\mathcal{C}_1 (-,?)\otimes J_2+I_1\otimes\mathcal{C}_2 (-,?)}{I_1\otimes J_1}}\\
 &\cong&\frac{\frac{\mathcal{C}_1 (-,?)\otimes J_1}{I_1\otimes J_1}}{\frac{\mathcal{C}_1 (-,?)\otimes J_2}{I_1\otimes J_2}}\\
 &\cong&\frac{\frac{\mathcal{C}_1 (-,?)}{I_1}\otimes J_1}{\frac{\mathcal{C}_1 (-,?)}{I_1}\otimes J_2}\\
 &\cong&\frac{\mathcal{C}_1 (-,?)}{I_1}\otimes \frac{J_1}{J_2}
 \end{eqnarray*}
 
 In this manner proving that   $\frac{\mathcal{C}_1 (-,?)}{I_1}\otimes \frac{J_1}{J_2}$ is a heredity ideal of $\frac{\mathcal{C}_1(-,?)}{I_1}\otimes \frac{\mathcal{C}_2(-,?)}{J_2}$ is sufficient.
 
 (i)  It is clear that $[\mathcal{C}_1(-,?)]^2=\mathcal{C}_1(-,?)$. This implies that 
 \[
 \left[\frac{\mathcal{C}_1(-,?)}{I_1}\right]^2=\frac{\mathcal{C}_1(-,?)^2+I_1}{I_1}=\frac{\mathcal{C}_1(-,?)}{I_1}
 \]
 Thus, we have 
 \[
 \left[\frac{\mathcal{C}_1 (-,?)}{I_1}\otimes \frac{J_1}{J_2}\right]^2=\left[\frac{\mathcal{C}_1 (-,?)}{I_1}\right]^2\otimes \left[  \frac{J_1}{J_2}\right]^2=\frac{\mathcal{C}_1 (-,?)}{I_1}\otimes \frac{J_1}{J_2}
 \]
 It follows that $\frac{\mathcal{C}_1 (-,?)}{I_1}\otimes \frac{J_1}{J_2}$ is idempotent.
 
 (ii)  Observe that
 \[
 \mathrm{rad}\left[ \frac{\mathcal{C}_1(-,?)}{I_1}\otimes \frac{\mathcal{C}_2(-,?)}{J_2} \right]=
 \mathrm{rad}\left[ \frac{\mathcal{C}_1(-,?)}{I_1} \right]\otimes  \frac{\mathcal{C}_2(-,?)}{J_2} +  \frac{\mathcal{C}_1(-,?)}{I_1} \otimes   \mathrm{rad}\left[ \frac{\mathcal{C}_2(-,?)}{J_2} \right].
 \]
 
 Since $\mathcal{C}_1$ and $\mathcal{C}_2$ are quasi-hereditary, we have 
 \[
 \frac{\mathcal{C}_1(-,?)}{I_1} \mathrm{rad}\left[ \frac{\mathcal{C}_1(-,?)}{I_1}\right]\frac{\mathcal{C}_1(-,?)}{I_1}=
  \frac{J_1}{J_2}\mathrm{rad}\left[ \frac{\mathcal{C}_2(-,?)}{J_2} \right]\frac{J_1}{J_2}=0.
 \]
 It follows that
 \begin{eqnarray*}
\left( \frac{\mathcal{C}_1 (-,?)}{I_1}\otimes \frac{J_1}{J_2} \right) \mathrm{rad}\left[ \frac{\mathcal{C}_1(-,?)}{I_1}\otimes \frac{\mathcal{C}_2(-,?)}{J_2} \right]
 \left( \frac{\mathcal{C}_1 (-,?)}{I_1}\otimes \frac{J_1}{J_2}\right)\cong\\
 \frac{\mathcal{C}_1(-,?)}{I_1} \mathrm{rad}\left[ \frac{\mathcal{C}_1(-,?)}{I_1}\right]\frac{\mathcal{C}_1(-,?)}{I_1}\otimes \frac{J_1}{J_2} \frac{\mathcal{C}_2(-,?)}{J_2} \frac{J_1}{J_2}
+ \left(\frac{\mathcal{C}_1(-,?)}{I_1}\right)^3\otimes \frac{J_1}{J_2}\mathrm{rad}\left[ \frac{\mathcal{C}_2(-,?)}{J_2} \right]\frac{J_1}{J_2}
 \end{eqnarray*}
 This is,
 \[
 \left( \frac{\mathcal{C}_1 (-,?)}{I_1}\otimes \frac{J_1}{J_2} \right) \mathrm{rad}\left[ \frac{\mathcal{C}_1(-,?)}{I_1}\otimes \frac{\mathcal{C}_2(-,?)}{J_2} \right]
 \left( \frac{\mathcal{C}_1 (-,?)}{I_1}\otimes \frac{J_1}{J_2}\right)\cong 0.
 \]
 
 (iii) Observe that $\frac{J_1(-,?)}{J_2(-,?)}$ is a projective $\frac{\mathcal{C}_2}{J_2}$-module. Thus, $\frac{\mathcal{C}_1 (-,?)}{I_1}\otimes \frac{J_1(-,?)}{J_2(-,?)}$
  is a projective $\frac{\mathcal{C}_1}{J_1}\otimes\frac{\mathcal{C}_2}{J_2}$-module.
\end{proof}


\section*{Acknowledgements}
The author would like to thank  Prof. Roberto Mart\'inez-Villa  for introducing him to the study of the quasi-hereditary algebras and for his valuable comments and suggestions for the preparation of this work. The author would also like to thank PROMEP/103.5/13/6535 for the financial support  received for the development of this project.

\end{document}